\documentclass[11pt]{article}

\usepackage{amsfonts,amsthm}
\usepackage{amsmath}  
\usepackage{latexsym}
\usepackage{amssymb}
\usepackage{mathrsfs}
\usepackage{tikz}
\usepackage[all]{xy}
\usepackage{hyperref}

\newtheorem{thm}{Theorem}[section]
\newtheorem{lem}[thm]{Lemma}
\newtheorem{prop}[thm]{Proposition}

\newtheorem{coro}[thm]{Corollary}
\newtheorem{conj}[thm]{Conjecture}
\newtheorem{rem}[thm]{Remark}
\newtheorem{exa}[thm]{Example}

\newtheorem*{term*}{Notation/Terminology}
\newtheorem{defi}[thm]{Definition}

\newcommand{\oH}{\overline{H}}
\newcommand{\bT}{\boldsymbol{t}}
\newcommand{\bbT}{\boldsymbol{T}}
\newcommand{\STab}{\mathrm{STab}}
\newcommand{\SSTab}{\mathrm{SSTab}}

\newcommand{\bk}{\bold{k}}
\newcommand{\cD}{\mathcal{D}}
\newcommand{\bC}{\mathbb{C}}
\newcommand{\CC}{\mathbb{C}}
\newcommand{\si}{\sigma}
\newcommand{\mS}{\mathfrak{S}}

\newcommand{\cc}{\mathrm{cc}}
\newcommand{\qc}{\mathrm{c}}

\newcommand{\AS}{\mathrm{AS}}
\newcommand{\Dom}{\mathrm{Dom}}

\newcommand{\diag}[3]{ \foreach \t in {1,...,#3} {\draw[thick] (#1+\t,#2-1) rectangle (#1+\t-1,#2);} }
\newcommand{\diagg}[4]{ \foreach \t in {1,...,#3} {\draw[thick] (#1+\t,#2-1) rectangle (#1+\t-1,#2);} \foreach \t in {1,...,#4} {\draw[thick] (#1+\t,#2-1) rectangle (#1+\t-1,#2-2);} }
\newcommand{\diaggg}[5]{ \foreach \t in {1,...,#3} {\draw[thick] (#1+\t,#2-1) rectangle (#1+\t-1,#2);} \foreach \t in {1,...,#4} {\draw[thick] (#1+\t,#2-1) rectangle (#1+\t-1,#2-2);}
                         \foreach \t in {1,...,#5} {\draw[thick] (#1+\t,#2-2) rectangle (#1+\t-1,#2-3);} }
\newcommand{\diagggg}[6]{ \foreach \t in {1,...,#3} {\draw[thick] (#1+\t,#2-1) rectangle (#1+\t-1,#2);} \foreach \t in {1,...,#4} {\draw[thick] (#1+\t,#2-1) rectangle (#1+\t-1,#2-2);}
                         \foreach \t in {1,...,#5} {\draw[thick] (#1+\t,#2-2) rectangle (#1+\t-1,#2-3);} \foreach \t in {1,...,#6} {\draw[thick] (#1+\t,#2-3) rectangle (#1+\t-1,#2-4);} }    
\newcommand{\diaggggg}[6]{ \foreach \t in {1,...,#3} {\draw[thick] (#1+\t,#2-1) rectangle (#1+\t-1,#2);} \foreach \t in {1,...,#4} {\draw[thick] (#1+\t,#2-1) rectangle (#1+\t-1,#2-2);}
                         \foreach \t in {1,...,#5} {\draw[thick] (#1+\t,#2-2) rectangle (#1+\t-1,#2-3);} \foreach \t in {1,...,#6} {\draw[thick] (#1+\t,#2-3) rectangle (#1+\t-1,#2-4);}
                         \foreach \t in {1,...,#6} {\draw[thick] (#1+\t,#2-4) rectangle (#1+\t-1,#2-5);} }                             

\begin{document}

\title{Fused braids and centralisers of tensor representations of $U_q(gl_N)$}

\author{N. Cramp\'e\footnote{Institut Denis-Poisson CNRS/UMR 7013 - Universit\'e de Tours - Universit\'e d'Orl\'eans, 
Parc de Grandmont, 37200 Tours, France. crampe1977@gmail.com}$\ $ and 
L. Poulain d'Andecy\footnote{Laboratoire de math\'ematiques de Reims UMR 9008, Universit\'e de Reims Champagne-Ardenne,
Moulin de la Housse BP 1039, 51100 Reims, France.
loic.poulain-dandecy@univ-reims.fr}}

\date{}
\maketitle

\begin{abstract}
We present in this paper the algebra of fused permutations and its deformation the fused Hecke algebra. The first one is defined on a set of combinatorial objects that we call fused permutations, and its deformation is defined on a set of topological objects that we call fused braids. We use these algebras to prove a Schur--Weyl duality theorem for any tensor products of any symmetrised powers of the natural representation of $U_q(gl_N)$. Then we proceed to the study of the fused Hecke algebras and in particular, we describe explicitely the irreducible representations and the branching rules. Finally, we aim to an algebraic description of the centralisers of the tensor products of $U_q(gl_N)$-representations under consideration. We exhibit a simple explicit element that we conjecture to generate the kernel from the fused Hecke algebra to the centraliser. We prove this conjecture in some cases and in particular, we obtain a description of the centraliser of any tensor products of any finite-dimensional representations of $U_q(sl_2)$.
\end{abstract}

\section{Introduction}

The Schur--Weyl duality relates the representation theory of the group $GL_N(\CC)$ to the representation theory of the symmetric groups $\mS_n$. In fact, for any $N>1$ and $n\geq 1$, if $V$ denotes the natural (vector) representation of dimension $N$ of $GL_N(\CC)$, the Schur--Weyl duality asserts that the centraliser of the action of $GL_N(\CC)$ on the tensor product $V^{\otimes n}$ is the image of the action by permutation of the symmetric group $\mS_n$.

Moreover, the Schur--Weyl duality can be extended to the standard deformations of the structures under consideration. Namely, on one hand, one replaces $GL_N(\CC)$ by the quantum group $U_q(gl_N)$ and on the other hand, one replace the symmetric group $\mS_n$ by the Hecke algebra $H_n(q)$. 

The fact that the centraliser of the action of $U_q(gl_N)$ ($GL_N(\CC)$ if $q^2=1$) is obtained as the image of the action of the Hecke algebra $H_n(q)$ (the symmetric group $\mS_n$ if $q^2=1$) is the first part of the statement, sometimes called in invariant theory the first fundamental theorem. In order to describe more precisely the centraliser, one needs the second part, sometimes called the second fundamental theorem, which identifies the kernel of the action of the Hecke algebra (starting from now, we include the case $q^2=1$ in the general case, and we indicate that in this paper $q$ is either an indeterminate or a non-zero complex number which is not a root of unity). This kernel is well-understood and can be described alternatively in terms of the representation theory of $H_n(q)$, or with a direct algebraic description (with an explicit generator, the $q$-antisymmetriser) in $H_n(q)$. One famous example is for $N=2$ where one obtains the Temperley--Lieb algebra which can be seen as a quotient of the Hecke algebra.

We note that the first part of the Schur--Weyl duality involves an algebra, here $H_n(q)$, which does not depend on the dimension $N$, while of course in the second part, the description of the kernel depends on $N$. We see the Hecke algebra $H_n(q)$ as a sort of ``universal'' centraliser allowing to obtain, for all $N$, the actual centralisers and we emphasise its role. We note that for any $N$, the centraliser coincides with $H_n(q)$ for small $n$ ($n\leq N$) but always starts to differ at $n=N+1$.

\vskip .2cm
In general, centralisers of tensor powers of representations of quantum groups are interesting objects for several reasons. First and most naturally, in representation theory, they allow to decompose these tensor powers into direct sums of irreducible representations. For example, the decomposition of $V^{\otimes n}$ can be obtained directly from the study of the representation theory of the Hecke algebra.

Second, the centralisers contain naturally elements satisfying the braid relations. In other words, they contain elements generating a finite-dimensional quotients of the algebra of the braid group (we emphasise that in general the centralisers are not generated by these elements, see \cite{LZ}). These images of the braid group in centralisers come from the quasi-triangular structure of quantum groups since the images of the $R$-matrix in the representation satisfy the braid relations. These has applications in particular for finding knots and links invariants. For example, the Temperley--Lieb algebra is used to construct the Jones polynomial of a link.

In fact, an image of the braid group is already found at the level of the Hecke algebra $H_n(q)$. Indeed, the Hecke algebra is well-known to be a quotient of the algebra of the braid group. This step (maybe at first sight a small step) of considering the Hecke algebra instead of the centralisers leads to quite interesting development for the knots and links invariants. In fact, by working directly at the level of the Hecke algebra, one obtains the HOMFLY-PT polynomial of a link, which is a two-variable generalisation of the Jones polynomial. This polynomial can be seen as interpolating the family of invariants coming from the centralisers of $V^{\otimes n}$ as the dimension $N$ varies. We also point out that, even though the Hecke algebra $H_n(q)$ is larger than the centraliser, its algebraic structure is somewhat simpler and the calculation of the HOMFLY-PT polynomial from $H_n(q)$ is a relatively simple algebraic procedure, see \emph{e.g.} \cite[section 4.5]{GP}. All in all this points out again to the importance and the useful role played by the Hecke algebra.

Finally, one can be interested in the centralisers in the context of the Yang--Baxter equation in mathematical physics. Here again our main point is the following: a very simple formula (called a Baxterisation formula, see for example \cite{IO,Jo}) builds abstract solutions of the Yang--Baxter equation inside the Hecke algebra. Then the genuine (matrix) solutions associated to the representation $V$ of $U_q(gl_N)$ (for any $N$) can be obtained by representing this simple abstract solution vie the action of $H_n(q)$ on $V^{\otimes n}$. This is our final pointer to the usefulness of the Hecke algebra as a ``universal'' centraliser.

\vskip .2cm
In this paper, we consider the following tensor product of representations of $U_q(gl_N)$:
\begin{equation}\label{intro-rep}
L_{(k_1)}^N\otimes L_{(k_2)}^N\otimes \dots\otimes L_{(k_n)}^N\ ,
\end{equation}
where $L_{(k)}^N$ is the $k$-th symmetric power of the vector representation $V=L_{(1)}^N$ of $U_q(gl_N)$. Here $k_1,k_2,\dots,k_n$ are arbitrary non-negative integers. From the quasitriangularity property of $U_q(gl_N)$, up to isomorphism, the centralisers do not depend on the order of $k_1,\dots,k_n$, so we could assume for example that $k_1\geq\dots\geq k_n$. However, we will only make such an assumption in this paper when necessary.

Our main goal is to introduce and study an algebra playing for these tensor representations the same role as the Hecke algebra for the representations $V^{\otimes n}$. So we will denote by $\bk=(k_1,k_2,\dots)$ an arbitrary sequence of non-negative integers and we will denote this algebra by $H_{\bk,n}(q)$ (of course, for a given $n\geq0$, the algebra $H_{\bk,n}(q)$ depends on $\bk$ only through the $n$ first entries $k_1,\dots,k_n$).

Our first main result which serves also as a motivation is that there is an action of $H_{\bk,n}(q)$ on the space (\ref{intro-rep}) such that its image coincides with the centraliser of the action of the quantum group $U_q(gl_N)$. This is the generalisation of the first part of the Schur--Weyl duality theorem. Note that, as before, the algebra $H_{\bk,n}(q)$ is ``universal'' in the sense that it does not depend on $N$, only its image in the space of endomorphisms does. For a given $N$, we denote it $\oH^N_{\bk,n}(q)$ and this is the centraliser of the $U_q(gl_N)$-action.

In fact, for fixed $\bk$, the algebras $H_{\bk,n}(q)$ form a chain of algebras as $n$ varies:
\begin{equation*}
\CC=H_{\bk,0}(q)\subset H_{\bk,1}(q)\subset H_{\bk,2}(q)\subset\dots\dots \subset H_{\bk,n}(q)\subset H_{\bk,n+1}(q)\subset \dots\dots\ ,
\end{equation*}
and the chain of centralisers, that we denote,
\begin{equation*}
\CC=\oH^N_{\bk,0}(q)\subset \oH^N_{\bk,1}(q)\subset \oH^N_{\bk,2}(q)\subset\dots\dots \subset \oH^N_{\bk,n}(q)\subset \oH^N_{\bk,n+1}(q)\subset \dots\dots\ ,
\end{equation*}
is obtained by taking a quotient at each level of the first chain of algebras. Here as in the usual Schur--Weyl duality, the two chains always coincide for $n$ small enough and always start to differ at $n=N+1$. In other words, we have that $\oH^N_{\bk,n}(q)=H_{\bk,n}(q)$ if and only if $n\leq N$ (and this for any $\bk$). 

\vskip .2cm
Regarding the definition of the algebras $H_{\bk,n}(q)$, we proceed in two steps, starting with the situation $q^2=1$ (that we denote $H_{\bk,n}(1)$) and then deforming this construction to obtain $H_{\bk,n}(q)$. The algebra $H_{\bk,n}(1)$ is constructed on a set of combinatorial objects, which we call ``fused permutation'', which can be conveniently described by diagrams. They generalise the usual permutations of the symmetric group $\mS_n$. As combinatorial objects, these fused permutations can also be described as certain sequences of multisets, or as $n$ by $n$ matrices with non-negative integer entries subject to the condition that the rows and columns sum to some of the integers from the sequence $\bk$. The multiplication of these objects is best described with the diagrams, and is a purely combinatorial procedure.

Then for the deformation $H_{\bk,n}(q)$, we add to these combinatorial objects a topological information (roughly, we now have strands and we allow crossings). The resulting objects we propose to call them ``fused braids''. The ingredients to multiply these fused braids are the usual Hecke relation together with an idempotent from the Hecke algebra, which is here the $q$-symmetriser. As a tentative name for the resulting algebra $H_{\bk,n}(q)$, we propose ``fused Hecke algebra'' (a more precise name would be something along the lines of $\bk$-symmetrised fused Hecke algebra). As in many cases of diagram algebras, the deformation $H_{\bk,n}(q)$ has the same dimension as $H_{\bk,n}(1)$ and admits a basis indexed by some standard diagrams from $H_{\bk,n}(1)$.

Algebras similar to the fused Hecke algebras have been studied in some particular cases, however mostly corresponding to the centralisers of $U_q(gl_2)$ (so in our notation related to $\oH^2_{\bk,n}(q)$ for some particular $\bk$). They were constructed from the Temperley--Lieb algebras. We refer to \cite{ALZ, LSA} for studies of cellular structure and non-semisimple representation theory and to \cite{Mat, Zi} for some studies related to physical models.

\vskip .2cm
Then we proceed to the study of the algebras $H_{\bk,n}(q)$, which we believe is now motivated. In particular, one objective we set is to generalise also the second part of the Schur--Weyl duality theorem, namely the description of the kernel of the quotient map from $H_{\bk,n}(q)$ to $\oH^N_{\bk,n}(q)$ for any $N$.

The path we follow towards this objective starts with the description of the representation theory of $H_{\bk,n}(q)$, which is also of independent interest. The representation theory of the algebra $H_{\bk,n}(q)$ could in principle be obtained by seeing it as a centraliser (for $N$ large enough), and then by using the well-known Littlewood--Richardson rule restricted to the tensor products under consideration (\ref{intro-rep}) (in this case, this is also called the Pieri rule). However, we find it more natural in this paper and also more convenient for our later purposes to provide an independent treatment, which relies only on the representation theory of the Hecke algebras. So we recover the Pieri rule by analysing the seminormal basis of skew-shape representations of usual Hecke algebras. In fact, we obtain along the way a more precise information than the combinatorial rule, since we identify explicitly the representation spaces of $H_{\bk,n}(q)$ inside the representation spaces of usual Hecke algebras, providing thus an explicit construction of the irreducible representations of $H_{\bk,n}(q)$. To summarise, we completely describe the representation theory of the semisimple algebras $H_{\bk,n}(q)$, together with the branching rules between $H_{\bk,n-1}(q)$ and $H_{\bk,n}(q)$.

Interestingly, once the representation theory of $H_{\bk,n}(q)$ is described, it is possible to identify a subalgebra (which is more precisely an ideal) of $H_{\bk,n}(q)$ isomorphic to $H_{\bk-1,n}(q)$ (where $\bk-1$ is the sequence obtained from $\bk$ by decreasing every entries by 1). This makes connections inside the whole family of algebras $H_{\bk,n}(q)$ between members with fixed $n$ and different $\bk$. Moreover, these ideals turn out to be precisely the ideals allowing to obtain for different $N$ the centralisers $\oH^N_{\bk,n}(q)$ from $H_{\bk,n}(q)$. 
So at this point, we have a complete description of the ideals from the point of view of representation theory.

Finally, we proceed to an algebraic description of these ideals, aiming to a concrete algebraic presentation of the centralisers. We exhibit an explicit and rather simple element of the fused Hecke algebra $H_{\bk,N+1}(q)$ and we conjecture that it is a generator of the ideal we are looking to. We are able to prove this conjecture in the following cases:

$\bullet$ Any sequence $\bk$ when $N=2$, so in particular we obtain a description of the centraliser of any tensor products of any finite-dimensional representations of $U_q(sl_2)$. 

$\bullet$ Any $N$ for representations (\ref{intro-rep}) having an arbitrary symmetrised power in the first factor and then only the natural representation.

$\bullet$ Any $N$ for representations (\ref{intro-rep}) involving only the natural representation and the symmetrised square.

We indicate that recently the centralisers for $N=2$ and a constant sequence $\bk=(k,k,k,\dots)$ have been studied from a different point of view in \cite{ALZ} (see also \cite{LZ2} for $k=2$), and again from a different perspective in \cite{CPV} for $N=2$ and a tensor product of three spaces.

\vskip .2cm
The representation theory (including the branching rules) of the fused Hecke algebras, and of the centralisers, are conveniently described by their so-called Bratteli diagrams. Moreover, the way centralisers appear as quotients of the fused Hecke algebras is also best described, from the point of view of representation theory, by the notion of quotients of Bratteli diagrams. We collect and organise in Appendix these notions and the terminology we will use throughout the paper, and we also provide examples.

\paragraph{Some perspectives.} In the situation of a constant sequence $\bk=(k,k,k,\dots)$, the fused Hecke algebras $H_{\bk,n}(q)$ contains naturally elements satisfying the braid group relations. These elements does not generate the whole algebra $H_{\bk,n}(q)$ in general. So the subalgebras generated by these elements, and their images in the centralisers $\oH^N_{\bk,n}(q)$ for various $N$, deserve a better study. We emphasise that these algebras are finite-dimensional quotients of the braid group algebra in which the braid generators satisfy a characteristic equation of order $k+1$.

Regarding the study of the Yang--Baxter equation, the fused Hecke algebras admit a Baxterisation formula, generalising the one in the usual Hecke algebra. The explicit formula will appear in a future work.

Finally one can also consider other tensor products than (\ref{intro-rep}) and/or other quantum groups than $U_q(gl_N)$. The starting point of the approach developed in this paper clearly generalises as follows. One can still consider fused braids, but use a different procedure for multiplying them. Indeed one can replace the Hecke algebras by other quotients of the braid group algebras and/or one can replace the $q$-symmetrisers by other idempotents. For example, one can keep the Hecke algebra and replace the $q$-symmetrisers by the $q$-antisymmetrisers (for alternating powers of representations). Also, one could replace the Hecke algebra by the BMW algebra (for other classical quantum groups) and consider analogues of the $q$-symmetriser and $q$-antisymmetriser.

\paragraph{Acknowledgements.} Both authors are partially supported by Agence National de la Recherche Projet AHA
ANR-18-CE40-0001. N.Cramp\'e warmly thanks the university of Reims for hospitality
during his visit in the course of this investigation.

\setcounter{tocdepth}{1}
\tableofcontents

\section{The algebra of fused permutations}\label{sec-def-fus-perm}

Let $\bk=(k_1,k_2,...)\in \mathbb{Z}_{\geq0}^{\infty}$ be an infinite sequence of non-negative integers, and let $n\in\mathbb{Z}_{>0}$\,.

\subsection{Definition of the algebra $H_{\bk,n}(1)$ of fused permutations}

\paragraph{Objects.} We consider diagrams as follows. We place two (horizontal) rows of $n$ ellipses (drawn as small black-filled ellipses), one on top of another and we connect top ellipses with bottom ellipses with edges. We require the following: for each $a\in\{1,\dots,n\}$, there are $k_a$ edges which start from the $a$-th top ellipse and there are $k_a$ edges which arrive at the $a$-th bottom ellipse. The total number of edges is then $k_1+\dots+k_n$.

Now take a diagram as above and let $a\in\{1,\dots,n\}$. There are $k_a$ edges starting from the $a$-th top ellipse, so denote $I_a$ the multiset $\{i_1,\dots,i_{k_a}\}$ indicating the bottom ellipses reached by these edges. This is indeed a multiset, meaning that repetitions are allowed since several of these edges can reach the same bottom ellipse.

Finally, we consider two diagrams as above equivalent if their sequences of multisets $(I_1,\dots,I_n)$ coincide. Diagrammatically, this simply means that the only information that matters is of the form: which ellipse is connected to which ones and by how many edges.

\begin{defi}
A fused permutation is an equivalence class of diagrams as explained above. We denote $\cD_{\bk,n}$ the set of fused permutations associated to $\bk$ and $n$.
\end{defi}
We will simply say fused permutation instead of a more precise terminology for elements of $\cD_{\bk,n}$ (such that for example $(\bk,n)$-fused permutations).  From now on, we will almost always identify a fused permutation with a diagram representing it.

\begin{rem}
\begin{itemize}
\item By definition, the set of fused permutations $\cD_{\bk,n}$ is in bijection with the set of sequences $(I_1,\dots,I_n)$ of multisets consisting of elements of $\{1,\dots,n\}$ such that: for each $a\in\{1,\dots,n\}$, the number of entries in $I_a$ is equal to $k_a$ and the number of occurrences of $a$ in the multisets $I_1,\dots,I_n$ is also equal to $k_a$.
\item The set $\cD_{\bk,n}$ is also in bijection with the set of $n$ by $n$ matrices with entries in $\mathbb{Z}_{\geq0}$ such that: the sum of the entries in the $a$-th row is equal to $k_a$ and the sum of the entries in the $a$-th column is also equal to $k_a$. The bijection is such that the entry in position $(a,b)$ of a matrix indicates how many times the $a$-th top ellipse is connected to the $b$-th bottom ellipse.
\end{itemize}
\end{rem}

\paragraph{Examples.} $\bullet$ If $\bk=(1,1,1,\dots)$ consists only of 1's then the set of fused permutations $\cD_{\bk,n}$ coincides with the set of permutations of $\{1,\dots,n\}$.

$\bullet$ Let $n=3$ and  take $k_1=2$ and $k_2=k_3=1$. There are 7 distinct fused permutations in $\cD_{\bk,3}$ and here is an example (we give a diagram and the corresponding sequence of multisets):

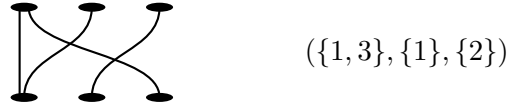
\begin{center}
\begin{tikzpicture}[scale=0.3]

\fill (21,2) ellipse (0.6cm and 0.2cm);\fill (21,-2) ellipse (0.6cm and 0.2cm);
\draw[thick] (20.8,2) -- (20.8,-2);\draw[thick] (21.2,2)..controls +(0,-2) and +(0,+2) .. (27,-2);  
\fill (24,2) ellipse (0.6cm and 0.2cm);\fill (24,-2) ellipse (0.6cm and 0.2cm);
\draw[thick] (24,2)..controls +(0,-2) and +(0,+2) .. (21,-2); 
\fill (27,2) ellipse (0.6cm and 0.2cm);\fill (27,-2) ellipse (0.6cm and 0.2cm);
\draw[thick] (27,2)..controls +(0,-2) and +(0,+2)..(24,-2);
\node at (38,0) {$(\{1,3\},\{1\},\{2\})$};

\end{tikzpicture}
\end{center}

$\bullet$ Let $n=2$ and  take $k_1=k_2=2$. We give below the three distinct fused permutations of $\cD_{\bk,3}$ (for each, we give a diagram and the corresponding sequence of multisets):

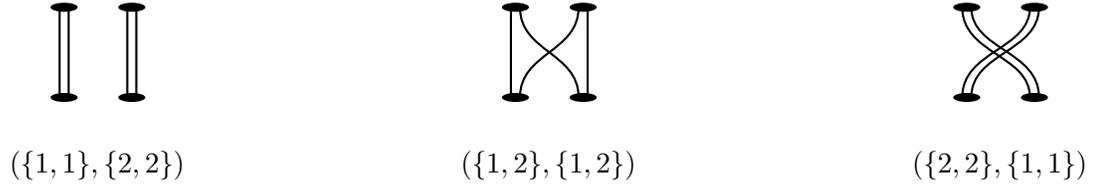
\begin{center}
\begin{tikzpicture}[scale=0.3]
\fill (1,2) ellipse (0.6cm and 0.2cm);\fill (1,-2) ellipse (0.6cm and 0.2cm);
\draw[thick] (0.8,2)..controls +(0,-2) and +(0,+2) .. (0.8,-2);\draw[thick] (1.2,2)..controls +(0,-2) and +(0,+2) .. (1.2,-2);  
\fill (4,2) ellipse (0.6cm and 0.2cm);\fill (4,-2) ellipse (0.6cm and 0.2cm);
\draw[thick] (3.8,2)..controls +(0,-2) and +(0,+2) .. (3.8,-2);\draw[thick] (4.2,2)..controls +(0,-2) and +(0,+2) .. (4.2,-2);
\node at (2.5,-5) {$(\{1,1\},\{2,2\})$};

\fill (21,2) ellipse (0.6cm and 0.2cm);\fill (21,-2) ellipse (0.6cm and 0.2cm);
\draw[thick] (20.8,2) -- (20.8,-2);\draw[thick] (21.2,2)..controls +(0,-2) and +(0,+2) .. (23.8,-2);  
\fill (24,2) ellipse (0.6cm and 0.2cm);\fill (24,-2) ellipse (0.6cm and 0.2cm);
\draw[thick] (23.8,2)..controls +(0,-2) and +(0,+2) .. (21.2,-2); \draw[thick] (24.2,2)..controls +(0,-2) and +(0,+2) .. (24.2,-2);
\node at (22.5,-5) {$(\{1,2\},\{1,2\})$};

\fill (41,2) ellipse (0.6cm and 0.2cm);\fill (41,-2) ellipse (0.6cm and 0.2cm);
\draw[thick] (40.8,2)..controls +(0,-2.1) and +(0,+1.9) .. (43.8,-2);\draw[thick] (41.2,2)..controls +(0,-1.9) and +(0,+2.1) .. (44.2,-2);
\fill (44,2) ellipse (0.6cm and 0.2cm);\fill (44,-2) ellipse (0.6cm and 0.2cm);
\draw[thick] (43.8,2)..controls +(0,-1.9) and +(0,+2.1) .. (40.8,-2);\draw[thick] (44.2,2)..controls +(0,-2) and +(0,+2) .. (41.2,-2); 
\node at (42.5,-5) {$(\{2,2\},\{1,1\})$};
\end{tikzpicture}
\end{center}

$\bullet$ Let $n=3$ and  take $k_1=k_2=k_3=2$. There are 21 distinct fused permutations in ${\cD}_{\bk,3}$  and here are three examples (for each, we give a diagram and the corresponding sequence of multisets):
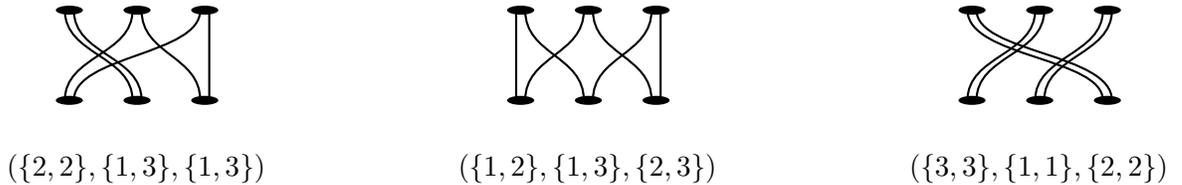
\begin{center}
\begin{tikzpicture}[scale=0.3]
\fill (1,2) ellipse (0.6cm and 0.2cm);\fill (1,-2) ellipse (0.6cm and 0.2cm);
\draw[thick] (0.8,2)..controls +(0,-2) and +(0,+2) .. (3.8,-2);\draw[thick] (1.2,2)..controls +(0,-2) and +(0,+2) .. (4.2,-2);  
\fill (4,2) ellipse (0.6cm and 0.2cm);\fill (4,-2) ellipse (0.6cm and 0.2cm);
\draw[thick] (3.8,2)..controls +(0,-2) and +(0,+2) .. (0.8,-2);\draw[thick] (4.2,2)..controls +(0,-2) and +(0,+2) .. (6.8,-2);
\fill (7,2) ellipse (0.6cm and 0.2cm);\fill (7,-2) ellipse (0.6cm and 0.2cm);
\draw[thick] (6.8,2)..controls +(0,-2) and +(0,+2) .. (1.2,-2);\draw[thick] (7.2,2) -- (7.2,-2);
\node at (4,-5) {$(\{2,2\},\{1,3\},\{1,3\})$};

\fill (21,2) ellipse (0.6cm and 0.2cm);\fill (21,-2) ellipse (0.6cm and 0.2cm);
\draw[thick] (20.8,2) -- (20.8,-2);\draw[thick] (21.2,2)..controls +(0,-2) and +(0,+2) .. (23.8,-2);  
\fill (24,2) ellipse (0.6cm and 0.2cm);\fill (24,-2) ellipse (0.6cm and 0.2cm);
\draw[thick] (23.8,2)..controls +(0,-2) and +(0,+2) .. (21.2,-2); \draw[thick] (26.8,2)..controls +(0,-2) and +(0,+2) .. (24.2,-2);
\fill (27,2) ellipse (0.6cm and 0.2cm);\fill (27,-2) ellipse (0.6cm and 0.2cm);
\draw[thick] (24.2,2)..controls +(0,-2) and +(0,+2) .. (26.8,-2);\draw[thick] (27.2,2) -- (27.2,-2);
\node at (24,-5) {$(\{1,2\},\{1,3\},\{2,3\})$};

\fill (41,2) ellipse (0.6cm and 0.2cm);\fill (41,-2) ellipse (0.6cm and 0.2cm);
\draw[thick] (40.8,2)..controls +(0,-2.1) and +(0,+1.9) .. (46.8,-2);\draw[thick] (41.2,2)..controls +(0,-1.9) and +(0,+2.1) .. (47.2,-2);
\fill (44,2) ellipse (0.6cm and 0.2cm);\fill (44,-2) ellipse (0.6cm and 0.2cm);
\draw[thick] (43.8,2)..controls +(0,-2) and +(0,+2) .. (40.8,-2);\draw[thick] (44.2,2)..controls +(0,-2) and +(0,+2) .. (41.2,-2); 
\fill (47,2) ellipse (0.6cm and 0.2cm);\fill (47,-2) ellipse (0.6cm and 0.2cm);
\draw[thick] (46.8,2)..controls +(0,-2) and +(0,+2) .. (43.8,-2);\draw[thick] (47.2,2)..controls +(0,-2) and +(0,+2) .. (44.2,-2);
\node at (44,-5) {$(\{3,3\},\{1,1\},\{2,2\})$};
\end{tikzpicture}
\end{center}

\paragraph{Multiplication.} We define the associative $\mathbb{C}$-algebra $H_{\bk,n}(1)$ as the $\bC$-vector space with basis indexed by the fused permutations in $\cD_{\bk,n}$, and with the multiplication given as follows. Let $d,d'\in \cD_{\bk,n}$ and we identify $d$, respectively, $d'$, with a diagram representing it.
\begin{itemize}
\item \emph{(Concatenation)} We place the diagram of $d$ on top of the diagram of $d'$ by identifying the bottom ellipses of $d$ with the top ellipses of $d'$.
\item \emph{(Removal of middle ellipses)} For each $a\in\{1,\dots,n\}$, there are $k_a$ edges arriving and $k_a$ edges leaving the $a$-th ellipse in the middle row. So for each $a\in\{1,\dots,n\}$, we delete the $a$-th ellipse in the middle row and sum over all possibilities of connecting the $k_a$ edges arriving at it to the $k_a$ edges leaving from it (at the $a$-th edge, there are thus $k_a!$ possibilities).
 \item  \emph{(Normalisation)} We divide the resulting sum by $k_1!\dots k_n!$.
\end{itemize}
At the end of the procedure described above, we obtain a sum of diagrams representing a sum of fused permutations (with rational coefficients). This is what we define to be $dd'$ in $H_{\bk,n}(1)$. This diagrammatic multiplication is well-defined since the result clearly depends only on the equivalences classes of the diagrams.

The algebra $H_{\bk,n}(1)$ is an associative algebra with unit, the unit element is the fused permutation corresponding to the diagram with only vertical edges (all the edges starting from the $a$-th top ellipse go the $a$-th bottom ellipse). 

\begin{rem}
The algebra $H_{\bk,n}(1)$ can be defined over any ring in which $k_1!\dots k_n!$ is invertible.
\end{rem}

\paragraph{Examples.} $\bullet$ If $\bk=(1,1,1,\dots)$ consists only of 1's then the algebra $H_{\bk,n}(1)$ obviously coincides with the complex group algebra $\CC\mS_n$ of the symmetric group $\mS_n$ on $n$ letters.

$\bullet$ Here is an example of a product of two elements of $H_{\bk,2}(1)$ with $k_1=k_2=2$:
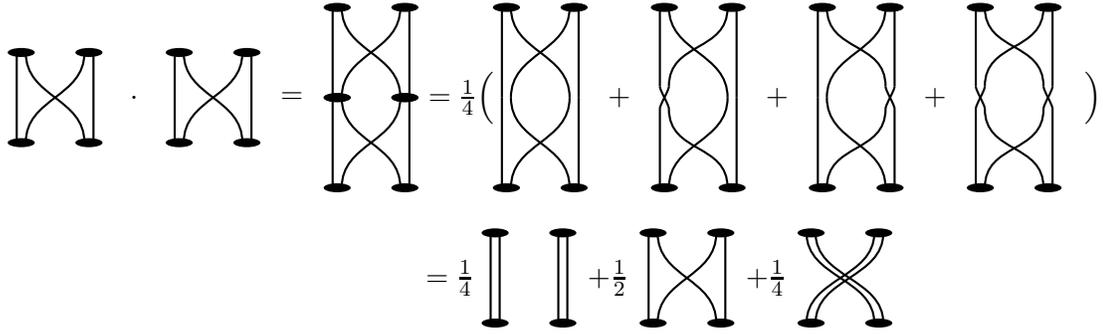
\begin{center}
\begin{tikzpicture}[scale=0.3]
\fill (1,2) ellipse (0.6cm and 0.2cm);\fill (1,-2) ellipse (0.6cm and 0.2cm);
\draw[thick] (0.8,2) -- (0.8,-2);\draw[thick] (1.2,2)..controls +(0,-2) and +(0,+2) .. (3.8,-2);  
\fill (4,2) ellipse (0.6cm and 0.2cm);\fill (4,-2) ellipse (0.6cm and 0.2cm);
\draw[thick] (3.8,2)..controls +(0,-2) and +(0,+2) .. (1.2,-2); \draw[thick] (4.2,2) -- (4.2,-2);
\node at (6,0) {$.$};
\fill (8,2) ellipse (0.6cm and 0.2cm);\fill (8,-2) ellipse (0.6cm and 0.2cm);
\draw[thick] (7.8,2) -- (7.8,-2);\draw[thick] (8.2,2)..controls +(0,-2) and +(0,+2) .. (10.8,-2);  
\fill (11,2) ellipse (0.6cm and 0.2cm);\fill (11,-2) ellipse (0.6cm and 0.2cm);
\draw[thick] (10.8,2)..controls +(0,-2) and +(0,+2) .. (8.2,-2); \draw[thick] (11.2,2) -- (11.2,-2);
\node at (13,0) {$=$};
\fill (15,4) ellipse (0.6cm and 0.2cm);\fill (15,0) ellipse (0.6cm and 0.2cm);
\draw[thick] (14.8,4) -- (14.8,0);\draw[thick] (15.2,4)..controls +(0,-2) and +(0,+2) .. (17.8,0);  
\fill (18,4) ellipse (0.6cm and 0.2cm);\fill (18,0) ellipse (0.6cm and 0.2cm);
\draw[thick] (17.8,4)..controls +(0,-2) and +(0,+2) .. (15.2,0); \draw[thick] (18.2,4) -- (18.2,0);

\fill (15,0) ellipse (0.6cm and 0.2cm);\fill (15,-4) ellipse (0.6cm and 0.2cm);
\draw[thick] (14.8,0) -- (14.8,-4);\draw[thick] (15.2,0)..controls +(0,-2) and +(0,+2) .. (17.8,-4);  
\fill (18,0) ellipse (0.6cm and 0.2cm);\fill (18,-4) ellipse (0.6cm and 0.2cm);
\draw[thick] (17.8,0)..controls +(0,-2) and +(0,+2) .. (15.2,-4); \draw[thick] (18.2,0) -- (18.2,-4);

\node at (20.5,0) {$=\frac{1}{4}\Bigl($};
\fill (22.5,4) ellipse (0.6cm and 0.2cm);
\draw[thick] (22.3,4) -- (22.3,0);\draw[thick] (22.7,4)..controls +(0,-2) and +(0,+2) .. (25.3,0);  
\fill (25.5,4) ellipse (0.6cm and 0.2cm);
\draw[thick] (25.3,4)..controls +(0,-2) and +(0,+2) .. (22.7,0); \draw[thick] (25.7,4) -- (25.7,0);

\fill (22.5,-4) ellipse (0.6cm and 0.2cm);
\draw[thick] (22.3,0) -- (22.3,-4);\draw[thick] (22.7,0)..controls +(0,-2) and +(0,+2) .. (25.3,-4);  
\fill (25.5,-4) ellipse (0.6cm and 0.2cm);
\draw[thick] (25.3,0)..controls +(0,-2) and +(0,+2) .. (22.7,-4); \draw[thick] (25.7,0) -- (25.7,-4);

\node at (27.5,0) {$+$};
\fill (29.5,4) ellipse (0.6cm and 0.2cm);
\draw[thick] (29.3,4) -- (29.3,0.5);\draw[thick] (29.7,4)..controls +(0,-2) and +(0,+2) .. (32.3,0);  
\fill (32.5,4) ellipse (0.6cm and 0.2cm);
\draw[thick] (32.3,4)..controls +(0,-2) and +(0,+2) .. (29.7,0.5); \draw[thick] (32.7,4) -- (32.7,0);

\draw[thick] (29.3,0.5)..controls +(0,-0.2) and +(0,+0.2) .. (29.7,-0.5);
\draw[thick] (29.7,0.5)..controls +(0,-0.2) and +(0,+0.2) .. (29.3,-0.5);

\fill (29.5,-4) ellipse (0.6cm and 0.2cm);
\draw[thick] (29.3,-0.5) -- (29.3,-4);\draw[thick] (29.7,-0.5)..controls +(0,-2) and +(0,+2) .. (32.3,-4);  
\fill (32.5,-4) ellipse (0.6cm and 0.2cm);
\draw[thick] (32.3,0)..controls +(0,-2) and +(0,+2) .. (29.7,-4); \draw[thick] (32.7,0) -- (32.7,-4);

\node at (34.5,0) {$+$};
\fill (36.5,4) ellipse (0.6cm and 0.2cm);
\draw[thick] (36.3,4) -- (36.3,0);\draw[thick] (36.7,4)..controls +(0,-2) and +(0,+2) .. (39.3,0.5);  
\fill (39.5,4) ellipse (0.6cm and 0.2cm);
\draw[thick] (39.3,4)..controls +(0,-2) and +(0,+2) .. (36.7,0); \draw[thick] (39.7,4) -- (39.7,0.5);

\draw[thick] (39.3,0.5)..controls +(0,-0.2) and +(0,+0.2) .. (39.7,-0.5);
\draw[thick] (39.7,0.5)..controls +(0,-0.2) and +(0,+0.2) .. (39.3,-0.5);

\fill (36.5,-4) ellipse (0.6cm and 0.2cm);
\draw[thick] (36.3,0) -- (36.3,-4);\draw[thick] (36.7,0)..controls +(0,-2) and +(0,+2) .. (39.3,-4);  
\fill (39.5,-4) ellipse (0.6cm and 0.2cm);
\draw[thick] (39.3,-0.5)..controls +(0,-2) and +(0,+2) .. (36.7,-4); \draw[thick] (39.7,-0.5) -- (39.7,-4);

\node at (41.5,0) {$+$};
\fill (43.5,4) ellipse (0.6cm and 0.2cm);
\draw[thick] (43.3,4) -- (43.3,0.5);\draw[thick] (43.7,4)..controls +(0,-2) and +(0,+2) .. (46.3,0.5);  
\fill (46.5,4) ellipse (0.6cm and 0.2cm);
\draw[thick] (46.3,4)..controls +(0,-2) and +(0,+2) .. (43.7,0.5); \draw[thick] (46.7,4) -- (46.7,0.5);

\draw[thick] (43.3,0.5)..controls +(0,-0.2) and +(0,+0.2) .. (43.7,-0.5);
\draw[thick] (43.7,0.5)..controls +(0,-0.2) and +(0,+0.2) .. (43.3,-0.5);
\draw[thick] (46.3,0.5)..controls +(0,-0.2) and +(0,+0.2) .. (46.7,-0.5);
\draw[thick] (46.7,0.5)..controls +(0,-0.2) and +(0,+0.2) .. (46.3,-0.5);

\fill (43.5,-4) ellipse (0.6cm and 0.2cm);
\draw[thick] (43.3,-0.5) -- (43.3,-4);\draw[thick] (43.7,-0.5)..controls +(0,-2) and +(0,+2) .. (46.3,-4);  
\fill (46.5,-4) ellipse (0.6cm and 0.2cm);
\draw[thick] (46.3,-0.5)..controls +(0,-2) and +(0,+2) .. (43.7,-4); \draw[thick] (46.7,-0.5) -- (46.7,-4);

\node at (48.5,0) {$\Bigr)$};

\node at (20,-8) {$=\frac{1}{4}$};
\fill (22,-6) ellipse (0.6cm and 0.2cm);\fill (22,-10) ellipse (0.6cm and 0.2cm);
\draw[thick] (21.8,-6) -- (21.8,-10);\draw[thick] (22.2,-6) -- (22.2,-10);
\fill (25,-6) ellipse (0.6cm and 0.2cm);\fill (25,-10) ellipse (0.6cm and 0.2cm);
\draw[thick] (24.8,-6) -- (24.8,-10);\draw[thick] (25.2,-6) -- (25.2,-10);

\node at (27,-8) {$+\frac{1}{2}$};

\fill (29,-6) ellipse (0.6cm and 0.2cm);\fill (29,-10) ellipse (0.6cm and 0.2cm);
\draw[thick] (28.8,-6) -- (28.8,-10);\draw[thick] (29.2,-6)..controls +(0,-2) and +(0,+2) .. (31.8,-10);  
\fill (32,-6) ellipse (0.6cm and 0.2cm);\fill (32,-10) ellipse (0.6cm and 0.2cm);
\draw[thick] (31.8,-6)..controls +(0,-2) and +(0,+2) .. (29.2,-10);  \draw[thick] (32.2,-6) -- (32.2,-10);

\node at (34,-8) {$+\frac{1}{4}$};

\fill (36,-6) ellipse (0.6cm and 0.2cm);\fill (36,-10) ellipse (0.6cm and 0.2cm);
\draw[thick] (35.8,-6)..controls +(0,-2) and +(0,+2) .. (38.8,-10);  
\draw[thick] (36.2,-6)..controls +(0,-2) and +(0,+2) .. (39.2,-10);  
\fill (39,-6) ellipse (0.6cm and 0.2cm);\fill (39,-10) ellipse (0.6cm and 0.2cm);
\draw[thick] (38.8,-6)..controls +(0,-2) and +(0,+2) .. (35.8,-10);  
\draw[thick] (39.2,-6)..controls +(0,-2) and +(0,+2) .. (36.2,-10);  

\end{tikzpicture}
\end{center}

$\bullet$ Here is an example of a product in $H_{\bk,3}(1)$ with $k_1=2$ and $k_2=k_3=1$:
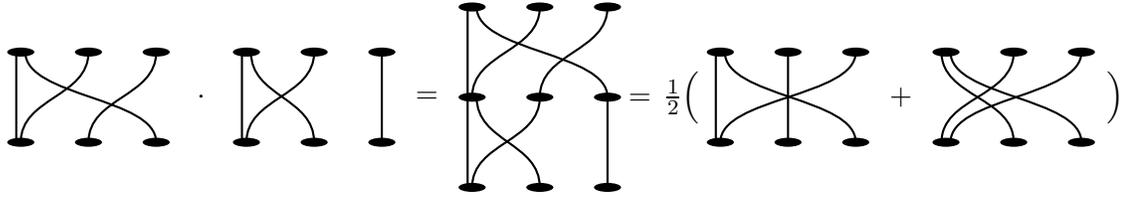
\begin{center}
\begin{tikzpicture}[scale=0.3]

\fill (1,0) ellipse (0.6cm and 0.2cm);\fill (1,-4) ellipse (0.6cm and 0.2cm);
\draw[thick] (0.8,0) -- (0.8,-4);\draw[thick] (1.2,0)..controls +(0,-2) and +(0,+2) .. (7,-4);  
\fill (4,0) ellipse (0.6cm and 0.2cm);\fill (4,-4) ellipse (0.6cm and 0.2cm);
\draw[thick] (4,0)..controls +(0,-2) and +(0,+2) .. (1,-4); 
\fill (7,0) ellipse (0.6cm and 0.2cm);\fill (7,-4) ellipse (0.6cm and 0.2cm);
\draw[thick] (7,0)..controls +(0,-2) and +(0,+2)..(4,-4);

\node at (9,-2) {$\cdot$};

\fill (11,-4) ellipse (0.6cm and 0.2cm);\fill (11,0) ellipse (0.6cm and 0.2cm);
\draw[thick] (10.8,0) -- (10.8,-4);\draw[thick] (11.2,0)..controls +(0,-2) and +(0,+2) .. (14,-4);  
\fill (14,-4) ellipse (0.6cm and 0.2cm);\fill (14,0) ellipse (0.6cm and 0.2cm);
\draw[thick] (14,0)..controls +(0,-2) and +(0,+2) .. (11,-4); 
\fill (17,-4) ellipse (0.6cm and 0.2cm);\fill (17,0) ellipse (0.6cm and 0.2cm);
\draw[thick] (17,0)..controls +(0,-2) and +(0,+2)..(17,-4);

\node at (19,-2) {$=$};

\fill (21,2) ellipse (0.6cm and 0.2cm);\fill (21,-2) ellipse (0.6cm and 0.2cm);
\draw[thick] (20.8,2) -- (20.8,-2);\draw[thick] (21.2,2)..controls +(0,-2) and +(0,+2) .. (27,-2);  
\fill (24,2) ellipse (0.6cm and 0.2cm);\fill (24,-2) ellipse (0.6cm and 0.2cm);
\draw[thick] (24,2)..controls +(0,-2) and +(0,+2) .. (21,-2); 
\fill (27,2) ellipse (0.6cm and 0.2cm);\fill (27,-2) ellipse (0.6cm and 0.2cm);
\draw[thick] (27,2)..controls +(0,-2) and +(0,+2)..(24,-2);

\fill (21,-6) ellipse (0.6cm and 0.2cm);
\draw[thick] (20.8,-2) -- (20.8,-6);\draw[thick] (21.2,-2)..controls +(0,-2) and +(0,+2) .. (24,-6);  
\fill (24,-6) ellipse (0.6cm and 0.2cm);
\draw[thick] (24,-2)..controls +(0,-2) and +(0,+2) .. (21,-6); 
\fill (27,-6) ellipse (0.6cm and 0.2cm);
\draw[thick] (27,-2)..controls +(0,-2) and +(0,+2)..(27,-6);

\node at (29.5,-2) {$=\,\frac{1}{2}\Bigl($};

\fill (32,0) ellipse (0.6cm and 0.2cm);\fill (32,-4) ellipse (0.6cm and 0.2cm);
\draw[thick] (31.8,0) -- (31.8,-4);\draw[thick] (32.2,0)..controls +(0,-2) and +(0,+2) .. (38,-4);  
\fill (35,0) ellipse (0.6cm and 0.2cm);\fill (35,-4) ellipse (0.6cm and 0.2cm);
\draw[thick] (35,0)..controls +(0,-2) and +(0,+2) .. (35,-4); 
\fill (38,0) ellipse (0.6cm and 0.2cm);\fill (38,-4) ellipse (0.6cm and 0.2cm);
\draw[thick] (38,0)..controls +(0,-2) and +(0,+2)..(32,-4);

\node at (40,-2) {$+$};

\fill (42,0) ellipse (0.6cm and 0.2cm);\fill (42,-4) ellipse (0.6cm and 0.2cm);
\draw[thick] (41.8,0) ..controls +(0,-2) and +(0,+2) ..  (45,-4);\draw[thick] (42.2,0)..controls +(0,-2) and +(0,+2) .. (48,-4);  
\fill (45,0) ellipse (0.6cm and 0.2cm);\fill (45,-4) ellipse (0.6cm and 0.2cm);
\draw[thick] (45,0)..controls +(0,-2) and +(0,+2) .. (41.8,-4); 
\fill (48,0) ellipse (0.6cm and 0.2cm);\fill (48,-4) ellipse (0.6cm and 0.2cm);
\draw[thick] (48,0)..controls +(0,-2) and +(0,+2)..(42.2,-4);

\node at (49.5,-2) {$\Bigr)$};

\end{tikzpicture}
\end{center}

\subsection{Double cosets and standard basis of $H_{\bk,n}(1)$}\label{subsec-doublecoset}

\paragraph{Algebra of double cosets.} We consider the following subgroup of the symmetric group $\mS_{k_1+\dots+k_n}$:
\[\mS^{\bk,n}=\mS_{k_1}\times\dots\times\mS_{k_n}\ ,\]
where $\mS_{k_1}$ is embedded as the subgroup permuting the letters $1,\dots,k_1$, $\mS_{k_2}$ is embedded as the subgroup permuting letters $k_1+1,\dots,k_1+k_2$, and so on. We denote $\mS^{\bk,n}\backslash\mS_{k_1+\dots+k_n}/\mS^{\bk,n}$ the set of double cosets of $\mS_{k_1+\dots+k_n}$ with respect of $\mS^{\bk,n}$.

We define the following element in the group algebra $\CC\mS_{k_1+\dots+k_n}$:
\[P_{\bk,n}=\frac{1}{k_1!\dots k_n!}\sum_{w\in\mS^{\bk,n}}w\ .\]
This element $P_{\bk,n}$ is an idempotent which satisfies $wP_{\bk,n}=P_{\bk,n}w=P_{\bk,n}$ for any $w\in \mS^{\bk,n}$. Then the subset $P_{\bk,n}\CC\mS_{k_1+\dots+k_n}P_{\bk,n}=\{P_{\bk,n}xP_{\bk,n}\ |\ x\in \CC\mS_{k_1+\dots+k_n}\}$ is an algebra with unit $P_{\bk,n}$.

\begin{prop}
The algebra $H_{\bk,n}(1)$ is isomorphic to $P_{\bk,n}\CC\mS_{k_1+\dots+k_n}P_{\bk,n}$.
\end{prop}
\begin{proof}
For $w\in \mS_{k_1+\dots+k_n}$, we see $w$ as a diagram made of two rows of $k_1+\dots+k_n$ dots connected by edges according to the permutation $w$ (the $i$-th top dot is connected to the $w(i)$-th bottom dot). The multiplication in $\mS_{k_1+\dots+k_n}$ is then simply the concatenation of diagrams.

Then, in the diagram of $w\in \mS_{k_1+\dots+k_n}$, in each of the two rows of dots, we glue the $k_1$ first dots into an ellipse, and the $k_2$ next dots into an ellipse and so on. We obtain thus a diagram as in the preceding subsection. We denote $[w]$ the corresponding fused permutation in $\cD_{\bk,n}$ (that is, the equivalence class of the diagram).

First, it is immediate that any fused permutation in $\cD_{\bk,n}$ can be written as $[w]$ for some $w\in\mS_{k_1+\dots+k_n}$. Moreover, we claim that, for $w_1,w_2\in\mS_{k_1+\dots+k_n}$, we have
\[[w_1]=[w_2]\ \ \ \ \Leftrightarrow\ \ \ \ \text{$w_1$ and $w_2$ are in the same double coset in $\mS^{\bk,n}\backslash\mS_{k_1+\dots+k_n}/\mS^{\bk,n}$.}\]
To prove the claim, then note first that if $w_1=xw_2y$ with $x,y\in\mS^{\bk,n}$ then it is clear that $[w_1]=[w_2]$ since the gluing of dots into ellipses will make trivial the effect of $x$ and $y$.

Reciprocally, assume that $[w_1]=[w_2]$. Up to top concatenation by elements of $\mS^{\bk,n}$, we can assume that, for $a=1,\dots,n$, we have
\[w(k_1+\dots+k_{a-1}+1)<w(k_1+\dots+k_{a-1}+2)<\dots<w(k_1+\dots+k_{a})\,\]
and similarly for $w_2$. Then $[w_1]=[w_2]$ means that $w_1$ and $w_2$ differs only by a bottom concatenation with a element of $\mS^{\bk,n}$. This proves that $w_1$ and $w_2$ differs only by left and right multiplication by elements of $\mS^{\bk,n}$.

Now choose a set $\mathcal{C}\subset \mS_{k_1+\dots+k_n}$ of representatives for the double cosets $\mS^{\bk,n}\backslash\mS_{k_1+\dots+k_n}/\mS^{\bk,n}$. From the previous claim, we have that the set $\{[w]\ |\ w\in\mathcal{C}\}$ is a basis of $H_{\bk,n}(1)$, while on the other hand, the set $\{P_{\bk,n}wP_{\bk,n}\ |\ w\in\mathcal{C}\}$ is clearly a basis of $P_{\bk,n}\CC\mS_{k_1+\dots+k_n}P_{\bk,n}$.

We conclude that the linear map defined by $[w]\to P_{\bk,n}wP_{\bk,n}$ is the desired isomorphism of algebras since the multiplication of diagrams $[w].[w']$ corresponds by construction to the multiplication $P_{\bk,n}wP_{\bk,n}w'P_{\bk,n}$ in $P_{\bk,n}\CC\mS_{k_1+\dots+k_n}P_{\bk,n}$.
\end{proof}

\begin{rem} The algebra $P_{\bk,n}\CC\mS_{k_1+\dots+k_n}P_{\bk,n}\cong H_{\bk,n}(1)$ is an example of a Hecke algebras, see \emph{e.g.} \cite[\S 11D]{CR}. More precisely here, it is the algebra of functions on the space $\mS^{\bk,n}\backslash\mS_{k_1+\dots+k_n}/\mS^{\bk,n}$ of double cosets (the multiplication corresponds to the natural convolution product). Thus $H_{\bk,n}(1)$ can also be seen as the endomorphism algebra $\text{End}_{\mS_{k_1+\dots+k_n}}(M_{\bk})$, where $M_{\bk}$ is the permutation module associated to the composition $(k_1,\dots,k_n)$, that is, the module induced form the trivial representation of $\mS^{\bk,n}$. In this paper, we will reserve the name Hecke algebra for the deformation of the symmetric group.
\end{rem}

\paragraph{Standard basis.} The subgroup $\mS^{\bk,n}$ is a parabolic subgroup of $\mS_{k_1+\dots+k_n}$, and as such it enjoys special properties. We refer to \cite[\S 2.1]{GP}.

For any $\pi\in \mS_{k_1+\dots+k_n}$ there exist unique elements $w\in \mS_{k_1+\dots+k_n}$ and $\pi_1\in\mS^{\bk,n}$ such that $\pi=w\pi_1$ and $\ell(\pi)=\ell(w)+\ell(\pi_1)$, where $\ell$ is the usual length function on $\mS_{k_1+\dots+k_n}$. The element $w$ is the unique element of minimal length in the left coset $\pi\mS^{\bk,n}$.

The set formed by the elements of minimal length in their left coset is a set of representatives for  $\mS_{k_1+\dots+k_n}/\mS^{\bk,n}$, called the set of distinguished left coset representatives. Denote it by $X_{\bk,n}$.

Further, for any $\pi\in \mS_{k_1+\dots+k_n}$ there exists a unique element $w\in \mS_{k_1+\dots+k_n}$ such that $\pi=\pi_1w\pi_2$ with $\pi_1,\pi_2\in\mS^{\bk,n}$ and $\ell(\pi)=\ell(\pi_1)+\ell(w)+\ell(\pi_2)$. The element $w$ is the unique element of minimal length in the double coset of $\pi$.

The set formed by the elements of minimal length in their double coset is a set of representatives for  $\mS^{\bk,n}\backslash\mS_{k_1+\dots+k_n}/\mS^{\bk,n}$, called the set of distinguished double coset representatives. Besides, this set is equal to $X_{\bk,n}\cap X^{-1}_{\bk,n}$.

It is then easy to see that $w\in \mS_{k_1+\dots+k_n}$ is a distinguished double coset representatives if and only if we have:
\[\begin{array}{l}
w(k_1+\dots+k_{a-1}+1)<w(k_1+\dots+k_{a-1}+2)<\dots<w(k_1+\dots+k_{a})\,,\\[0.5em]
w^{-1}(k_1+\dots+k_{a-1}+1)<w^{-1}(k_1+\dots+k_{a-1}+2)<\dots<w^{-1}(k_1+\dots+k_{a})\,,
\end{array}\ \ \ \forall a=1,\dots,n\ .\]

From now on, for brevity, we will simply say $w\in \mS^{\bk,n}\backslash\mS_{k_1+\dots+k_n}/\mS^{\bk,n}$ to indicate that $w$ is one of these distinguished double coset representatives. Morever, we will denote by $f_w$ the fused permutation in $\cD_{\bk,n}$ corresponding to $w$ and will refer to the set
\begin{equation}\label{basis-clas}\{f_w\ |\ w\in \mS^{\bk,n}\backslash\mS_{k_1+\dots+k_n}/\mS^{\bk,n}\}
\end{equation}
as the standard basis of $H_{\bk,n}(1)$.

In the diagrammatic point of view, the choice of distinguished double coset representatives reflects the following fact: we can assume that the $k_a$ edges leaving the $a$-th top ellipse do not cross and similarly for the $k_a$ edges arriving at the $a$-th bottom ellipse (this for $a=1,\dots,n$). All the diagrams of fused permutations drawn above were drawn like this.

\section{Fused braids and the fused Hecke algebra}\label{sec-fus-br}

We refer for example to $\cite{GP}$ and $\cite{KT}$ for the standard facts we will recall on the Hecke algebra.

\subsection{The Hecke algebra}

Let $m>0$. Let $q\in\CC^{\times}$ such that $q^2$ is not a non trivial root of unity ($q^2=1$ is allowed). The Hecke algebra $H_m(q)$ is the $\CC$-algebra generated by elements $\si_1,\dots,\si_{m-1}$ with defining relations:
\begin{equation}\label{rel-H}
\begin{array}{ll}
\si_i^2=(q-q^{-1})\si_i+1 & \text{for}\ i\in\{1,\dots,m-1\}\,,\\[0.2em]
\si_i\si_{i+1}\si_i=\si_{i+1}\si_i\si_{i+1}\ \ \  & \text{for $i\in\{1,\dots,m-2\}$}\,,\\[0.2em]
\si_i\si_j=\si_j\si_i\ \ \  & \text{for $i,j\in\{1,\dots,m-1\}$ such that $|i-j|>1$}\,.
\end{array}
\end{equation} 
By convention $H_0(q):=\CC$ (note also that $H_1(q)=\CC$). If $q^2=1$ the Hecke algebra is the group algebra $\CC\mS_m$ of the symmetric group $\mS_m$. In this case, we denote $s_1,\dots,s_{m-1}$ the generators $\sigma_1,\dots,\sigma_{m-1}$; then $s_i$ corresponds to the transposition $(i,i+1)$ of  $\mS_m$. The restriction on $q$ ensures that $H_m(q)$ is semisimple and is a flat deformation isomorphic to $\CC\mS_m$.

For any element $w\in \mS_m$, let $w=s_{a_1}\dots s_{a_k}$ be a reduced expression for $w$ in terms of the generators $s_1,\dots,s_{m-1}$, and define $\si_w:=\si_{a_1}\dots \si_{a_k}\in H_m(q)$. This definition does not depend on the reduced expression for $w$ and it is a standard fact that the set $\{\si_w\}_{w\in\mS_m}$ forms a basis of $H_m(q)$.

For example, the following set of elements (where the product of sets $A.B$ is $\{a.b\ |\ a\in A,\ b\in B\}$) forms a basis of $H_m(q)$:
\begin{equation}\label{baseH}
\left\{\begin{array}{c} 1,\\ \si_1 \end{array}\right\}\cdot \left\{\begin{array}{c} 1,\\ \si_2, \\ \si_2\si_1 \end{array}\right\}\cdot \left\{\begin{array}{c} 1,\\ \si_3, \\ \si_3\si_2, \\ \si_3\si_2\si_1 \end{array}\right\}\cdot\ \ldots\ \cdot \left\{\begin{array}{c} 1,\\ \si_{m-1}, \\ \vdots \\ \si_{m-1}\dots\si_1 \end{array}\right\}\ .
\end{equation}

The algebras $\{H_m(q)\}_{m\geq0}$ form a chain of algebras:
\begin{equation}\label{chain-H}
\CC=H_0(q)\subset H_1(q)\subset H_2(q)\subset\dots\dots \subset H_m(q)\subset H_{m+1}(q)\subset \dots\dots\ ,
\end{equation}
where the natural inclusions of algebras are given by $H_m(q)\ni\si_i\mapsto\si_i\in H_{m+1}(q)$.

\paragraph{Diagrammatic presentation of $H_m(q)$.} We use the standard diagrammatic presentation of the Hecke algebra $H_m(q)$ coming from the standard diagrammatic presentation of the braid group (we refer for example to \cite{KT} for a precise formulation of braids and braid diagrams). 

Algebraically, the braid group is the group generated by $\sigma_1,\dots,\sigma_{m-1}$ and the second and third lines of relations in (\ref{rel-H}).

The diagrammatic presentation of a braid is by considering a rectangular strip with a line of $m$ dots at its top and a line of $m$ dots at its bottom. We connect each top dot to a bottom dot by a strand inside the strip. At each point of the strip at most two strands are intersecting, and at each intersection, we indicate which strand ``pass over'' the other one. An intersection is called a crossing and
we call a crossing  positive (resp. negative) when the strand coming from the left passes over (resp. under) the strand coming from the right.
Such diagram is called a braid with $m$ strands and braids are considered up to homotopy, which consists in being able of moving continuously the strands while leaving their end points fixed.

In terms of diagrams, the multiplication in the braid group is simply by concatenation of the diagram. If $\alpha$ and $\beta$ are two braids, to perform the product $\alpha\beta$, we place the diagram of $\alpha$ above the diagram of $\beta$ by identifying the bottom line of dots of $\alpha$ with the top line of dots in $\beta$, and then deleting the middle dots.

From now on we will always identify a braid with a braid diagram representing it. The identity element of the braid group is the braid where all the strands are vertical and parallel. Each generator $\si_i$ of the braid group is associated to the following braid:
\begin{center}
 \begin{tikzpicture}[scale=0.3]
\node at (0,0) {$\si_i=$};
\node at (2,3) {$1$};\fill (2,2) circle (0.2);\fill (2,-2) circle (0.2);
\draw[thick] (2,2) -- (2,-2);
\node at (4,0) {$\dots$};
\draw[thick] (6,2) -- (6,-2);
\node at (6,3) {$i-1$};\fill (6,2) circle (0.2);\fill (6,-2) circle (0.2);
\node at (10,3) {$i$};\fill (10,2) circle (0.2);\fill (10,-2) circle (0.2);
\node at (14,3) {$i+1$};\fill (14,2) circle (0.2);\fill (14,-2) circle (0.2);
\draw[thick] (14,2)..controls +(0,-2) and +(0,+2) .. (10,-2);
\fill[white] (12,0) circle (0.4);
\draw[thick] (10,2)..controls +(0,-2) and +(0,+2) .. (14,-2);
\draw[thick] (18,2) -- (18,-2);
\node at (18,3) {$i+2$};\fill (18,2) circle (0.2);\fill (18,-2) circle (0.2);
\node at (20,0) {$\dots$};
\draw[thick] (22,2) -- (22,-2);\fill (22,2) circle (0.2);\fill (22,-2) circle (0.2);
\node at (22,3) {$m$};
\end{tikzpicture}
\end{center}
The previous braid provides an example of positive crossing. The inverse $\si_i^{-1}$ of $\si_i$ is the following braid
\begin{center}
 \begin{tikzpicture}[scale=0.3]
\node at (-0.5,0) {$\si_i^{-1}=$};
\node at (2,3) {$1$};\fill (2,2) circle (0.2);\fill (2,-2) circle (0.2);
\draw[thick] (2,2) -- (2,-2);
\node at (4,0) {$\dots$};
\draw[thick] (6,2) -- (6,-2);
\node at (6,3) {$i-1$};\fill (6,2) circle (0.2);\fill (6,-2) circle (0.2);
\node at (10,3) {$i$};\fill (10,2) circle (0.2);\fill (10,-2) circle (0.2);
\node at (14,3) {$i+1$};\fill (14,2) circle (0.2);\fill (14,-2) circle (0.2);
\draw[thick] (10,2)..controls +(0,-2) and +(0,+2) .. (14,-2);
\fill[white] (12,0) circle (0.4);
\draw[thick] (14,2)..controls +(0,-2) and +(0,+2) .. (10,-2);
\draw[thick] (18,2) -- (18,-2);
\node at (18,3) {$i+2$};\fill (18,2) circle (0.2);\fill (18,-2) circle (0.2);
\node at (20,0) {$\dots$};
\draw[thick] (22,2) -- (22,-2);\fill (22,2) circle (0.2);\fill (22,-2) circle (0.2);
\node at (22,3) {$m$};
\end{tikzpicture}
\end{center}

The first relation in (\ref{rel-H}) can be read as $\sigma_i^{-1}=\sigma_i-(q-q^{-1})$, and so the Hecke algebra $H_m(q)$ has the following diagrammatic description: it is the algebra spanned by all braids with $m$ strands imposing moreover the following relation for any crossing:
\begin{center}
 \begin{tikzpicture}[scale=0.25]
\draw[thick] (0,2)..controls +(0,-2) and +(0,+2) .. (4,-2);
\fill[white] (2,0) circle (0.4);
\draw[thick] (4,2)..controls +(0,-2) and +(0,+2) .. (0,-2);
\node at (6,0) {$=$};
\draw[thick] (12,2)..controls +(0,-2) and +(0,+2) .. (8,-2);
\fill[white] (10,0) circle (0.4);
\draw[thick] (8,2)..controls +(0,-2) and +(0,+2) .. (12,-2);
\node at (17,0) {$-\,(q-q^{-1})$};
\draw[thick] (21,2) -- (21,-2);\draw[thick] (25,2) -- (25,-2);
\end{tikzpicture}
\end{center}
This relation has to be understood as a local relation, meaning that for any braid and any of its crossing, the braid is equal to a sum of two terms : the braid obtained by replacing the crossing by its opposite and $\pm(q-q^{-1})$ (depending on the sign of the original crossing) times the braid obtained by replacing the crossing by two pieces of vertical strands. In particular this allows one to transform all the negative crossings into positive ones. 

A diagrammatic interpretation of the basis $\{\sigma_w\ |\ w\in\mS_m\}$ of $H_m(q)$ is then the following: For any permutation $w$, take a permutation diagram representing $w$ (namely, the particular case $\bk=(1,1,1,\dots)$ of the diagrams defined in Section \ref{sec-def-fus-perm}) with a minimal number of intersections. And promote each intersection into a positive crossing and see the resulting diagram as an element of $H_m(q)$. This is $\sigma_w$.
 
\paragraph{The $q$-symmetriser of $H_m(q)$.} For $L\in\mathbb{Z}$, we define the $q$-numbers as follows
\begin{equation}\label{quantum-numbers}
[L]_q:=\frac{q^L-q^{-L}}{q-q^{-1}}=\pm(q^{L-1}+q^{L-3}+\dots+q^{-(L-1)})\,,\ \ \ \ \{L\}_q:=\frac{q^{2L}-1}{q^2-1}=\pm(1+q^2+\dots+q^{2(L-1)})\ ,
\end{equation}
the sign $\pm$ being the sign of $L$. We also set $[L]_q!:=[1]_q[2]_q\dots[L]_q$ and $\{L\}_q!:=\{1\}_q\{2\}_q\dots\{L\}_q$. 

The $q$-symmetriser of $H_m(q)$ is the following element:
\begin{equation}\label{def-P}
P_m:=\frac{\sum_{w\in\mS_m}q^{\ell(w)}\si_w}{\sum_{w\in\mS_m}q^{2\ell(w)}}=\frac{\sum_{w\in\mS_m}q^{\ell(w)}\si_w}{\{m\}_q!}=q^{-n(n-1)/2}\frac{\sum_{w\in\mS_m}q^{\ell(w)}\si_w}{[m]_q!}\ .
\end{equation}
Note that the first equality of the denominators is easy to see, by induction on $m$, from the basis (\ref{baseH}). It is well-known that the element $P_m$ is a minimal central idempotent in $H_m(q)$ corresponding to the one-dimensional representation given by $\sigma_i\mapsto q$ for $i=1,\dots,m-1$. In particular, we have
\begin{equation}\label{rel-sym}
P_m^2=P_m\ \ \ \ \ \text{and}\ \ \ \ \ \ \si_iP_m=P_m\si_i=q P_m\,,\ i=1,\dots,m-1.
\end{equation}
If $q^2=1$, the projector $P_m$ is the symmetriser of $\CC\mS_m$ projecting on the trivial representation of $\mS_m$.

\subsection{Fused braids}

As in the preceding section, let $\bk=(k_1,k_2,...)\in \mathbb{Z}_{\geq0}^{\infty}$ be an infinite sequence of non-negative integers, and let $n\in\mathbb{Z}_{>0}$\,.

\paragraph{Objects.} We consider the following objects, which are similar to the braid diagrams of the previous subsection, but in which we replace the two lines of dots by two lines of $n$ ellipses (drawn again as small black-filled ellipses). Moreover, we connect top ellipses with bottom ellipses by strands (as before with the usual braids) but now we require the following: for each $a\in\{1,\dots,n\}$, there are $k_a$ strands attached to the $a$-th top ellipse and $k_a$ strands attached to the $a$-th bottom ellipse. To be more precise, the strands which are attached to the same ellipse are not attached to the same point of the ellipse. Instead they are attached next to each other at the same ellipse (hence the need of ellipses instead of points or dots as for the usual braids). Examples are drawn below. The total number of strands is then $k_1+\dots+k_n$.

As before we require that at each point of the strip at most two strands are intersecting and we keep the same terminology of positive and negative crossings. Again as before we consider such diagrams up to homotopy, namely up to continuously moving the strands while leaving their end points fixed.

Such an equivalence class of diagrams we call a fused braid (we will not use a more precise name such as $(\bk,n)$-fused braid) and we will from now on identify a fused braid with a diagram representing it. 

\paragraph{Examples.} $\bullet$ If $\bk=(1,1,1,\dots)$ consists only of 1's then a fused braid is simply a usual braid.

$\bullet$ Here are examples of 6 fused braids when $k_1=k_2=k_3=2$:
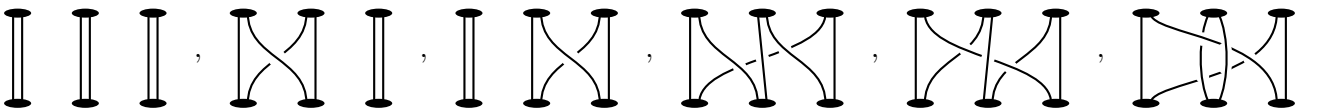
\begin{center}
\begin{tikzpicture}[scale=0.3]
\fill (1,2) ellipse (0.6cm and 0.2cm);\fill (1,-2) ellipse (0.6cm and 0.2cm);
\draw[thick] (0.8,2) -- (0.8,-2);\draw[thick] (1.2,2) -- (1.2,-2);
\fill (4,2) ellipse (0.6cm and 0.2cm);\fill (4,-2) ellipse (0.6cm and 0.2cm);
\draw[thick] (3.8,2) -- (3.8,-2);\draw[thick] (4.2,2) -- (4.2,-2);
\fill (7,2) ellipse (0.6cm and 0.2cm);\fill (7,-2) ellipse (0.6cm and 0.2cm);
\draw[thick] (6.8,2) -- (6.8,-2);\draw[thick] (7.2,2) -- (7.2,-2);
\node at (9,0) {$,$};
\fill (11,2) ellipse (0.6cm and 0.2cm);\fill (11,-2) ellipse (0.6cm and 0.2cm);!
\draw[thick] (10.8,2) -- (10.8,-2);\draw[thick] (13.8,2)..controls +(0,-2) and +(0,+2) .. (11.2,-2);\fill[white] (12.5,0) circle (0.4);
\fill (14,2) ellipse (0.6cm and 0.2cm);\fill (14,-2) ellipse (0.6cm and 0.2cm);
\draw[thick] (14.2,2) -- (14.2,-2);\draw[thick] (11.2,2)..controls +(0,-2) and +(0,+2) .. (13.8,-2);
\fill (17,2) ellipse (0.6cm and 0.2cm);\fill (17,-2) ellipse (0.6cm and 0.2cm);
\draw[thick] (16.8,2) -- (16.8,-2);\draw[thick] (17.2,2) -- (17.2,-2);
\node at (19,0) {$,$};
\fill (21,2) ellipse (0.6cm and 0.2cm);\fill (21,-2) ellipse (0.6cm and 0.2cm);
\draw[thick] (20.8,2) -- (20.8,-2);\draw[thick] (21.2,2) -- (21.2,-2);
\fill (24,2) ellipse (0.6cm and 0.2cm);\fill (24,-2) ellipse (0.6cm and 0.2cm);
\draw[thick] (23.8,2) -- (23.8,-2);\draw[thick] (26.8,2)..controls +(0,-2) and +(0,+2) .. (24.2,-2);\fill[white] (25.5,0) circle (0.4);
\fill (27,2) ellipse (0.6cm and 0.2cm);\fill (27,-2) ellipse (0.6cm and 0.2cm);
\draw[thick] (24.2,2)..controls +(0,-2) and +(0,+2) .. (26.8,-2);\draw[thick] (27.2,2) -- (27.2,-2);
\node at (29,0) {$,$};
\draw[thick] (36.8,2)..controls +(0,-2) and +(0,+2) .. (31.2,-2);\draw[thick] (37.2,2) -- (37.2,-2);
\fill[white] (33,-0.4) circle (0.3);\fill[white] (35,0.4) circle (0.3);;\fill[white] (34,0) circle (0.3);
\fill (31,2) ellipse (0.6cm and 0.2cm);\fill (31,-2) ellipse (0.6cm and 0.2cm);
\draw[thick] (30.8,2) -- (30.8,-2);\draw[thick] (31.2,2)..controls +(0,-2) and +(0,+2) .. (33.8,-2);
\fill (34,2) ellipse (0.6cm and 0.2cm);\fill (34,-2) ellipse (0.6cm and 0.2cm);
\draw[thick] (33.8,2) -- (34.2,-2);\draw[thick] (34.2,2)..controls +(0,-2) and +(0,+2) .. (36.8,-2);
\fill (37,2) ellipse (0.6cm and 0.2cm);\fill (37,-2) ellipse (0.6cm and 0.2cm);
\node at (39,0) {$,$};
\draw[thick] (46.8,2)..controls +(0,-2) and +(0,+2) .. (44.2,-2);
\draw[thick] (43.8,2)..controls +(0,-2) and +(0,+2) .. (41.2,-2);
\fill[white] (43,0.4) circle (0.3);\fill[white] (45,-0.4) circle (0.3);
\draw[thick] (41.2,2)..controls +(0,-2) and +(0,+2) .. (46.8,-2);
\fill[white] (44,0) circle (0.3);
\fill (41,2) ellipse (0.6cm and 0.2cm);\fill (41,-2) ellipse (0.6cm and 0.2cm);
\draw[thick] (40.8,2) -- (40.8,-2);
\fill (44,2) ellipse (0.6cm and 0.2cm);\fill (44,-2) ellipse (0.6cm and 0.2cm);
\draw[thick] (44.2,2) -- (43.8,-2);
\fill (47,2) ellipse (0.6cm and 0.2cm);\fill (47,-2) ellipse (0.6cm and 0.2cm);
\draw[thick] (47.2,2) -- (47.2,-2);
\node at (49,0) {$,$};
\draw[thick] (56.8,2)..controls +(0,-3) and +(0,+1) .. (51.2,-2);
\fill[white] (53.5,-0.8) circle (0.3);\fill[white] (54.5,-0.4) circle (0.3);
\draw[thick] (53.8,2)..controls +(-0.5,-1) and +(-0.5,1) .. (53.8,-2);
\fill[white] (53.5,0.85) circle (0.3);
\fill[white] (54,0.8) circle (0.3);\fill[white] (55.6,0) circle (0.3);
\draw[thick] (51.2,2)..controls +(0,-1) and +(0,+3) .. (56.8,-2);
\fill[white] (54.5,0.4) circle (0.3);
\draw[thick] (54.2,2)..controls +(0.5,-1) and +(0.5,1) .. (54.2,-2);
\draw[thick] (50.8,2) -- (50.8,-2);
\fill (51,2) ellipse (0.6cm and 0.2cm);\fill (51,-2) ellipse (0.6cm and 0.2cm);
\fill (54,2) ellipse (0.6cm and 0.2cm);\fill (54,-2) ellipse (0.6cm and 0.2cm);
\fill (57,2) ellipse (0.6cm and 0.2cm);\fill (57,-2) ellipse (0.6cm and 0.2cm);
\draw[thick] (57.2,2) -- (57.2,-2);
\end{tikzpicture}
\end{center}

\subsection{Definition of the fused Hecke algebra $H_{\bk,n}(q)$}

In the definition below, we consider a vector space with basis indexed by fused braids and we identify the vector basis with their indices (in other words, we consider formal linear combinations of fused braids).
\begin{defi}\label{vector-fused-braids}
 The $\CC$-vector space $H_{\bk,n}(q)$ is the quotient of the vector space with basis indexed by fused braids by the following relations:
 \begin{itemize}
   \item[(i)] The Hecke relation:  
   \begin{center}
 \begin{tikzpicture}[scale=0.25]
\draw[thick] (0,2)..controls +(0,-2) and +(0,+2) .. (4,-2);
\fill[white] (2,0) circle (0.4);
\draw[thick] (4,2)..controls +(0,-2) and +(0,+2) .. (0,-2);
\node at (6,0) {$=$};
\draw[thick] (12,2)..controls +(0,-2) and +(0,+2) .. (8,-2);
\fill[white] (10,0) circle (0.4);
\draw[thick] (8,2)..controls +(0,-2) and +(0,+2) .. (12,-2);
\node at (17,0) {$-\,(q-q^{-1})$};
\draw[thick] (21,2) -- (21,-2);\draw[thick] (25,2) -- (25,-2);
\end{tikzpicture}
\end{center}
  \item[(ii)]  The idempotent relations: for top ellipses,
 \begin{center}
 \begin{tikzpicture}[scale=0.4]
\fill (2,2) ellipse (0.8cm and 0.2cm);
\draw[thick] (2.2,2)..controls +(0,-1.5) and +(1,1) .. (1.2,0);
\fill[white] (2,0.7) circle (0.2);
\draw[thick] (1.8,2)..controls +(0,-1.5) and +(-1,1) .. (2.8,0);
\node at (4.5,1) {$=$};
\node at (7,1) {$q$};
\fill (9,2) ellipse (0.8cm and 0.2cm);
\draw[thick] (8.8,2)..controls +(0,-1.5) and +(0.5,0.5) .. (8.2,0);
\draw[thick] (9.2,2)..controls +(0,-1.5) and +(-0.5,0.5) .. (9.8,0);

\node at (13,1) {and};

\fill (18,2) ellipse (0.8cm and 0.2cm);
\draw[thick] (17.8,2)..controls +(0,-1.5) and +(-1,1) .. (18.8,0);
\fill[white] (18,0.7) circle (0.2);
\draw[thick] (18.2,2)..controls +(0,-1.5) and +(1,1) .. (17.2,0);
\node at (20.5,1) {$= $};
\node at (23,1) {$q^{-1}$};
\fill (25,2) ellipse (0.8cm and 0.2cm);
\draw[thick] (24.8,2)..controls +(0,-1.5) and +(0.5,0.5) .. (24.2,0);
\draw[thick] (25.2,2)..controls +(0,-1.5) and +(-0.5,0.5) .. (25.8,0);
\end{tikzpicture}
\end{center}
and for bottom ellipses,
\begin{center}
 \begin{tikzpicture}[scale=0.4]
\fill (2,0) ellipse (0.8cm and 0.2cm);
\draw[thick] (1.8,0)..controls +(0,1.5) and +(-1,-1) .. (2.8,2);
\fill[white] (2,1.25) circle (0.2);
\draw[thick] (2.2,0)..controls +(0,1.5) and +(1,-1) .. (1.2,2);
\node at (4.5,1) {$=$};
\node at (7,1) {$q$};
\fill (9,0) ellipse (0.8cm and 0.2cm);
\draw[thick] (9.2,0)..controls +(0,1.5) and +(-0.5,-0.5) .. (9.8,2);
\draw[thick] (8.8,0)..controls +(0,1.5) and +(0.5,-0.5) .. (8.2,2);

\node at (13,1) {and};

\fill (18,0) ellipse (0.8cm and 0.2cm);
\draw[thick] (18.2,0)..controls +(0,1.5) and +(1,-1) .. (17.2,2);
\fill[white] (18,1.25) circle (0.2);
\draw[thick] (17.8,0)..controls +(0,1.5) and +(-1,-1) .. (18.8,2);
\node at (20.5,1) {$=$};
\node at (23,1) {$q^{-1}$};
\fill (25,0) ellipse (0.8cm and 0.2cm);
\draw[thick] (25.2,0)..controls +(0,1.5) and +(-0.5,-0.5) .. (25.8,2);
\draw[thick] (24.8,0)..controls +(0,1.5) and +(0.5,-0.5) .. (24.2,2);
\end{tikzpicture}
\end{center}
 \end{itemize}
\end{defi}
The first relation is the Hecke relation, valid locally for all crossings as in the situation of classical braids and Hecke algebra. The other relations are also local relations, valid for crossings near the ellipses. In words, they impose the following: in a fused braid, if two strands start from the same ellipse and their first crossing is crossing each other, then the original fused braid is equal to the fused braid obtained by removing this crossing and multiplying by $q^{\pm1}$ depending on the sign of the crossing; and similarly for two strands arriving at the same ellipse.

\paragraph{Multiplication.} Now we define a product on the vector space $H_{\bk,n}(q)$, which makes it an associative unital algebra. In order to multiply fused braids, we use the Hecke algebra and its $q$-symmetriser. Namely let $b,b'$ be two fused braids. We define $bb'$ as the result of the following procedure:
\begin{itemize}
\item \emph{(Concatenation)} We place the diagram of $b$ on top of the diagram of $b'$ by identifying the bottom ellipses of $b$ with the top ellipses of $b'$
\item \emph{(Removal of middle ellipses)} For each $a\in\{1,\dots,n\}$, there are $k_a$ strands arriving and $k_a$ strands leaving the $a$-th ellipse in the middle row. We remove this middle ellipse and replace it by the $q$-symmetriser $P_{k_a}$ of the Hecke algebra (\ref{def-P}).
\end{itemize}
More explicitly, in order to remove the $a$-th middle ellipse, we take $w\in\mS_{k_a}$ and we first construct the diagram where this middle ellipse is replaced by the element $\si_w$ of $H_{k_a}$ connecting the $k_a$ incoming strands to the $k_a$ outgoing ones. Then we make the sum over $w\in\mS_{k_a}$ of the resulting diagrams, each with the coefficient $q^{\ell(w)}$ (that is, multiplied by $q$ to the power the number of crossings we added). Finally, we normalise by dividing this sum by $\sum_{w\in\mS_m}q^{2\ell(w)}=\{m\}_q!$. 

It is immediate that the fused braid with only non-crossing vertical strands is the unit element for this multiplication.

\begin{defi}
The fused Hecke algebra is the algebra whose underlying vector space is $H_{\bk,n}(q)$ from Definition \ref{vector-fused-braids} and with multiplication defined with the procedure above. We continue to use $H_{\bk,n}(q)$ to denote this algebra .
\end{defi}

\begin{rem}\label{rem-def}
In this paper, we work for simplicity with a complex non-zero number $q$, with the condition that $q^2$ is not a non-trivial root of unity. This last condition is for staying in the semisimple regime. Nevertheless we point out that the algebras $H_{\bk,n}(q)$ can be defined for a complex non-zero number $q$ with the only condition that $q^{2l}\neq 1$ for $l=2,\dots,K$ where $K=\text{max}\{k_1,\dots,k_n\}$.

Alternatively, we can consider the generic algebra $H_{\bk,n}(q)$ defined over the ring $\CC[q,q^{-1},\bigl(\{K\}!\bigr)^{-1}]$, where $q$ is an indeterminate. We refer to this situation (only in the last section) as ``generic $q$''. The statements on representations are then to be understood over the field $\CC(q)$.
\end{rem}

\begin{exa}\label{ex:rem} We illustrate below the procedure to remove a middle ellipse when the number of incoming (and thus also of outgoing) strands is 2:
 
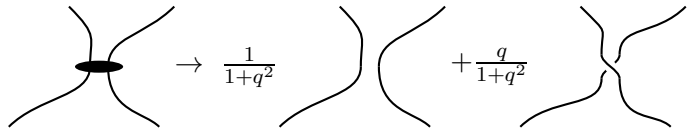
\begin{center}
 \begin{tikzpicture}[scale=0.4]
\fill (0,0) ellipse (0.8cm and 0.2cm);
\draw[thick] (-0.3,0)..controls +(0,1) and +(1,-1) .. (-1,2);
\draw[thick] (0.3,0)..controls +(0,1) and +(-1,-1) .. (2.5,2);
\draw[thick] (-0.3,0)..controls +(0,-1) and +(1,1) .. (-3,-2);
\draw[thick] (0.3,0)..controls +(0,-1.5) and +(-0.5,0.5) .. (2,-2);
\node at (3,0) {$\rightarrow$};
\node at (5,0) {$\frac{1}{1+q^2}$};
\draw[thick] (8.7,0)..controls +(0,1) and +(1,-1) .. (8,2);
\draw[thick] (9.3,0)..controls +(0,1) and +(-1,-1) .. (11.5,2);
\draw[thick] (8.7,0)..controls +(0,-1) and +(1,1) .. (6,-2);
\draw[thick] (9.3,0)..controls +(0,-1.5) and +(-0.5,0.5) .. (11,-2);
\node at (13,0) {$+\frac{q}{1+q^2}$};
\draw[thick] (16.7,0.4)..controls +(0,1) and +(1,-1) .. (16,2);
\draw[thick] (17.3,0.4)..controls +(0,1) and +(-1,-1) .. (19.5,2);
\draw[thick] (17.3,0.4)..controls +(0,-0.3) and +(0,0.3) .. (16.7,-0.4);
\fill[white] (17,0) circle (0.2);
\draw[thick] (16.7,0.4)..controls +(0,-0.3) and +(0,0.3) .. (17.3,-0.4);
\draw[thick] (16.7,-0.4)..controls +(0,-1) and +(1,1) .. (14,-2);
\draw[thick] (17.3,-0.4)..controls +(0,-1.5) and +(-0.5,0.5) .. (19,-2);
\end{tikzpicture}
\end{center}
\end{exa}

\paragraph{Examples.} $\bullet$ If $\bk=(1,1,1,\dots)$ consists only of 1's then the algebra $H_{\bk,n}(q)$ obviously coincides with the Hecke algebra $H_n(q)$.

$\bullet$ Here is an example of a product of two elements of $H_{\bk,2}(q)$ with $k_1=k_2=2$:
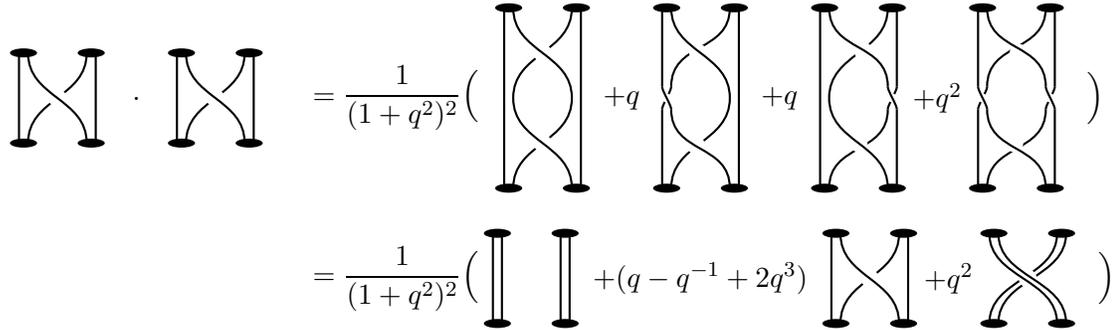
\begin{center}
\begin{tikzpicture}[scale=0.3]
\fill (1,2) ellipse (0.6cm and 0.2cm);\fill (1,-2) ellipse (0.6cm and 0.2cm);
\fill (4,2) ellipse (0.6cm and 0.2cm);\fill (4,-2) ellipse (0.6cm and 0.2cm);
\draw[thick] (3.8,2)..controls +(0,-2) and +(0,+2) .. (1.2,-2); \draw[thick] (4.2,2) -- (4.2,-2);
\fill[white] (2.5,0) circle (0.4);
\draw[thick] (0.8,2) -- (0.8,-2);\draw[thick] (1.2,2)..controls +(0,-2) and +(0,+2) .. (3.8,-2);  

\node at (6,0) {$.$};

\fill (8,2) ellipse (0.6cm and 0.2cm);\fill (8,-2) ellipse (0.6cm and 0.2cm);
\fill (11,2) ellipse (0.6cm and 0.2cm);\fill (11,-2) ellipse (0.6cm and 0.2cm);
\draw[thick] (10.8,2)..controls +(0,-2) and +(0,+2) .. (8.2,-2); \draw[thick] (11.2,2) -- (11.2,-2);
\fill[white] (9.5,0) circle (0.4);
\draw[thick] (7.8,2) -- (7.8,-2);\draw[thick] (8.2,2)..controls +(0,-2) and +(0,+2) .. (10.8,-2);  

\node at (17.5,0) {$=\displaystyle\frac{1}{(1+q^2)^2}\Bigl($};
\fill (22.5,4) ellipse (0.6cm and 0.2cm);
\fill (25.5,4) ellipse (0.6cm and 0.2cm);
\draw[thick] (25.3,4)..controls +(0,-2) and +(0,+2) .. (22.7,0); \draw[thick] (25.7,4) -- (25.7,0);
\fill[white] (24,2) circle (0.4);
\draw[thick] (22.3,4) -- (22.3,0);\draw[thick] (22.7,4)..controls +(0,-2) and +(0,+2) .. (25.3,0);  

\fill (22.5,-4) ellipse (0.6cm and 0.2cm);
\fill (25.5,-4) ellipse (0.6cm and 0.2cm);
\draw[thick] (25.3,0)..controls +(0,-2) and +(0,+2) .. (22.7,-4); \draw[thick] (25.7,0) -- (25.7,-4);
\fill[white] (24,-2) circle (0.4);
\draw[thick] (22.3,0) -- (22.3,-4);\draw[thick] (22.7,0)..controls +(0,-2) and +(0,+2) .. (25.3,-4);  

\node at (27.5,0) {$+ q$};
\fill (29.5,4) ellipse (0.6cm and 0.2cm);
\fill (32.5,4) ellipse (0.6cm and 0.2cm);
\draw[thick] (32.3,4)..controls +(0,-2) and +(0,+2) .. (29.7,0.5); \draw[thick] (32.7,4) -- (32.7,0);
\fill[white] (30.9,2) circle (0.4);
\draw[thick] (29.3,4) -- (29.3,0.5);\draw[thick] (29.7,4)..controls +(0,-2) and +(0,+2) .. (32.3,0);  

\draw[thick] (29.7,0.5)..controls +(0,-0.2) and +(0,+0.2) .. (29.3,-0.5);
\fill[white] (29.5,0) circle (0.2);
\draw[thick] (29.3,0.5)..controls +(0,-0.2) and +(0,+0.2) .. (29.7,-0.5);

\fill (29.5,-4) ellipse (0.6cm and 0.2cm);
\fill (32.5,-4) ellipse (0.6cm and 0.2cm);
\draw[thick] (32.3,0)..controls +(0,-2) and +(0,+2) .. (29.7,-4); \draw[thick] (32.7,0) -- (32.7,-4);
\fill[white] (30.8,-2.1) circle (0.4);
\draw[thick] (29.3,-0.5) -- (29.3,-4);\draw[thick] (29.7,-0.5)..controls +(0,-2) and +(0,+2) .. (32.3,-4);

\node at (34.5,0) {$+q$};
\fill (36.5,4) ellipse (0.6cm and 0.2cm);
\fill (39.5,4) ellipse (0.6cm and 0.2cm);
\draw[thick] (39.3,4)..controls +(0,-2) and +(0,+2) .. (36.7,0); \draw[thick] (39.7,4) -- (39.7,0.5);
\fill[white] (38.2,2.1) circle (0.4);
\draw[thick] (36.3,4) -- (36.3,0);\draw[thick] (36.7,4)..controls +(0,-2) and +(0,+2) .. (39.3,0.5);  

\draw[thick] (39.7,0.5)..controls +(0,-0.2) and +(0,+0.2) .. (39.3,-0.5);
\fill[white] (39.5,0) circle (0.2);
\draw[thick] (39.3,0.5)..controls +(0,-0.2) and +(0,+0.2) .. (39.7,-0.5);

\fill (36.5,-4) ellipse (0.6cm and 0.2cm);
\fill (39.5,-4) ellipse (0.6cm and 0.2cm);
\draw[thick] (39.3,-0.5)..controls +(0,-2) and +(0,+2) .. (36.7,-4); \draw[thick] (39.7,-0.5) -- (39.7,-4);
\fill[white] (38,-2) circle (0.4);
\draw[thick] (36.3,0) -- (36.3,-4);\draw[thick] (36.7,0)..controls +(0,-2) and +(0,+2) .. (39.3,-4);  

\node at (41.5,0) {$+q^2$};
\fill (43.5,4) ellipse (0.6cm and 0.2cm);
\fill (46.5,4) ellipse (0.6cm and 0.2cm);
\draw[thick] (46.3,4)..controls +(0,-2) and +(0,+2) .. (43.7,0.5); \draw[thick] (46.7,4) -- (46.7,0.5);
\fill[white] (45.1,2) circle (0.4);
\draw[thick] (43.3,4) -- (43.3,0.5);\draw[thick] (43.7,4)..controls +(0,-2) and +(0,+2) .. (46.3,0.5);  

\draw[thick] (43.7,0.5)..controls +(0,-0.2) and +(0,+0.2) .. (43.3,-0.5);
\fill[white] (43.5,0) circle (0.2);
\draw[thick] (43.3,0.5)..controls +(0,-0.2) and +(0,+0.2) .. (43.7,-0.5);

\draw[thick] (46.7,0.5)..controls +(0,-0.2) and +(0,+0.2) .. (46.3,-0.5);
\fill[white] (46.5,0) circle (0.2);
\draw[thick] (46.3,0.5)..controls +(0,-0.2) and +(0,+0.2) .. (46.7,-0.5);

\fill (43.5,-4) ellipse (0.6cm and 0.2cm);
\fill (46.5,-4) ellipse (0.6cm and 0.2cm);
\draw[thick] (46.3,-0.5)..controls +(0,-2) and +(0,+2) .. (43.7,-4); \draw[thick] (46.7,-0.5) -- (46.7,-4);
\fill[white] (44.9,-2) circle (0.4);
\draw[thick] (43.3,-0.5) -- (43.3,-4);\draw[thick] (43.7,-0.5)..controls +(0,-2) and +(0,+2) .. (46.3,-4);  

\node at (48.5,0) {$\Bigr)$};

\node at (17.5,-8) {$=\displaystyle\frac{1}{(1+q^2)^2}\Bigl($};
\fill (22,-6) ellipse (0.6cm and 0.2cm);\fill (22,-10) ellipse (0.6cm and 0.2cm);
\draw[thick] (21.8,-6) -- (21.8,-10);\draw[thick] (22.2,-6) -- (22.2,-10);
\fill (25,-6) ellipse (0.6cm and 0.2cm);\fill (25,-10) ellipse (0.6cm and 0.2cm);
\draw[thick] (24.8,-6) -- (24.8,-10);\draw[thick] (25.2,-6) -- (25.2,-10);

\node at (31,-8) {$+(q-q^{-1}+2q^3)$};

\fill (37,-6) ellipse (0.6cm and 0.2cm);\fill (37,-10) ellipse (0.6cm and 0.2cm);
\fill (40,-6) ellipse (0.6cm and 0.2cm);\fill (40,-10) ellipse (0.6cm and 0.2cm);
\draw[thick] (39.8,-6)..controls +(0,-2) and +(0,+2) .. (37.2,-10);  \draw[thick] (40.2,-6) -- (40.2,-10);
\fill[white] (38.5,-8) circle (0.4);
\draw[thick] (36.8,-6) -- (36.8,-10);\draw[thick] (37.2,-6)..controls +(0,-2) and +(0,+2) .. (39.8,-10);  

\node at (42,-8) {$+q^2$};

\fill (44,-6) ellipse (0.6cm and 0.2cm);\fill (44,-10) ellipse (0.6cm and 0.2cm);
\fill (47,-6) ellipse (0.6cm and 0.2cm);\fill (47,-10) ellipse (0.6cm and 0.2cm);
\draw[thick] (46.8,-6)..controls +(0,-2) and +(0,+2) .. (43.8,-10);  
\draw[thick] (47.2,-6)..controls +(0,-2) and +(0,+2) .. (44.2,-10);  
\fill[white] (45.5,-8) circle (0.4);
\draw[thick] (43.8,-6)..controls +(0,-2) and +(0,+2) .. (46.8,-10);  
\draw[thick] (44.2,-6)..controls +(0,-2) and +(0,+2) .. (47.2,-10);  

\node at (49,-8) {$\Bigr)$};

\end{tikzpicture}
\end{center}

In the right hand side of the first line we proceed as follows. For the first diagram, we apply the Hecke relation and this results to the identity term and the term with coefficient $(q-q^{-1})$. For the second diagram, we take the strand connecting the first top ellipse to the first bottom ellipse, and we move it on the left of the diagram, then we apply the idempotent relations; this results with one term with coefficient $q^3$. We proceed similarly for the third diagram. We do almost nothing except moving the strands in the fourth diagram.

$\bullet$ Let $\bk=(2,2,\dots)$ the infinite sequence of $2$'s. We define below some elements of $H_{\bk,n}(q)$:
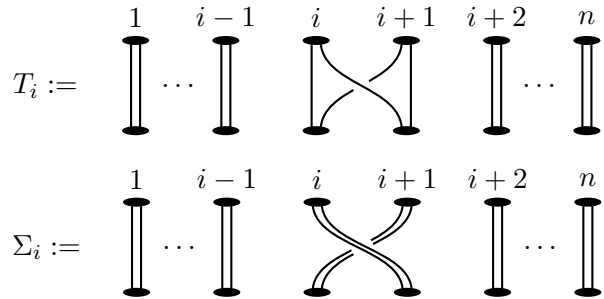
\begin{center}
 \begin{tikzpicture}[scale=0.3]
\node at (-2,0) {$T_i:=$};
\node at (2,3) {$1$};\fill (2,2) ellipse (0.6cm and 0.2cm);\fill (2,-2) ellipse (0.6cm and 0.2cm);
\draw[thick] (1.8,2) -- (1.8,-2);\draw[thick] (2.2,2) -- (2.2,-2);
\node at (4,0) {$\dots$};
\draw[thick] (5.8,2) -- (5.8,-2);\draw[thick] (6.2,2) -- (6.2,-2);
\node at (6,3) {$i-1$};\fill (6,2) ellipse (0.6cm and 0.2cm);\fill (6,-2) ellipse (0.6cm and 0.2cm);
\node at (10,3) {$i$};\fill (10,2) ellipse (0.6cm and 0.2cm);\fill (10,-2) ellipse (0.6cm and 0.2cm);
\node at (14,3) {$i+1$};\fill (14,2) ellipse (0.6cm and 0.2cm);\fill (14,-2) ellipse (0.6cm and 0.2cm);

\draw[thick] (9.8,2) -- (9.8,-2);
\draw[thick] (13.8,2)..controls +(0,-2) and +(0,+2) .. (10.2,-2);
\fill[white] (12,0) circle (0.4);
\draw[thick] (10.2,2)..controls +(0,-2) and +(0,+2) .. (13.8,-2);
\draw[thick] (14.2,2) -- (14.2,-2);

\draw[thick] (17.8,2) -- (17.8,-2);\draw[thick] (18.2,2) -- (18.2,-2);
\node at (18,3) {$i+2$};\fill (18,2) ellipse (0.6cm and 0.2cm);\fill (18,-2) ellipse (0.6cm and 0.2cm);
\node at (20,0) {$\dots$};
\draw[thick] (21.8,2) -- (21.8,-2);\draw[thick] (22.2,2) -- (22.2,-2);\fill (22,2) ellipse (0.6cm and 0.2cm);\fill (22,-2) ellipse (0.6cm and 0.2cm);
\node at (22,3) {$n$};
\end{tikzpicture}
\end{center}

\begin{center}
 \begin{tikzpicture}[scale=0.3]
\node at (-2,0) {$\Sigma_i:=$};
\node at (2,3) {$1$};\fill (2,2) ellipse (0.6cm and 0.2cm);\fill (2,-2) ellipse (0.6cm and 0.2cm);
\draw[thick] (1.8,2) -- (1.8,-2);\draw[thick] (2.2,2) -- (2.2,-2);
\node at (4,0) {$\dots$};
\draw[thick] (5.8,2) -- (5.8,-2);\draw[thick] (6.2,2) -- (6.2,-2);
\node at (6,3) {$i-1$};\fill (6,2) ellipse (0.6cm and 0.2cm);\fill (6,-2) ellipse (0.6cm and 0.2cm);
\node at (10,3) {$i$};\fill (10,2) ellipse (0.6cm and 0.2cm);\fill (10,-2) ellipse (0.6cm and 0.2cm);
\node at (14,3) {$i+1$};\fill (14,2) ellipse (0.6cm and 0.2cm);\fill (14,-2) ellipse (0.6cm and 0.2cm);

\draw[thick] (13.8,2)..controls +(0,-2) and +(0,+2) .. (9.8,-2);
\draw[thick] (14.2,2)..controls +(0,-2) and +(0,+2) .. (10.2,-2);
\fill[white] (12,0) circle (0.5);
\draw[thick] (10.2,2)..controls +(0,-2) and +(0,+2) .. (14.2,-2);
\draw[thick] (9.8,2)..controls +(0,-2) and +(0,+2) .. (13.8,-2);

\draw[thick] (17.8,2) -- (17.8,-2);\draw[thick] (18.2,2) -- (18.2,-2);
\node at (18,3) {$i+2$};\fill (18,2) ellipse (0.6cm and 0.2cm);\fill (18,-2) ellipse (0.6cm and 0.2cm);
\node at (20,0) {$\dots$};
\draw[thick] (21.8,2) -- (21.8,-2);\draw[thick] (22.2,2) -- (22.2,-2);\fill (22,2) ellipse (0.6cm and 0.2cm);\fill (22,-2) ellipse (0.6cm and 0.2cm);
\node at (22,3) {$n$};
\end{tikzpicture}
\end{center}
The elements $\Sigma_i$ satisfy the braid relations : $\Sigma_i\Sigma_{i+1}\Sigma_i=\Sigma_{i+1}\Sigma_i\Sigma_{i+1}$. We leave to the reader the verification of the following relations
\[T_i\Sigma_i=\Sigma_iT_i=(q-q^{-1})\Sigma_i+q^2 T_i\ \ \ \ \text{and}\ \ \ \Sigma_i^2=(q-q^{-1})^2\{2\}_q \Sigma_i+(q-q^{-1})\{3\}_q T_i+1\ ,\]
which implies easily the following characteristic equation of order 3 for $\Sigma_i$:
\[\Sigma_i^3-(q^4-1+q^{-2})\Sigma_i^2-(q^4-q^2+q^{-2})\Sigma_i+q^2=(\Sigma_i-q^4)(\Sigma_i+1)(\Sigma_i-q^{-2})=0\ .\]

$\bullet$ More generally, let $\bk=(k,k,\dots)$ be the infinite sequence of $k$'s for an integer $k\geq 1$. We define $\Sigma_i\in H_{\bk,n}(q)$ similarly to the previous example, namely $\Sigma_i$ is the fused braid for which all strands starting from ellipse $i$ pass over the strands starting from ellipse $i+1$. All other strands are vertical. 

Then the braid relation $\Sigma_i\Sigma_{i+1}\Sigma_i=\Sigma_{i+1}\Sigma_i\Sigma_{i+1}$ is satisfied. So inside the algebra $H_{\bk,n}(q)$ the elements $\Sigma_1,\dots,\Sigma_{n-1}$ generate a subalgebra which is a quotient of the algebra of the braid group. Note that these elements do not generate the whole algebra $H_{\bk,n}(q)$ when $n>2$ and $k>1$.

We will see in the next subsection that the algebra $H_{\bk,n}(q)$ is finite-dimensional. In particular, for $n=2$, the algebra $H_{\bk,2}(q)$ is of dimension $k+1$. Moreover, we will see later that the algebra $H_{\bk,n}(q)$ acts on a tensor product of representations of $U_q(gl_N)$. In this representation, the elements $\Sigma_i$ correspond to the $R$-matrix, and from this, one can obtain that $\Sigma_i$ satisfies a certain characteristic equation (see for example \cite{LZ}). We skip the details and we just indicate that one obtains that the minimal characteristic equation for $\Sigma_i$ in $H_{\bk,n}(q)$ is:
\[\prod_{l=0}^k\Bigl(\Sigma_i-(-1)^{k+l} q^{-k+l(l+1)}\Bigr)=0\ .\]

\subsection{Standard basis}

Inside the Hecke algebra $H_{k_1+\dots+k_n}(q)$, we consider the parabolic subalgebra isomorphic to
\[H_{k_1}\otimes\dots\otimes H_{k_n}\ .\]
The first factor $H_{k_1}$ above is obtained as the subalgebra generated by $\si_1,\dots,\si_{k_1-1}$, the second factor $H_{k_2}$ is obtained as the subalgebra generated by $\si_{k_1+1},\dots,\si_{k_1+k_2-1}$, and so on. We define the following element of $H_{k_1+\dots+k_n}(q)$:
\begin{equation}\label{def-Pkn}
P_{\bk,n}:=P_{k_1}\otimes\dots\otimes P_{k_n}=P_{(1,\,...\,,k_1)}P_{(k_1+1,\,...\,,k_1+k_2)}\dots P_{(k_1+\dots+k_{n-1}+1,\,...\,,k_1+\dots+k_n)}\ .
\end{equation}
In the second expression, $P_{(1,\,...\,,k_1)}$ is the $q$-symmetriser on the generators $\si_1,\dots,\si_{k_1-1}$, $P_{(k_1+1,\,...\,,k_1+k_2)}$ is the $q$-symmetriser on the generators $\si_{k_1+1},\dots,\si_{k_1+k_2-1}$, and so on up to $P_{(k_1+\dots+k_{n-1}+1,\,...\,,k_1+\dots+k_n)}$ which is the $q$-symmetriser on the generators $\si_{k_1+\dots+k_{n-1}+1},\dots,\si_{k_1+\dots+k_{n}-1}$. 

The element $P_{\bk,n}$ is clearly an idempotent of $H_{k_1+\dots+k_n}(q)$. Then the subset $P_{\bk,n}H_{k_1+\dots+k_n}(q)P_{\bk,n}=\{P_{\bk,n}xP_{\bk,n}\ |\ x\in H_{k_1+\dots+k_n}(q)\}$ of $H_{k_1+\dots+k_n}(q)$ is an algebra with unit $P_{\bk,n}$. This algebra has the following canonical basis
\[\{P_{\bk,n}\si_wP_{\bk,n}\ |\ w\in \mS^{\bk,n}\backslash\mS_{k_1+\dots+k_n}/\mS^{\bk,n}\}\ ,\]
where by $w\in \mS^{\bk,n}\backslash\mS_{k_1+\dots+k_n}/\mS^{\bk,n}$, we mean $w$ is a distinguished double coset representative, as explained in Subsection \ref{subsec-doublecoset}. This follows from the flatness of the deformation from the symmetric group to the Hecke algebra which ensures that the dimension is the number of double cosets. And moreover, we have that for any $\pi\in \mS_{k_1+\dots+k_n}$, there exists an $w$ as above and $\pi_1,\pi_2\in \mS^{\bk,n}$ such that $\pi=\pi_1w\pi_2$ and $\ell(\pi)=\ell(\pi_1)+\ell(w)+\ell(w_2)$. Therefore we have $\si_{\pi}=\si_{w_1}\si_{w}\si_{\pi_2}$ and then $P_{\bk,n}\si_{\pi}P_{\bk,n}$ is proportional to $\si_w$ using Relations (\ref{rel-sym}). This shows that the above set is indeed a spanning set of the correct cardinality.

\begin{prop}\label{prop-PHP}
 The algebra $H_{\bk,n}(q)$ is isomorphic to $P_{\bk,n}H_{k_1+\dots+k_n}(q)P_{\bk,n}$.
\end{prop}
\begin{proof}
Starting from a fused braid, replace the $a$-th top ellipse by the idempotent $P_{k_a}$ of $H_{k_a}(q)$. More precisely, we mean that for any basis element $\si_w$ of $H_{k_a}(q)$, we can replace the $a$-th top ellipse by $\si_w$ (in its diagrammatic form) by plugging the $k_a$-th strand starting from the ellipse to the $k_a$ bottom dots of $\si_w$. So we do this for any $\si_w$ in $H_{k_a}(q)$, and we make the sum with the same coefficients as in the definition (\ref{def-P}) of $P_{k_a}$. Similarly, we can replace the $a$-th bottom ellipse by the idempotent $P_{k_a}$ of $H_{k_a}(q)$. Doing this for any top and bottom ellipse, we obtain an element of $H_{k_1+\dots+k_n}(q)$, which is more precisely in $P_{\bk,n}H_{k_1+\dots+k_n}(q)P_{\bk,n}$ by construction. We claim that this procedure produces a well-defined map
\[\phi\ :\ H_{\bk,n}(q)\to P_{\bk,n}H_{k_1+\dots+k_n}(q)P_{\bk,n}\]
which is moreover a morphism of algebras.

First we have to check that the relations defining the vector space $H_{\bk,n}(q)$ in Definition \ref{vector-fused-braids} are preserved. The local Hecke relation is obviously preserved since it is also true by definition in $H_{k_1+\dots+k_n}(q)$. The fact that the idempotent relations are also preserved follows immediately from the properties (\ref{rel-sym}) of the $q$-symmetriser.

Then the multiplication in $H_{\bk,n}(q)$ is also preserved by the map $\phi$ since by construction  it corresponds to the multiplication in $P_{\bk,n}H_{k_1+\dots+k_n}(q)P_{\bk,n}$.

\medskip
Now let $\pi\in \mS_{k_1+\dots+k_n}$ and consider $\sigma_{\pi}\in H_{k_1+\dots+k_n}(q)$ in its diagrammatic representation. Then, gluing the $k_1$ first dots into an ellipse, and the $k_2$ next dots into an ellipse and so on, we obtain the diagram of a fused braid. We denote $[\si_{\pi}]$ the corresponding element of $H_{\bk,n}(q)$. By the map $\phi$, the element $[\si_{\pi}]$ is sent to $\si_{\pi}$, which shows the surjectivity of $\phi$.

It remains to show that $H_{\bk,n}(q)$ is of dimension less than the dimension of $P_{\bk,n}H_{k_1+\dots+k_n}(q)P_{\bk,n}$. To see this, we show that the set 
\[\{[\si_w]\ |\ w\in \mS^{\bk,n}\backslash\mS_{k_1+\dots+k_n}/\mS^{\bk,n}\}\]
is a spanning set of  $H_{\bk,n}(q)$. First we have that $\{[\si_{\pi}]\ |\ \pi\in \mS_{k_1+\dots+k_n}\}$ is a spanning set since we can apply the local Hecke relation inside $H_{\bk,n}(q)$. Moreover, for any $\pi\in \mS_{k_1+\dots+k_n}$, there exists $w\in \mS^{\bk,n}\backslash\mS_{k_1+\dots+k_n}/\mS^{\bk,n}$ and $\pi_1,\pi_2\in \mS^{\bk,n}$ such that $\pi=\pi_1w\pi_2$ and $\ell(\pi)=\ell(\pi_1)+\ell(w)+\ell(w_2)$. Therefore we have $\si_{\pi}=\si_{w_1}\si_{w}\si_{\pi_2}$. Thus from the idempotent relations in Definition \ref{vector-fused-braids}, it is clear that $[\si_{\pi}]$ is proportional to $[\si_{w}]$, concluding the proof of the proposition.
\end{proof}

\paragraph{Chain property.} By convention, we consider that $H_{\bk,0}(q)=\CC$. Note also that $H_{\bk,1}(q)=\CC$. The algebras $\{H_{\bk,n}(q)\}_{n\geq0}$ form a chain of algebras:
\begin{equation}\label{chain-PHP}
\CC=H_{\bk,0}(q)\subset H_{\bk,1}(q)\subset H_{\bk,2}(q)\subset\dots\dots \subset H_{\bk,n}(q)\subset H_{\bk,n+1}(q)\subset \dots\dots\ ,
\end{equation}
generalising the chain (\ref{chain-H}). Here, using Proposition \ref{prop-PHP}, the natural inclusions of algebras are given by 
$$H_{\bk,n}(q)\ni P_{\bk,n}xP_{\bk,n}\mapsto P_{\bk,n+1}xP_{\bk,n+1}\in H_{\bk,n+1}(q)\ .$$
Indeed an element $x\in H_{k_1+\dots+k_n}(q)$ can be seen as the element $x\otimes 1$ in $H_{k_1+\dots+k_n}(q)\otimes H_{k_{n+1}}(q)\subset H_{k_1+\dots+k_n+k_{n+1}}(q)$. Then, in $ H_{k_1+\dots+k_n+k_{n+1}}(q)$ the element $P_{\bk,n+1}xP_{\bk,n+1}$  can be seen as $P_{\bk,n}xP_{\bk,n}\otimes P_{k_{n+1}}$. The subalgebra consisting of these elements is clearly isomorphic to $P_{\bk,n}H_{k_1+\dots+k_n}(q)P_{\bk,n}=H_{\bk,n}(q)$.

Diagrammatically, the chain property reads naturally as follows: the algebra $H_{\bk,n}(q)$ is embedded in $H_{\bk,n+1}(q)$ by considering only the fused braids in $H_{\bk,n+1}(q)$ such that all strands starting from top ellipse $n+1$ go vertically and without crossings to the bottom ellipse $n+1$.

We note that there is more generally a notion of parabolic subalgebras for $H_{\bk,n}(q)$ with natural diagrammatic and algebraic formulations that we leave to the reader.

\paragraph{Standard basis and deformation.} Recall the definition of elements $[\si_w]\in H_{\bk,n}(q)$ for any $w\in\mS_{k_1+\dots+k_n}$ that we used during the proof of Proposition \ref{prop-PHP}. Diagrammatically, $[\si_w]$ was obtained by putting the diagram of $\si_w$ between the two lines of ellipses, or equivalently, by gluing some dots into ellipses in the diagram of $\si_w$. For simplicity, we denote $F_w$ the corresponding element of $H_{\bk,n}(q)$. We proved that
\begin{equation}\label{basis-PHP}\{F_w\ |\ w\in \mS^{\bk,n}\backslash\mS_{k_1+\dots+k_n}/\mS^{\bk,n}\}
\end{equation}
is a basis of $H_{\bk,n}(q)$ ($w$ runs over a set of distinguished double coset representatives as in Subsection \ref{subsec-doublecoset}).

Finally we have that the two definitions (the one in the preceding section and the one in this section if $q=1$) of the algebra $H_{\bk,n}(1)$ result in the same algebras so that there is no ambiguity. The fused Hecke algebra $H_{\bk,n}(q)$ is a flat deformation of the algebra $H_{\bk,n}(1)$ of fused permutations.

We note that in the diagrammatic point of view, the basis (\ref{basis-PHP}) of $H_{\bk,n}(q)$ is naturally obtained from the basis (\ref{basis-clas}) of $H_{\bk,n}(1)$ in the following way. Namely, take a basis element $f_w$ of $H_{\bk,n}(1)$, where $w\in \mS^{\bk,n}\backslash\mS_{k_1+\dots+k_n}/\mS^{\bk,n}$, and consider a diagram representing it with a minimal number of intersections: namely, the $k_a$ edges leaving the $a$-th top ellipse do not intersect and similarly for the $k_a$ edges arriving at the $a$-th bottom ellipse (this for $a=1,\dots,n$). Then, promote each intersection as a positive crossing. The resulting element can be seen as a fused braid of $H_{\bk,n}(q)$. This element is $F_w$.

\section{Classical Schur--Weyl duality}\label{sec-SW}

Our goal in this section is to recall the well-known Schur--Weyl duality between $U_q(gl_N)$ and the Hecke algebra. The classical results presented in this section will serve as a model for the formulation of the subsequent results. Moreover they will also be used in most of our proofs. We refer to \cite[Chap. 9]{GW} for the classical Schur--Weyl duality for $U(gl_N)$ and to \cite[\S 8.6]{KS} for its standard analogue for $U_q(gl_N)$ and the relevant definitions.

Recall that for a representation $\rho\ :\ U_q(gl_N)\to\text{End}(V)$ on a vector space $V$, what we call the centraliser of the action of $U_q(gl_N)$ on $V$ is the following subalgebra of $\text{End}(V)$:
\[\text{End}_{U_q(gl_N)}(V)=\{\phi\in\text{End}(V)\ |\ \phi\circ\rho(x)=\rho(x)\circ\phi\,,\forall x\in U_q(gl_N)\}.\]

\paragraph{Representations of $U_q(gl_N)$.} Let $N>1$ and denote $gl_N$ the Lie algebra of $N\times N$ matrices and $sl_N$ the Lie subalgebra of traceless matrices. Let $U_q(gl_N)$ (respectively, $U_q(sl_N)$) denote the standard deformation of the universal enveloping algebra $U(gl_N)$ of $gl_N$ (respectively, of $sl_N$). If $q=1$, $U_1(gl_N)=U(gl_N)$ and $U_1(sl_N)=U(sl_N)$.

Let $\lambda$ be a partition with a number of non-zero parts less or equal to $N$, that is, let $\lambda=(\lambda_1,\dots,\lambda_N)$ with $\lambda_1\geq\dots\geq\lambda_N\geq0$. We denote $L^N_{\lambda}$ the irreducible highest-weight representation of $U_q(gl_N)$ with highest-weight corresponding to $\lambda$.

\begin{rem}
We formulate all the results in this paper for $U_q(gl_N)$. Nevertheless, one can replace everywhere $gl_N$ by $sl_N$ without further modification. We recall for convenience of the reader the following facts. The restriction of $L^N_{\lambda}$ to $U_q(sl_N)$ remains an irreducible highest-weight representation. We still denote its restriction $L^N_{\lambda}$. Then, for a partition $\lambda=(\lambda_1,\dots,\lambda_N)$, the associated highest-weight of $sl_N$ corresponds to the consecutive differences $\lambda_i-\lambda_{i+1}$, for $i=1,\dots,N-1$, and the associated representation $L^N_{\lambda}$ of $U_q(sl_N)$ only depends on these differences. In particular, if $\lambda_N>0$ the representation $L^N_{\lambda}$ is equivalent to the representation $L^N_{\lambda'}$, where $\lambda'=(\lambda_1-1,\dots,\lambda_N-1)$. In terms of Young diagrams (their definition will be recalled later), this corresponds to the possibility of deleting columns of length $N$. This is valid only for $U_q(sl_N)$ and therefore we will never apply it in this paper.
\end{rem}

\begin{exa}
\begin{itemize}
\item Let $\lambda$ be the partition $(1)$. The representation $L^N_{(1)}$ is the vector representation of $U_q(gl_N)$ of dimension $N$ (the deformation of the natural representation of $U(gl_N)$ by $N\times N$ matrices).
\item If $\lambda=(k)$ is the partition consisting of a single integer $k>0$ then $L^N_{(k)}$ is the deformation of the $k$-th symmetric tensor power of the vector representation $L^N_{(1)}$.
\item If $N=2$ and $\lambda=(k)$, then $L^2_{(k)}$ is an irreducible representation of $U_q(gl_2)$ of dimension $k+1$. Its restriction to $U_q(sl_2)$ is the unique irreducible representation of $U_q(sl_2)$ of dimension $k+1$. It is the so-called ``spin $\frac{k}{2}$'' representation of $U_q(sl_2)$.
\end{itemize}
\end{exa}

\subsection{Classical Schur--Weyl duality for $U_q(gl_N)$}

Let $n>0$. We consider the representation of $U_q(gl_N)$ on the tensor power $(L^N_{(1)})^{\otimes n}$. The assertions concerning this representation of $U_q(gl_N)$ and its centraliser, commonly referred as the (quantum) Schur--Weyl duality, can be summarised as follows.

First, for any finite-dimensional vector space $V$, there is a representation of the Hecke algebra $H_n(q)$ on $V^{\otimes n}$. To give the action, let $(e_1,\dots,e_d)$ be a basis of $V$ and define the linear operator $\check{R}\in\text{End}(V\otimes V)$ by
\begin{equation}\label{hR-Hn-s}
\check{R}(e_i\otimes e_j):=\left\{\begin{array}{ll}
q\,e_i\otimes e_j\ \  & \text{if $i=j$,}\\[0.8em]
e_j\otimes e_i+(q-q^{-1})\,e_i\otimes e_j & \text{if $i<j$,}\\[0.4em]
e_j\otimes e_i & \text{if $i>j$.}
\end{array}\right.\ \ \ \text{where $i,j=1,\dots,d$.}
\end{equation}
Then the following map from the set of generators of $H_n(q)$ to $\text{End}(V^{\otimes n})$:
\begin{equation}\label{SW-Hn}
\sigma_i \mapsto \check{R}_{i,i+1}\ \quad\text{for $i=1,\dots,n-1$\,,}
\end{equation}
extends to an algebra homomorphism (where $\check{R}_{i,i+1}$ is the operator $\text{Id}_{V^{\otimes i-1}}\otimes\check{R}\otimes\text{Id}_{V^{\otimes n-i-1}}$).

We apply this construction for $V=L^N_{(1)}$ carrying a realisation of the vector representation of $U_q(gl_N)$ (see for example \cite{PdA} for the conventions and normalisations we are using here about $U_q(gl_N)$, its action on $L^N_{(1)}$ and its coproduct).
\begin{defi}\label{def-oH}
We denote by $\oH^N_n(q)$ the image of the map $H_n(q)\to\text{End}\bigl((L^N_{(1)})^{\otimes n}\bigr)$ and by $I_n^N$ its kernel.
\end{defi}
In the theorem below, $\lambda\vdash n$ means that $\lambda$ is a partition of $n$, that is, $\lambda=(\lambda_1,\dots,\lambda_l)$ is a family of  integers such that $\lambda_1\geq\lambda_2\geq\dots\geq\lambda_l\geq 0$ and $\lambda_1+\dots+\lambda_l=n$. The number $l(\lambda)$ is the number of non-zero parts of $\lambda$.
\begin{thm}[Schur--Weyl duality]\label{theo-SW}
The algebra $\oH^N_n(q)$ coincides with the centraliser of the action of $U_q(gl_N)$ on $(L^N_{(1)})^{\otimes n}$. Moreover, as a $U_q(gl_N)\otimes H_n(q)$-module, the space $(L^N_{(1)})^{\otimes n}$ decomposes as follows:
\begin{equation}\label{SW}
(L^N_{(1)})^{\otimes n}=\bigoplus_{\begin{array}{c}
\\[-1.6em]
\scriptstyle{\lambda\vdash n} \\[-0.4em]
\scriptstyle{l(\lambda)\leq N}
\end{array}} L^N_{\lambda}\otimes V_{\lambda}\ ,
\end{equation}
where the $V_{\lambda}$'s are pairwise non-equivalent irreducible representations of $H_n(q)$.
\end{thm}
A construction of the representations $V_{\lambda}$ of $H_n(q)$ for $\lambda\vdash n$ will be recalled in the next subsection.

In particular, the Schur--Weyl duality provides a description of the decomposition of the tensor product $(L^N_{(1)})^{\otimes n}$ into a direct sum of irreducible $U_q(gl_N)$-modules:
\[(L^N_{(1)})^{\otimes n}=\bigoplus_{\begin{array}{c}
\\[-1.6em]
\scriptstyle{\lambda\vdash n} \\[-0.4em]
\scriptstyle{l(\lambda)\leq N}
\end{array}} (L^N_{\lambda})^{\oplus \dim(V_{\lambda})}\ ,
\]
or, in words, the multiplicity of the $U_q(gl_N)$-module $L^N_{\lambda}$ in the decomposition is equal to the dimension of the $H_n(q)$-module $V_{\lambda}$ if $\lambda\vdash n$ and is $0$ otherwise.

Concerning a description of the centraliser of the action of $U_q(gl_N)$ on $(L^N_{(1)})^{\otimes n}$, the preceding theorem expresses it as an homomorphic image of the Hecke algebra $H_n(q)$. So a more precise understanding of this centraliser will be obtained through a description of the kernel $I_n^N$ of the action of $H_n(q)$. In order to give such a description (and also for further use), we need to recall the semisimple representation theory of Hecke algebras.

\subsection{Representation theory of the chain of Hecke algebras $H_n(q)$}

Let $n\geq 0$. We recall the well-known representation theory of the Hecke algebras $H_n(q)$ in a form that we will use to study the representation theory of the algebras $H_{\bk,n}(q)$. We refer to, \emph{e.g.}, \cite[\S 10]{GP} or \cite[\S 5]{KT}.

\paragraph{Combinatorics of partitions.} Let $\lambda\vdash n$ be a partition of $n$, that is, $\lambda=(\lambda_1,\dots,\lambda_l)$ is a family of  integers such that $\lambda_1\geq\lambda_2\geq\dots\geq\lambda_l\geq 0$ and $\lambda_1+\dots+\lambda_l=n$. We say that $\lambda$ is a partition {\em of size} $n$ and set $|\lambda|:=n$. The number $l(\lambda)$ of non-zero parts is called the {\em length} of $\lambda$. By definition, the empty partition $\lambda=\emptyset$ is a partition of size 0 and of length 0.

 A pair $(x,y)\in\mathbb{Z}^2$ is called a {\em node}. For a node $\theta=(x,y)$, the classical content of $\theta$ is denoted by $\cc(\theta)$ and is defined by $\cc(\theta):=y-x$\,. The $q$-content, or simply the content, of the node $\theta=(x,y)$ is $\qc(\theta):=q^{2\cc(\theta)}=q^{2(y-x)}$. 
 
The Young diagram of $\lambda=(\lambda_1,\dots,\lambda_l)$ is the set of nodes $(x,y)$ such that $x\in\{1,\dots,l\}$ and $y\in\{1,\dots,\lambda_x\}$. The Young diagram of $\lambda$ will be seen as a left-justified array of $l$ rows such that the $j$-th row contains $\lambda_j$ nodes for all $j=1,\dots,l$ (a node will be pictured by an empty box). We number the rows from top to bottom. 

A skew partition consists of two partitions $\mu,\lambda$ such that, as sets of nodes, $\mu\subset\lambda$. It is commonly denoted by $\lambda/\mu$. The Young diagram of $\lambda/\mu$ consists of the sets of nodes which are in $\lambda$ and not in $\mu$. The size $|\lambda/\mu|$ of a skew partition $\lambda/\mu$ is $|\lambda|-|\mu|$ and it is the number of nodes in the Young diagram. As an example,
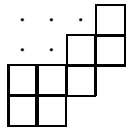
\[\begin{array}{cccc}
 \cdot & \hspace{-0.35cm}\cdot & \hspace{-0.35cm}\cdot & \hspace{-0.35cm}\fbox{\phantom{\scriptsize{$2$}}} \\[-0.2em]
\cdot & \hspace{-0.35cm}\cdot & \hspace{-0.35cm}\fbox{\phantom{\scriptsize{$2$}}} & \hspace{-0.35cm}\fbox{\phantom{\scriptsize{$2$}}}\\[-0.2em]
\fbox{\phantom{\scriptsize{$2$}}} &\hspace{-0.35cm}\fbox{\phantom{\scriptsize{$2$}}} &\hspace{-0.35cm}\fbox{\phantom{\scriptsize{$2$}}} &\\[-0.2em]
\fbox{\phantom{\scriptsize{$2$}}} &\hspace{-0.35cm}\fbox{\phantom{\scriptsize{$2$}}} & &
\end{array}\]
is the Young diagram corresponding to $\lambda=(4,4,3,2)$ and $\mu=(3,2)$. We will make no distinction between a skew partition and its Young diagram; this will not cause any confusion here. We will say that $(x,y)$ is a node of $\lambda/\mu$, or $(x,y)\in\lambda/\mu$, if $(x,y)$ is a node in the Young diagram of $\lambda/\mu$.

A Young tableau of shape $\lambda/\mu$ is a map from the set of nodes of $\lambda/\mu$ to $\mathbb{Z}_{\geq1}$. It is represented by filling the nodes of the $\lambda/\mu$ by numbers in $\mathbb{Z}_{\geq1}$. The size of a Young tableau is the size of its shape.

A Young tableau of size $n$ is called standard if the map from the set of nodes is a bijection with $\{1,\dots,n\}$ and if moreover the numbers are strictly ascending along rows and down columns of the Young diagram. We set
\[\STab(\lambda/\mu):=\{\text{standard Young tableaux of shape $\lambda/\mu$}\}\ .\]

Let $\bT$ be a standard Young tableau of size $n$ and let $\theta_i$ be the node of $\bT$ with number $i$. We set $\cc_i(\bT):=\cc(\theta_i)$ and $\qc_i(\bT):=\qc(\theta_i)$ for $i=1,\dots,n$. Here are two examples of standard Young tableaux of size $4$ with their sequence of contents (the shape of the second one is a skew partition):
\[\bT=\begin{array}{ccc}
\fbox{\scriptsize{$1$}} & \hspace{-0.35cm}\fbox{\scriptsize{$2$}} & \hspace{-0.35cm}\fbox{\scriptsize{$4$}} \\[-0.2em]
\fbox{\scriptsize{$3$}} & &
\end{array}\ : \qquad \qc_1(\bT)=1\,,\ \ \ \qc_2(\bT)=q^2\,,\ \ \ \qc_3(\bT)=q^{-2}\,,\ \ \ \qc_4(\bT)=q^4\,.\]
\[\bT=\begin{array}{ccc}
\cdot & \hspace{-0.35cm}\fbox{\scriptsize{$2$}} & \hspace{-0.35cm}\fbox{\scriptsize{$4$}} \\[-0.2em]
\fbox{\scriptsize{$1$}} & \hspace{-0.35cm}\fbox{\scriptsize{$3$}} &
\end{array}\ : \qquad \qc_1(\bT)=q^{-2}\,,\ \ \ \qc_2(\bT)=q^2\,,\ \ \ \qc_3(\bT)=1\,,\ \ \ \qc_4(\bT)=q^4\,.\]

Let $\bT$ be a standard Young tableau of size $n$ and $i\in\{1,\dots,n-1\}$. We denote $s_i(\bT)$ the Young tableau obtained from $\bT$ by exchanging $i$ and $i+1$ (note that $s_i(\bT)$ is not necessarily standard). We record here a fundamental fact concerning standard Young tableaux of a given shape. This is surely well-known to experts, however we provide a sketch of a proof for convenience of the reader (see also \cite[Proposition 6.5]{PdA2} for a slightly different formulation).
\begin{lem}\label{lem-tableaux}
Let $\lambda/\mu$ be a skew partition. Any two standard Young tableaux of shape $\lambda/\mu$ can be obtained one from another by a sequence of elementary transpositions $s_{i_1},\dots,s_{i_k}$ such that the tableaux remain standard after every step.
\end{lem}
\begin{proof}
Let $\bT$ and $\bT'$ be two standard Young tableaux of the same shape. We reason by induction on the size $n$ and we assume that $n\geq 2$  (if $n=1$, there is nothing to prove). Denote $\theta$ (respectively, $\theta'$) the node of $\bT$ (respectively, $\bT'$) containing $n$.

If $\theta=\theta'$ then we can use the induction hypothesis to obtain a desired sequence of elementary transpositions acting only on $\{1,\dots,n-1\}$ and transforming $\bT$ into $\bT'$.

If $\theta\neq\theta'$ we note that since $\bT$ and $\bT'$ are standard, then $\theta$ must be the rightmost node of its line and the lowest node of its column, and similarly for $\theta'$. In particular, $\theta$ and $\theta'$ can not be in adjacent diagonals. It follows that the following procedure in three steps is possible:

$\bullet$ First, we use the induction hypothesis to obtain a sequence of elementary transpositions acting only on $\{1,\dots,n-1\}$ and transforming $\bT'$ into a standard Young tableau with $n-1$ in the node $\theta$.

$\bullet$ Then we apply the elementary transposition $(n-1,n)$ to obtain a standard Young tableau with $n$ in $\theta$.

$\bullet$ Finally, we use again the induction hypothesis to obtain a sequence of elementary transpositions acting only on $\{1,\dots,n-1\}$, which transforms this standard Young tableau into $\bT$.
\end{proof}

\paragraph{Representations of $H_n(q)$.} We refer to \cite{Ho,Ra}. Let $n\geq 1$ and $\lambda/\mu$ be a skew partition of size $n$. Let $V_{\lambda/\mu}$ be a $\CC$-vector space with a basis $\{v_{\bT}\}_{\bT\in\STab(\lambda/\mu)}$ indexed by the standard Young tableaux of shape $\lambda/\mu$. The following formula for the generators $\sigma_1,\dots,\sigma_{n-1}$ defines a representation of the Hecke algebra $H_n(q)$ on the space $V_{\lambda/\mu}$ (this can be checked with a straightforward verification of the defining formulas of $H_n(q)$):
\begin{equation}\label{rep-si}
\si_{i}(v_{\bT})=\frac{(q-q^{-1})\qc_{i+1}(\bT)}{\qc_{i+1}(\bT)-\qc_{i}(\bT)}\,v_{\bT}+\frac{q\,\qc_{i+1}(\bT)-q^{-1}\qc_{i}(\bT)}{\qc_{i+1}(\bT)-\qc_{i}(\bT)}\,v_{s_i(\bT)}\ ,\ \ \ \quad\text{for $i\in\{1,\dots,n-1\}$.}
\end{equation}
where $s_i(\bT)$ is the Young tableau obtained from $\bT$ by exchanging $i$ and $i+1$ and where we define $v_{\bT'}:=0$ for any non-standard Young tableau $\bT'$. Note that $\qc_{i}(\bT)\neq\qc_{i+1}(\bT)$ for any $\bT\in\STab(\lambda/\mu)$. The basis $\{v_{\bT}\}_{\bT\in\STab(\lambda/\mu)}$ is sometimes called the seminormal basis of the representation $V_{\lambda/\mu}$.

If we denote by $d_{i,j}(\bT):=\cc_{j}(\bT)-\cc_i(\bT)$ the axial distance between the nodes with number $j$ and $i$ in the Young tableau $\bT$, then Formula (\ref{rep-si}) can be written in terms of $q$-numbers defined in (\ref{quantum-numbers}) as follows:
\begin{equation}\label{rep-si2}
\si_{i}(v_{\bT})=\frac{q^{d_{i,i+1}(\bT)}}{[d_{i,i+1}(\bT)]_q}\,v_{\bT}+\frac{[d_{i,i+1}(\bT)+1]_q}{[d_{i,i+1}(\bT)]_q}\,v_{s_i(\bT)}\ ,\ \ \ \quad\text{for $i\in\{1,\dots,n-1\}$.}
\end{equation}
These formulas can be specialised for $q^2=1$ and provide representations of the symmetric group $\mS_n$.

\paragraph{Irreducible representations $V_{\lambda}$.} If $\lambda$ is a partition of $n$ then the representation $V_{\lambda}$ is irreducible. Moreover, as $\lambda$ runs over the set of partitions of $n$, the representations $V_{\lambda}$ are pairwise non-isomorphic and exhaust the set of irreducible representations of $H_n(q)$ up to isomorphism. The irreducible representations $V_{\lambda}$ are the ones appearing in the Schur--Weyl duality in Theorem \ref{theo-SW}.

We note that if $\lambda=(n)$ (a line of $n$ boxes) then the representation $V_{(n)}$ is the one-dimensional representation given by $\si_i\mapsto q$ for $i=1,\dots,n-1$. Its associated minimal central idempotent (in the sense described in Appendix \ref{app-idem}) is what we called the $q$-symmetriser and denoted $P_n$.

We note also that if $\lambda=(1,\dots,1)$ (a column of $n$ boxes) then the representation $V_{(1,\dots,1)}$ is the one-dimensional representation given by $\si_i\mapsto -q^{-1}$ for $i=1,\dots,n-1$. Its associated minimal central idempotent will be given later and is called the $q$-antisymmetriser.

\paragraph{Branching rules.} From Formulas (\ref{rep-si}), the branching rules for the chain of algebras $\{H_n(q)\}_{n\geq0}$ are almost immediate to obtain. Indeed let $\lambda\vdash n$ and let $\mu\vdash n-1$ be a subpartition of $\lambda$, that is, $\mu$ is obtained from $\lambda$ by deleting one node, say $\theta$. Then consider the subset $I_{\mu}$ of $\STab(\lambda)$ consisting of standard Young tableaux of shape $\lambda$ containing $n$ in the node $\theta$. The set $I_{\mu}$ is clearly in bijection with $\STab(\mu)$ and it is immediate from (\ref{rep-si}) to see that the subspace of $V_{\lambda}$ generated by $I_{\mu}$ is a representation of $H_{n-1}(q)$ isomorphic to $V_{\mu}$. We conclude then that if $\lambda$ is a partition of $n$ then the restriction of $V_{\lambda}$ to $H_{n-1}(q)$ decomposes into irreducible as follows:
\begin{equation}\label{BR}
\text{Res}_{H_{n-1}(q)}(V_{\lambda})\cong \bigoplus_{\begin{array}{c}
\\[-1.6em]
\scriptstyle{\mu\,\vdash n-1} \\[-0.4em]
\scriptstyle{\mu\subset \lambda}
\end{array}} V_{\mu}\ .
\end{equation}
The first levels of the Bratteli diagram (see Appendix \ref{app-bra} for definitions) for the chain $\{H_n(q)\}_{n\geq0}$ are shown in Appendix \ref{app1}.

\subsection{Identification of the kernel in the classical Schur--Weyl duality}

Let $N\geq2$. We recall that the centraliser of the action of $U_q(gl_N)$ on $(L^N_{(1)})^{\otimes n}$ was denoted $\oH^N_n(q)$ and that it is obtained as the quotient of $H_n(q)$ by a certain ideal $I_n^N$, see Definition \ref{def-oH}.

\subsubsection{Representation-theoretic description of the kernel}

In Appendix, we fix the terminology and notations for quotients of semisimple algebras and quotients of Bratteli diagrams. 

The result in the following proposition is a well-known consequence of the Schur--Weyl duality. The first item follows quite immediately from the statements in Theorem \ref{theo-SW}. For the second item, one has to notice that a partition $\lambda$ is such that $l(\lambda)>N$ if and only if its Young diagram contains a column of size $N+1$, if and only if there is a path in the Bratteli diagram from $(1,1,\dots,1)\vdash N+1$ to $\lambda$. We indicate that the chain structure on the quotients is explicited in a more general situation later in Section \ref{sec-quot-N}.
\begin{prop}\label{prop-quot-rep}
\begin{enumerate}
\item For $n\geq 0$, the ideal $I_n^N$ of $H_n(q)$ corresponds to the following subset of partitions:
\[\{\ \lambda\vdash n\ |\ l(\lambda)>N\ \}\ ;\]
\item The Bratteli diagram of the chain $\{\oH^N_n(q)\}_{n\geq0}$ is the quotient of the Bratteli diagram of the chain $\{H_n(q)\}_{n\geq0}$ generated by:
\[S_{min}:=\{(1,1,\dots,1)\vdash N+1\}\ \ \text{(the one-column partition of size $N+1$)}\ .\]
\end{enumerate}
\end{prop}
We note the following immediate consequence:
\begin{equation}\label{comp-SW}
\oH^N_n(q)=H_n(q)\ \ \ \ \Leftrightarrow\ \ \ \ N\geq n\ .
\end{equation}

\subsubsection{Algebraic description of the kernel}

Proposition \ref{prop-quot-rep} describes the ideal $I_n^N$ of $H_n(q)$ on the side of representations. From this, one can easily obtain an algebraic description of this ideal $I_n^N$ and thus of the quotient $\oH^N_n(q)$.

\paragraph{The $q$-antisymmetriser.} Let $m>0$ and define the following element of $H_m(q)$:
\begin{equation}\label{def-P'}
P'_m:=\frac{\sum_{w\in\mS_m}(-q^{-1})^{\ell(w)}\si_w}{\sum_{w\in\mS_m}q^{-2\ell(w)}}=\frac{\sum_{w\in\mS_m}(-q^{-1})^{\ell(w)}\si_w}{q^{-n(n-1)}\{m\}_q!}=q^{n(n-1)/2}\frac{\sum_{w\in\mS_m}(-q^{-1})^{\ell(w)}\si_w}{[m]_q!}\ .
\end{equation}
It is well-known that the element $P'_m$ is a minimal central idempotent in $H_m(q)$ projecting on the one-dimensional representation given by $\sigma_i\mapsto -q^{-1}$ for $i=1,\dots,m-1$. In particular, we have
\[(P'_m)^2=P'_m\ \ \ \ \ \text{and}\ \ \ \ \ \ \si_iP'_m=P'_m\si_i=-q^{-1} P'_m\,,\ i=1,\dots,m-1.\]
If $q=1$, the projector $P'_m$ is the antisymmetriser of $\CC\mS_m$ projecting on the sign representation of $\mS_m$.

From Proposition \ref{prop-quot-rep} and using the general facts presented in Appendix \ref{app-quot-Brat}, we obtain:
\begin{prop}
If $n\leq N$ then the centraliser $\oH^N_n(q)$ coincides with $H_n(q)$. If $n>N$ then the centraliser $\oH^N_n(q)$ is the quotient of $H_n(q)$ over the relation:
\[\sum_{w\in\mS_{N+1}}(-q^{-1})^{\ell(w)}\si_w=0\ ,\]
\end{prop}
In words, to obtain $\oH^N_n(q)$, we cancel the $q$-antisymmetriser on $N+1$ letters. Note that $\si_w$ when $w\in\mS_{N+1}$ is a word in the generators $\si_1,\dots,\si_N$ and as such can be seen as an element of $H_n(q)$ if $n>N$ (this is a convenient slight abuse of notation).

\begin{exa}[Temperley--Lieb algebra]\label{exa-TL}
Let $N=2$ and $n\geq3$. What the preceding proposition states is that the centraliser $\oH^2_n(q)$ is the quotient of the Hecke algebra over the relation:
\[\si_1\si_2\si_1-q(\si_1\si_2+\si_2\si_1)+q^2(\si_1+\si_2)-q^{3}=0\ .\]
This centraliser $\oH^2_n(q)$ is called the Temperley--Lieb algebra. It is easy to see using the braid relations that conjugating this relation by $\si_1\si_2\si_3$, one obtains the similar relation with indices 2,3, and hence that the similar relation with indices $i,i+1$ for all $i=1,\dots,n-2$ is implied. Then setting $\tau_i:=\si_i-q$, one recovers the other standard presentation  of the Temperley--Lieb algebra:
\[\tau_i^2=-(q+q^{-1})\tau_i\,,\ \ \ \ \tau_i\tau_{i+1}\tau_i=\tau_i\,,\ \ \ \ \tau_{i+1}\tau_{i}\tau_{i+1}=\tau_{i+1}\ \ \ \text{and}\ \ \ \tau_i\tau_j=\tau_j\tau_i\ \text{if $|i-j|>1$.}\]
\end{exa}

\section{Generalisation of Schur--Weyl duality for symmetric powers}\label{sec-cent}

Let $\bk=(k_1,k_2,...)\in \mathbb{Z}_{\geq0}^{\infty}$ be an infinite sequence of non-negative integers, and let $n\in\mathbb{Z}_{\geq0}$\,. Let also $N>1$.

We recall that for $k\geq 1$ we have denoted $L_{(k)}$ the irreducible representation of $U_q(gl_N)$ corresponding to the partition $(k)$ (the one-line partition with $k$ boxes). We consider here the representation of $U_q(gl_N)$ on the following tensor product:
\[L^N_{(k_1)}\otimes\dots\otimes L^N_{(k_n)}\ .\]
By convention, for $n=0$, this is the trivial representation.

We will use the definition and properties of minimal central idempotents (here, for the Hecke algebras) recalled in general in Appendix \ref{app-idem}.

\subsection{Schur--Weyl duality for $L^N_{(k_1)}\otimes\dots\otimes L^N_{(k_n)}$}

Recall that, for any $k\geq1$, the element $P_k$ of $H_k(q)$ defined by (\ref{def-P}) is the primitive central idempotent associated to the one-dimensional representation $V_{(k)}$. As such, $P_k$ is equal to the identity in the representation $V_{(k)}$ of $H_k(q)$ and $0$ in every other irreducible representations of $H_k(q)$.

Note that the dimension of $V_{(k)}$ is equal to 1. Therefore, from Formula (\ref{SW}) expressing the Schur--Weyl duality, we deduce immediately that we have the following decomposition of $U_q(gl_N)$-modules
\begin{equation}\label{dec-small}
(L_{(1)}^N)^{\otimes k}\cong L^N_{(k)}\oplus U\ ,
\end{equation}
with the property that
\begin{equation}\label{action-Pkn-small}{P_k}_{|_{L^N_{(k)}}}=\text{Id}_{L^N_{(k)}}\ \ \ \ \text{and}\ \ \ \ P_k(U)=0\ .
\end{equation}

Then using (\ref{dec-small}) and expanding, we have the following decomposition of $U_q(gl_N)$-modules
\begin{equation}\label{dec}
\begin{array}{rcl} 
(L_{(1)}^N)^{\otimes k_1+\dots+k_n} & = & (L^N_{(1)})^{\otimes k_1}\otimes \dots\otimes (L^N_{(1)})^{\otimes k_n}\ \ \\[0.8em]
 & = & L^N_{(k_1)}\otimes \dots\otimes L^N_{(k_n)}\ \ \bigoplus\ \ U'\ .
 \end{array}
 \end{equation}
Now, from its definition (\ref{def-Pkn}), the element $P_{\bk,n}$ of $H_{k_1+\dots+k_n}(q)$ acts as $P_{k_1}\otimes\dots\otimes P_{k_n}$ on $(L_{(1)}^N)^{\otimes k_1+\dots+k_n}$ (this can be seen directly from the definition (\ref{SW-Hn}) of the action of $H_{k_1+\dots+k_n}(q)$ on this space). Therefore, using (\ref{action-Pkn-small}) for $k=k_1$,\dots,$k=k_n$, we obtain that the above decomposition (\ref{dec}) is such that:
\begin{equation}\label{action-Pkn}
{P_{\bk,n}}_{|_{L^N_{(k_1)}\otimes \dots\otimes L^N_{(k_n)}}}=\text{Id}_{L^N_{(k_1)}\otimes \dots\otimes L^N_{(k_n)}}\ \ \ \ \text{and}\ \ \ \ P_{\bk,n}(U')=0\ .
\end{equation}
In other words, the action of $P_{\bk,n}$ on $(L_{(1)}^N)^{\otimes k_1+\dots+k_n}$ is the projection on $L^N_{(k_1)}\otimes \dots\otimes L^N_{(k_n)}$ associated to the decomposition (\ref{dec}). As recalled in Appendix \ref{app-ss}, we have therefore naturally an action of $P_{\bk,n}H_{k_1+\dots+k_n}(q)P_{\bk,n}$ on $L^N_{(k_1)}\otimes \dots\otimes L^N_{(k_n)}$ given by restriction
\begin{equation}\label{act-PHP}
\begin{array}{rcl}
P_{\bk,n}H_{k_1+\dots+k_n}(q)P_{\bk,n} & \to & \text{End}\bigl(L^N_{(k_1)}\otimes \dots\otimes L^N_{(k_n)}\bigr) \\[0.5em]
P_{\bk,n}xP_{\bk,n} & \mapsto & {P_{\bk,n}xP_{\bk,n}}_{|_{L^N_{(k_1)}\otimes \dots\otimes L^N_{(k_n)}}}
\end{array}\ .
\end{equation}

Now we are ready to state the analogue of (the first part) of the Schur--Weyl duality.
\begin{thm}\label{theo-SWk}
There is a representation of the algebra $H_{\bk,n}(q)$ on $L^N_{(k_1)}\otimes \dots\otimes L^N_{(k_n)}$ and the image of the map $H_{\bk,n}(q)\to\text{End}\bigl(L^N_{(k_1)}\otimes \dots\otimes L^N_{(k_n)}\bigr)$ coincides with the centraliser of the action of $U_q(gl_N)$.
\end{thm}
\begin{proof}
Let $\phi\in U_q(gl_N)$ and let $x\in H_{k_1+\dots+k_n}(q)$. From the classical Schur--Weyl duality, we know that the actions of $P_{\bk,n}xP_{\bk,n}$ and of $\phi$ on $(L_{(1)}^N)^{\otimes k_1+\dots+k_n}$ commute. Moreover, both actions leave the subspace $L^N_{(k_1)}\otimes \dots\otimes L^N_{(k_n)}$ invariant, so their restrictions to this subspace commute as well. So we have that the image of the map (\ref{act-PHP}) is included in the centraliser of the action of $U_q(gl_N)$.

For the reverse inclusion, let $\overline{y}\in \text{End}\bigl(L^N_{(k_1)}\otimes \dots\otimes L^N_{(k_n)}\bigr)$ which commutes with the action of $U_q(gl_N)$. Extend it by $0$ on the subspace $U'$ appearing in the decomposition (\ref{dec}) to get an element $\overline{Y}$ acting on $(L_{(1)}^N)^{\otimes k_1+\dots+k_n}$. Obviously this element $\overline{Y}$ commutes with the action of $U_q(gl_N)$. Therefore, from the Schur--Weyl duality, we have an element $Y\in H_{k_1+\dots+k_n}(q)$ such that $\overline{Y}$ is the action of $Y$ on $(L_{(1)}^N)^{\otimes k_1+\dots+k_n}$. Then, we have that $\overline{y}$ is the action of $P_{\bk,n}YP_{\bk,n}$ on $L^N_{(k_1)}\otimes \dots\otimes L^N_{(k_n)}$ and thus we conclude that $\overline{y}$ belongs to the image of the map (\ref{act-PHP}).

We just proved that the image of the map (\ref{act-PHP}) coincides with the centraliser of the action of $U_q(gl_N)$. Finally, the isomorphism of $H_{\bk,n}(q)$ with $P_{\bk,n}H_{k_1+\dots+k_n}(q)P_{\bk,n}$ obtained in Proposition \ref{prop-PHP} provides by composition with (\ref{act-PHP}) the required action of $H_{\bk,n}(q)$ on $L^N_{(k_1)}\otimes \dots\otimes L^N_{(k_n)}$ with the desired properties.
\end{proof}

In view of the preceding result, we make the following definition.
\begin{defi}\label{def-oPHP}
We denote by $\oH^N_{\bk,n}(q)$ the image of the representation $H_{\bk,n}(q)\to\text{End}\bigl(L^N_{(k_1)}\otimes \dots\otimes L^N_{(k_n)}\bigr)$ and by $I_{\bk,n}^N$ its kernel.
\end{defi}
From the preceding Theorem, we have that $\oH^N_{\bk,n}(q)$ is the centraliser of the action of $U_q(gl_N)$ on $L^N_{(k_1)}\otimes \dots\otimes L^N_{(k_n)}$ and that it is isomorphic to the quotient of the algebra $H_{\bk,n}(q)$ by the ideal $I_{\bk,n}^N$.

\paragraph{First description of the ideals $I_{\bk,n}^N$.} We collect a preliminary result on the ideals $I_{\bk,n}^N$ for later use. Recall that $I_{k_1+\dots+k_n}^N$ is the ideal of $H_{k_1+\dots+k_n}(q)$ corresponding to the representation of $H_{k_1+\dots+k_n}(q)$ on $(L_{(1)}^N)^{\otimes k_1+\dots+k_n}$ in the classical Schur--Weyl duality (see Definition \ref{def-oH}).
\begin{prop}\label{prop-SWk-ideal}
Under the isomorphism between $H_{\bk,n}(q)$ and $P_{\bk,n}H_{k_1+\dots+k_n}(q)P_{\bk,n}$, the ideal $I_{\bk,n}^N$ corresponds to $P_{\bk,n}I_{k_1+\dots+k_n}^NP_{\bk,n}$.
\end{prop}
\begin{proof}
Let $x\in H_{k_1+\dots+k_n}(q)$. We have that $P_{\bk,n}xP_{\bk,n}$ belongs to the kernel of the map (\ref{act-PHP}) if and only if the restriction on $L^N_{(k_1)}\otimes \dots\otimes L^N_{(k_n)}$ of the action of $P_{\bk,n}xP_{\bk,n}$ is $0$. Since $P_{\bk,n}(U')=0$, where $U'$ is the complementary subspace appearing in (\ref{dec}), this is equivalent to saying that the action of $P_{\bk,n}xP_{\bk,n}$ on the whole space $(L_{(1)}^N)^{\otimes k_1+\dots+k_n}$ is 0. In other words, it is equivalent to the fact that $P_{\bk,n}xP_{\bk,n}$ belongs to the ideal $I_{k_1+\dots+k_n}^N$ of $H_{k_1+\dots+k_n}(q)$. 

Finally $P_{\bk,n}xP_{\bk,n}\in I_{k_1+\dots+k_n}^N$ implies $P_{\bk,n}xP_{\bk,n}\in P_{\bk,n}I_{k_1+\dots+k_n}^NP_{\bk,n}$ since $P_{\bk,n}^2=P_{\bk,n}$, and reciprocally, $P_{\bk,n}xP_{\bk,n}\in P_{\bk,n}I_{k_1+\dots+k_n}^NP_{\bk,n}$ implies $P_{\bk,n}xP_{\bk,n}\in I_{k_1+\dots+k_n}^N$ since $I_{k_1+\dots+k_n}^N$ is an ideal.

This concludes the proof since $I_{\bk,n}^N$ is defined as the kernel of the composition of the isomorphism  between $H_{\bk,n}(q)\cong P_{\bk,n}H_{k_1+\dots+k_n}(q)P_{\bk,n}$ and the map (\ref{act-PHP}).
\end{proof}

\begin{rem}\label{rem-quot-sub}
The centraliser $\oH_{\bk,n}^N(q)$ is described here as a quotient of the algebra $H_{\bk,n}(q)$. We recall that this algebra is isomorphic to $P_{\bk,n}H_{k_1+\dots+k_n}(q)P_{\bk,n}$. Applying this isomorphism, it is easy to check that by construction the centraliser $\oH_{\bk,n}^N(q)$ is sent to $P_{\bk,n}\oH_{k_1+\dots+k_n}^N(q)P_{k,n}$ (where we still denote by $P_{\bk,n}$ the image of $P_{\bk,n}$ in $\oH_{k_1+\dots+k_n}^N(q)$).

This remark illustrates the fact that, for any $N$, we can follow two different paths to reach the centralisers $\oH^N_{\bk,n}(q)$ from $H_{k_1+\dots+k_n}(q)$, graphically depicted as follows
\[\begin{array}{ccc}
H_{k_1+\dots+k_n}(q) & \longrightarrow & \oH_{k_1+\dots+k_n}^N(q) \\[0.5em]
\downarrow &  &\downarrow \\[0.5em]
H_{\bk,n}(q) & \longrightarrow & \oH^N_{\bk,n}(q)
\end{array}
\]
On one hand, we can first take a quotient (depending on $N$) to obtain $\oH_{k_1+\dots+k_n}^N(q)$  and then consider inside each $\oH_{k_1+\dots+k_n}^N(q)$ the subalgebras obtained by multiplying by the idempotent on both sides. On the other hand, we can first consider the subalgebra $H_{\bk,n}(q)$ of $H_{k_1+\dots+k_n}(q)$ obtained by multiplying by the idempotent on both sides, and then take a quotient depending on $N$. Our approach in this paper is to follow the second road and we emphasize the role of the algebra $H_{\bk,n}(q)$ which does not depend on $N$.
\end{rem}

\section{Representation theory and Bratteli diagram of $\{H_{\bk,n}(q)\}_{n\geq 0}$}\label{sec-rep}

We recall that due to the restrictions on $q$, the algebra $H_{\bk,n}(q)$ is semisimple (see Proposition \ref{prop-PHP} and Appendix \ref{app-ss}). In this section, we provide a description of the representation theory of the algebra $H_{\bk,n}(q)$, which relies only on the well-known representation theory of the Hecke algebras. The knowledge of the representation theory of $H_{\bk,n}(q)$ will allow us to give a first description of its quotients $\oH_{\bk,n}(q)$. This description will be entirely representation-theoretic. We will use it to study in the last section a description of the quotients in the diagrammatic presentation of $\oH_{\bk,n}(q)$.

\subsection{Induction step}

Let $k\geq 1$. Let $\lambda/\mu$ be a skew partition of size $k$ and $V_{\lambda/\mu}$ the corresponding representation of $H_k(q)$ constructed in Section \ref{sec-SW}. The next proposition identifies the subspace $P_k(V_{\lambda/\mu})$ in terms of the seminormal basis $\{v_{\bT}\}_{\bT\in\STab(\lambda/\mu)}$, where we recall that $P_k$ is the $q$-symmetriser of $H_k$. This result will serve later as the induction step to understand the irreducible representations of $H_{\bk,n}(q)$ from the ones of $H_{\bk,n-1}(q)$. 

We note the remarkable fact that the image of $P_k$ can be expressed in terms of the basis $\{v_{\bT}\}$ with no $q$ appearing.
\begin{prop}\label{prop-rep}
We have:
\[P_k(V_{\lambda/\mu})=\left\{\begin{array}{ll}
\CC\bigl(\sum\limits_{\bT\in \STab(\lambda/\mu)}v_{\bT}\bigr) & \text{if $\lambda/\mu$ contains at most one box in each column,}\\[0.5em]
0 & \text{otherwise.}
\end{array}\right.\]
\end{prop}
\begin{proof}
The defining formula and the fundamental properties of $P_k$ are in (\ref{def-P}) and (\ref{rel-sym}). We note first that we have:
\begin{equation}\label{Pk-si}
P_k=P_k\frac{1+q\si_i}{1+q^2}\ \ \ \ \ \text{for any $i=1,\dots,k-1$.}
\end{equation}

Let $\bT\in\STab(\lambda/\mu)$ and let $i\in\{1,\dots,k-1\}$. We denote $d:=d_{i,i+1}(\bT)=\cc_{i+1}(\bT)-\cc_i(\bT)$. From Formula (\ref{rep-si}), we calculate the action of $\si_i$ on the subspace of $V_{\lambda/\mu}$ generated by $v_{\bT}$ and $v_{s_i(\bT)}$, and deduce the action of $1+q\sigma_i$. There are three cases:
\begin{itemize}
\item[\textbf{(a)}] If $i+1$ is in the same column as $i$ just below it in $\bT$ then $v_{s_i(\bT)}=0$. The action is given by:
\[\si_i(v_{\bT})=-q^{-1}v_{\bT}\ \ \ \ \Rightarrow\ \ \ \ (1+q\si_i)(v_{\bT})=0\,.\]

\item[\textbf{(b)}] If $i+1$ is in the same line as $i$ just to its right in $\bT$ then $v_{s_i(\bT)}=0$. The action is given by:
\[\si_i(v_{\bT})=q v_{\bT}\ \ \ \ \Rightarrow\ \ \ \ (1+q\si_i)(v_{\bT})=(1+q^2)v_{\bT}\,.\]

\item[\textbf{(c)}] If $i$ and $i+1$ are neither in the same line nor in the same column in $\bT$ then $s_i(\bT)\in\STab(\lambda/\mu)$. The action in the basis $\{v_{\bT},v_{s_i(\bT)}\}$ is given by:
\[\si_i=\left(\begin{array}{cc}
\displaystyle \frac{q-q^{-1}}{1-q^{2d}} & \displaystyle\frac{q-q^{-2d-1}}{1-q^{-2d}} \\[1em]
\displaystyle \frac{q-q^{2d-1}}{1-q^{2d}} & \displaystyle\frac{q-q^{-1}}{1-q^{-2d}}
\end{array}\right)\ \ \ \ \Rightarrow\ \ \ \ 1+q\si_i=\left(\begin{array}{cc}
\displaystyle \frac{q^2-q^{2d}}{1-q^{2d}} & \displaystyle\frac{q^2-q^{-2d}}{1-q^{-2d}} \\[1em]
\displaystyle \frac{q^2-q^{2d}}{1-q^{2d}} & \displaystyle\frac{q^2-q^{-2d}}{1-q^{-2d}}
\end{array}\right)\,.\]
We find that the image of $(1+q\si_i)$ is included in the line $\CC(v_{\bT}+v_{s_i(\bT)})$ and moreover an easy calculation shows that $\si_i(v_{\bT}+v_{s_i(\bT)})=q(v_{\bT}+v_{s_i(\bT)})$.
\end{itemize}

We will combine (\ref{Pk-si}) with these elementary calculations to prove the proposition. First assume that the skew partition $\lambda/\mu$ contains two boxes in the same column. As $\lambda/\mu$ is a skew partition, we have two adjacent boxes in this column. Then for a standard Young tableau $\bT$ of shape $\lambda/\mu$, these two boxes must contain the numbers $i$ and $i+1$ for some $i\in\{1,\dots,k-1\}$ (with $i+1$ below $i$). Therefore, from Case \textbf{(a)} above, we have:
\[P_k(v_{\bT})=P_k\frac{1+q\si_i}{1+q^2}(v_{\bT})=0\ .\]
This shows that $P_k(V_{\lambda/\mu})=0$ in this case.

\medskip
Now, let $\{\alpha_{\bT}\}_{\bT\in\STab(\lambda/\mu)}$ be arbitrary complex numbers. The sums below are always indexed by the set $\STab(\lambda/\mu)$. We have:
\[P_k\bigl(\sum \alpha_{\bT}v_{\bT}\bigr)=P_k\frac{1+q\si_i}{1+q^2}\bigl(\sum \alpha_{\bT}v_{\bT}\bigr)\ \ \ \ \ \text{for any $i=1,\dots,k-1$.}\]
From Case \textbf{(c)} above, we have that $P_k\bigl(\sum \alpha_{\bT}v_{\bT}\bigr)$ is proportional to $P_k\bigl(\sum \alpha'_{\bT}v_{\bT}\bigr)$, where the coefficients $\alpha'_{\bT}$ satisfy
\[\alpha'_{\bT}=\alpha'_{s_i(\bT)}\ \ \ \ \ \ \text{for every $\bT$ and $i$ such that $s_i(\bT)$ is standard.}\]
Using Lemma \ref{lem-tableaux}, we conclude that $P_k\bigl(\sum \alpha_{\bT}v_{\bT}\bigr)$ is proportional to $P_k\bigl(\sum v_{\bT}\bigr)$, namely
\begin{equation}\label{alpha}
P_k\bigl(\sum \alpha_{\bT}v_{\bT}\bigr)\in\CC P_k\bigl(\sum v_{\bT}\bigr)\ .
\end{equation}
Then if the skew partition $\lambda/\mu$ contains at most one box in each column, we have, from Cases \textbf{(b)} and \textbf{(c)} above, that:
\[\forall\,i=1,\dots,k-1\,,\ \ \ \si_i\bigl(\sum v_{\bT}\bigr)=q\bigl(\sum v_{\bT}\bigr)\ .\]
From the explicit formula for the idempotent $P_k$, this gives that $P_k\bigl(\sum v_{\bT}\bigr)=\sum v_{\bT}$. With (\ref{alpha}), this shows that $P_k(V_{\lambda/\mu})=\CC\bigl(\sum v_{\bT}\bigr)$ and concludes the proof of the proposition.
\end{proof}

\subsection{Complete description}

Let $\bk=(k_1,k_2,...)\in \mathbb{Z}_{\geq0}^{\infty}$ and $n\in\mathbb{Z}_{>0}$ as before. From the generalities recalled at the beginning of Section \ref{app-ss}, we need to understand the subspaces $P_{\bk,n}(V)$ for any irreducible representation $V$ of the algebra $H_{k_1+\dots+k_n}(q)$. The irreducible representations of $H_{k_1+\dots+k_n}(q)$ are the representations $V_{\lambda}$, where $\lambda$ runs over the partitions of size $k_1+\dots+k_n$.

\paragraph{Semistandard Young tableaux and Kostka numbers.} A sequence of non-negative integers $\nu=(\nu_1,\dots,\nu_l)$ such that $\nu_1+\dots+\nu_l=n$ is called a composition of $n$. We say that the size $|\nu|$ is equal to $n$. We make no difference between $\nu$ and the same sequence where we added some parts equal to $0$ at the end.

Let $\bT$ be an arbitrary Young tableau of size $n$. For $a\in\mathbb{Z}_{\geq1}$, let $\nu_a$ be the number of times the integer $a$ appears in the tableau $\bT$. The sequence $\nu=(\nu_1,\nu_2,\dots)$ forms a composition of $n$. We say that $\bT$ is a tableau of weight $\nu$.

A Young tableau is called semistandard if the numbers are weakly ascending along rows and strictly ascending down columns of the Young diagram. For a skew partition $\lambda/\mu$ of size $n$ and a composition $\nu$ of $n$, we set
\[\SSTab(\lambda/\mu,\nu):=\{\text{semistandard Young tableaux of shape $\lambda/\mu$ and of weight $\nu$}\}\ .\]
For example, a standard Young tableau is a semistandard Young tableau of weight $(1,\dots,1)$.

For saving space we use the following notation for the composition of $k_1+\dots+k_n$ obtained by restricting the sequence of integers $\bk$ to the first $n$ entries:
\[\bk_{\vert n}:=(k_1,\dots,k_n)\ .\]
Apart from the standard Young tableaux, we will mainly use the notion of semistandard Young tableaux for partitions of size $k_1+\dots+k_n$ and of weight $\bk_{\vert n}=(k_1,\dots,k_n)$. For example, if $n=4$ and $\bk_{\vert 4}=(2,2,2,2)$ then 
\[\SSTab\bigl((4,4),\bk_{\vert 4}\bigr):=\Bigl\{ \begin{array}{cccc}
\fbox{\scriptsize{$1$}} & \hspace{-0.35cm}\fbox{\scriptsize{$1$}} & \hspace{-0.35cm}\fbox{\scriptsize{$2$}} & \hspace{-0.35cm}\fbox{\scriptsize{$2$}} \\[-0.2em]
\fbox{\scriptsize{$3$}} & \hspace{-0.35cm}\fbox{\scriptsize{$3$}} & \hspace{-0.35cm}\fbox{\scriptsize{$4$}} & \hspace{-0.35cm}\fbox{\scriptsize{$4$}}
\end{array}\ ,\qquad \begin{array}{cccc}
\fbox{\scriptsize{$1$}} & \hspace{-0.35cm}\fbox{\scriptsize{$1$}} & \hspace{-0.35cm}\fbox{\scriptsize{$3$}} & \hspace{-0.35cm}\fbox{\scriptsize{$3$}} \\[-0.2em]
\fbox{\scriptsize{$2$}} & \hspace{-0.35cm}\fbox{\scriptsize{$2$}} & \hspace{-0.35cm}\fbox{\scriptsize{$4$}} & \hspace{-0.35cm}\fbox{\scriptsize{$4$}}
\end{array}\ ,\qquad\begin{array}{cccc}
\fbox{\scriptsize{$1$}} & \hspace{-0.35cm}\fbox{\scriptsize{$1$}} & \hspace{-0.35cm}\fbox{\scriptsize{$2$}} & \hspace{-0.35cm}\fbox{\scriptsize{$3$}} \\[-0.2em]
\fbox{\scriptsize{$2$}} & \hspace{-0.35cm}\fbox{\scriptsize{$3$}} & \hspace{-0.35cm}\fbox{\scriptsize{$4$}} & \hspace{-0.35cm}\fbox{\scriptsize{$4$}}
\end{array}\Bigr\}\]
are all the semistandard tableaux of shape $(4,4)$ and of weight $\bk_{\vert 4}$.

For a skew partition $\lambda/\mu$ of size $n$ and a composition $\nu$ of $n$, the number of semistandard Young tableaux of shape $\lambda/\mu$ and of weight $\nu$ is called a \emph{Kostka number} and is denoted:
\[K_{\lambda/\mu,\nu}:=|\SSTab(\lambda/\mu,\nu)|\ .\]
One of their main properties is that $K_{\lambda/\mu,\nu}$ does not depend on the ordering of the parts of $\nu$. A direct combinatorial proof of this fact can be found in \cite[Theorem 7.10.2]{St}.

\paragraph{Dominance order and Kostka numbers.} For a partition $\lambda=(\lambda_1,\dots,\lambda_l)$, we use the convention that $\lambda_{l+1}=\lambda_{l+2}=\dots=0$. For two partitions $\lambda,\mu$ of the same size, we denote 
$$\lambda\geq\mu\ \ \ \Longleftrightarrow\ \ \ \lambda_1+\dots+\lambda_i\geq \mu_1+\dots+\mu_i\,,\ \ \ \forall i.$$
This is the dominance ordering of partitions. 

We are going to use the following combinatorial construction several times in the rest of the paper. If $\mu$ is a composition, we denote $\mu^{\text{ord}}$ the partition obtained by reordering the parts of $\mu$ in decreasing order.
\begin{lem}\label{lem-comb}
Let $\lambda$ be a partition and $\mu=(\mu_1,\dots,\mu_n)$ a composition such that $|\lambda|=|\mu|$. Assume that $\lambda\geq \mu^{\text{ord}}$.
\begin{itemize}
\item[(i)] $\lambda$ has at least $\mu_n$ non-empty columns.
\item[(ii)] Fill the last box of the first $\mu_n$ columns of $\lambda$ with the letter $n$. Then, as long as there is a box containing $n$ with an empty box in the same line on its right, move the letter $n$ in the empty box (in other words, slide the boxes with $n$ to the right as far as possible):\\
There is a semistandard Young tableau $\bbT\in \SSTab(\lambda,\mu)$ with the letters $n$ in these positions.
\end{itemize}
\end{lem}
\begin{proof} Denote $\mu^{\text{ord}}=(\mu'_1,\dots,\mu'_n)$.

$(i)$ The condition $\lambda\geq\mu^{\text{ord}}$ implies in particular that $\lambda_1\geq\mu'_1$ which is the largest part of $\mu$. So we have in particular $\lambda_1\geq\mu_n$. Thus there are more than $\mu_n$ non-empty columns in $\lambda$.

$(ii)$ First note that since $\lambda_1+\dots+\lambda_n\geq \mu'_1+\dots+\mu'_n=|\mu|=|\lambda|$ then $\lambda$ has at most $n$ non-empty parts. We use induction on $n$ (the case $n=1$ is trivial since in this case, $\lambda$ is a single line of boxes). After placing the letters $n$ as indicated, the remaining empty boxes in $\lambda$ form a partition $\tilde{\lambda}$ which is given by:
\[\tilde{\lambda}_1=\lambda_1\,,\ \ \dots\ \ \,\ \tilde{\lambda}_{i-1}=\lambda_{i-1}\,,\ \tilde{\lambda}_i=\lambda_i-(\mu_n-\lambda_{i+1})\,,\ \tilde{\lambda}_{i+1}=\lambda_{i+2}\,,\ \dots\ ,\ \tilde{\lambda}_{n-1}=\lambda_n\,,\ \tilde{\lambda}_n=0\ ,\]
for some $i\in\{1,\dots,n\}$ (this number $i$ is such that $\lambda_1,\dots,\lambda_i\geq\mu_n$ and $\lambda_{i+1}<\mu_n$; if $i=n$, it is to be understood that the only modified part is $\tilde{\lambda}_n=\lambda_n-\mu_n$). Let $\nu=(\mu_1,\dots,\mu_{n-1})$. It remains to show that $\tilde{\lambda}\geq\nu^{\text{ord}}$. Indeed by induction we will have the existence of a semistandard Young tableau in $\SSTab(\tilde{\lambda},\nu)$, which together with the boxes containing $n$ will create an element of $\SSTab(\lambda,\mu)$.

The fact that $\tilde{\lambda}\geq\nu^{\text{ord}}$ is checked as follows. Denote $\nu^{\text{ord}}=(\nu'_1,\dots,\nu'_{n-1})$. First, we have (using $\lambda\geq\mu^{\text{ord}}$)
\[\tilde{\lambda}_1+\dots+\tilde{\lambda}_k=\lambda_1+\dots+\lambda_k\geq\mu'_1+\dots+\mu'_k\ \ \ \ \text{if $k=1,\dots,i-1$.}\]
This is greater or equal to $\nu'_1+\dots+\nu'_k$ since in fact $\nu'_1\leq \mu'_1$, $\dots$, $\nu'_k\leq\mu'_k$.

If $k=i,\dots,n-1$, we have (using $\lambda\geq\mu^{\text{ord}}$)
\[\tilde{\lambda}_1+\dots+\tilde{\lambda}_k=\lambda_1+\dots+\lambda_k+\lambda_{k+1}-\mu_n\geq\mu'_1+\dots+\mu'_k+\mu'_{k+1}-\mu_n\,.\]
If $\mu_n<\mu'_{k+1}$ then this is greater than $\mu'_1+\dots+\mu'_k$ which is in turn greater or equal to $\nu'_1+\dots+\nu'_k$ as above. Otherwise if $\mu_n$ is one of the integer $\mu'_1,\dots,\mu'_{k+1}$, say $\mu'_j$, then this is equal to $\mu'_1+\dots+\mu'_{j-1}+\mu'_{j+1}+\dots+\mu'_{k+1}$, which is in turn equal to $\nu'_1+\dots+\nu'_k$. Indeed we have here $\nu'_1=\mu'_1$, $\dots$, $\nu'_{j-1}=\mu'_{j-1}$, $\nu'_j=\mu'_{j+1}$, $\dots$, $\nu'_k=\mu'_{k+1}$.
\end{proof}

The preceding construction easily implies in particular the following known properties of Kostka numbers.
\begin{lem}\label{lem-Kos}
Let $\lambda$ a partition and $\mu$ a composition such that $|\lambda|=|\mu|$. We have
$$K_{\lambda,\mu}\neq 0\ \ \ \ \ \Longleftrightarrow\ \ \ \ \ \lambda\geq \mu^{\text{ord}}\ .$$
\end{lem}
\begin{proof}
From the recalled symmetry property of Kostka numbers, we have $K_{\lambda,\mu}=K_{\lambda,\mu^{\text{ord}}}$. So we can assume that the parts of $\mu$ are already ordered in decreasing order, namely, that we have $\mu^{\text{ord}}=\mu$.

First assume that $K_{\lambda,\mu}\neq 0$ so that there is $\bbT\in \SSTab(\lambda,\mu)$. Let $i\geq 1$. By semistandardness, in $\bbT$ the numbers $1,\dots,i$ all appear in the first $i$ lines of $\bbT$. So we must have $\lambda_1+\dots+\lambda_i\geq \mu_1+\dots+\mu_i$. This proves that $\lambda\geq \mu$.

Reciprocally, take $\lambda$ such that $\lambda\geq \mu$. From Lemma \ref{lem-comb}, there is an element in $\bbT\in \SSTab(\lambda,\mu)$, therefore $K_{\lambda,\mu}\neq 0$. 
\end{proof}

\subsubsection{Main result}

We are now ready to describe the representation theory of the chain of algebras
\[\CC=H_{\bk,0}(q)\subset H_{\bk,1}(q)\subset H_{\bk,2}(q)\subset\dots\dots \subset H_{\bk,n}(q)\subset H_{\bk,n+1}(q)\subset \dots\dots\ .\]
Let $\lambda\vdash k_1+\dots+k_n$ and recall that $V_{\lambda}$ is a vector space with basis indexed by $\STab(\lambda)$ carrying the irreducible representation of $H_{k_1+\dots+k_n}(q)$.

Let $\bT\in\STab(\lambda)$. We denote by $\overline{\bT}$ the Young tableau obtained from $\bT$ by the following map from $\{1,\dots,k_1+\dots+k_n\}$ to $\{1,\dots,n\}$:
\[1,\dots,k_1 \mapsto 1\,,\ \ \ \ \ k_1+1,\dots,k_1+k_2\mapsto 2\,,\ \ \ \ \dots\ \ \ \ k_1+\dots+k_{n-1}+1,\dots,k_1+\dots+k_n\mapsto n\,,\]
that is, we replace in $\bT$ the $k_1$ first integers by 1, the next $k_2$ ones by 2, and so on. We obtain this way a Young tableau $\overline{\bT}$ of weight $\bk_{\vert n}$ (note that $\overline{\bT}$ does not have to be semistandard, as shown in the example below).

Now, let $\bbT\in\SSTab(\lambda,\bk_{\vert n})$ a semistandard Young tableau of shape $\lambda$ and of weight $\bk_{\vert n}$. We define in $V_{\lambda}$ the following vector:
\begin{equation}\label{def-wT}
w_{\bbT}:=\sum_{\begin{array}{c}
\\[-1.6em]
\scriptstyle{\bT\in\STab(\lambda)} \\[-0.4em]
\scriptstyle{\overline{\bT}=\bbT}
\end{array}} v_{\bT}\ \in V_{\lambda}\ .
\end{equation}
\begin{exa}
Let $n=2$, $\bk_{\vert 2}=(2,2)$ and $\lambda=(3,1)$. There is only one semistandard Young tableau of shape $\lambda$ with weight $\bk_{\vert 2}$, and that is $\begin{array}{ccc}
\fbox{\scriptsize{$1$}} & \hspace{-0.35cm}\fbox{\scriptsize{$1$}} & \hspace{-0.35cm}\fbox{\scriptsize{$2$}} \\[-0.2em]
\fbox{\scriptsize{$2$}} & &
\end{array}$. We have then : $w_{\begin{array}{ccc}
\fbox{\scriptsize{$1$}} & \hspace{-0.35cm}\fbox{\scriptsize{$1$}} & \hspace{-0.35cm}\fbox{\scriptsize{$2$}} \\[-0.2em]
\fbox{\scriptsize{$2$}} & &
\end{array}}=v_{\begin{array}{ccc}
\fbox{\scriptsize{$1$}} & \hspace{-0.35cm}\fbox{\scriptsize{$2$}} & \hspace{-0.35cm}\fbox{\scriptsize{$3$}} \\[-0.2em]
\fbox{\scriptsize{$4$}} & &
\end{array}}+v_{\begin{array}{ccc}
\fbox{\scriptsize{$1$}} & \hspace{-0.35cm}\fbox{\scriptsize{$2$}} & \hspace{-0.35cm}\fbox{\scriptsize{$4$}} \\[-0.2em]
\fbox{\scriptsize{$3$}} & &
\end{array}}$. The remaining standard Young tableau $\bT=\begin{array}{ccc}
\fbox{\scriptsize{$1$}} & \hspace{-0.35cm}\fbox{\scriptsize{$3$}} & \hspace{-0.35cm}\fbox{\scriptsize{$4$}} \\[-0.2em]
\fbox{\scriptsize{$2$}} & &
\end{array}$ of shape $\lambda$ gives a Young tableau $\overline{\bT}=\begin{array}{ccc}
\fbox{\scriptsize{$1$}} & \hspace{-0.35cm}\fbox{\scriptsize{$2$}} & \hspace{-0.35cm}\fbox{\scriptsize{$2$}} \\[-0.2em]
\fbox{\scriptsize{$1$}} & &
\end{array}$ which is not semistandard.
\end{exa}

Below, we denote by $\bk^{\text{ord}}_{\vert n}$ the partition of $k_1+\dots+k_n$ obtained by ordering in decreasing order the numbers $k_1,\dots,k_n$.
\begin{thm}\label{thm-rep}
For any $\lambda\vdash k_1+\dots+k_n$, set $W_{\bk,\lambda}:=P_{\bk,n}(V_{\lambda})$.
\begin{enumerate}
\item The space $W_{\bk,\lambda}$ is spanned by the vectors $w_{\bbT}$, where $\bbT\in\SSTab(\lambda,\bk_{\vert n})$.
\item A complete set of pairwise non-isomorphic irreducible (non-zero) representations of $H_{\bk,n}(q)$ is 
\[\{W_{\bk,\lambda}\}_{\lambda\in S_{\bk,n}}\ \ \ \ \ \text{with $S_{\bk,n}:=\{\lambda\vdash k_1+\dots+k_n\ |\ \lambda\geq\bk^{\text{ord}}_{\vert n}\}$}.\]
The dimension of $W_{\bk,\lambda}$ is the Kostka number $K_{\lambda,\bk_{\vert_n}}=|\SSTab(\lambda,\bk_{\vert n})|$.
\item For $\lambda\in S_{\bk,n}$, the restriction of $W_{\bk,\lambda}$ to $H_{\bk,n-1}(q)$ decomposes as:
\begin{equation}\label{BR-PHP}
\text{Res}_{H_{\bk,n-1}(q)}(W_{\bk,\lambda})\,\cong\,\bigoplus_{\mu\in\text{Res}_\bk(\lambda)} W_{\bk,\mu}\,,
\end{equation}
where we have set:
$$\text{Res}_{\bk}(\lambda):=\bigl\{\mu\in S_{\bk,n-1}\ |\  \text{$\mu\subset\lambda$ and $\lambda/\mu$ contains at most one box in each column}\bigr\}.$$
\end{enumerate}
\end{thm}
The condition $\lambda\geq \bk^{\text{ord}}_{\vert n}$ implies easily that $l(\lambda)\leq n$ (see below), so that we have:
\[S_{\bk,n}\subset\{\lambda\vdash k_1+\dots+k_n\ |\ l(\lambda)\leq n\}\ .\] 
In general, the inclusion is strict (see however Subsection \ref{subsec-rep-const} for a situation where it is equivalent to $l(\lambda)\leq n$).

Denote by $k'_1\geq k'_2\geq\dots\geq k'_n$ the parts of $\bk_{\vert n}$ after reordering. Then for partitions $\lambda\vdash k_1+\dots+k_n$ appearing in $S_{\bk,n}$, the condition $\lambda\geq \bk^{\text{ord}}_{\vert n}$ can be expressed by the two following equivalent set of inequalities (using $|\lambda|=\lambda_1+\dots+\lambda_n=k'_1+\dots+k'_n$):
\begin{equation}\label{k-condition}
\lambda\in S_{\bk,n}\ \ \Leftrightarrow\ \ \left\{\begin{array}{rcl} \lambda_1 & \geq & k'_1 \\
\lambda_1+\lambda_2 & \geq & k'_1+k'_2\\
 & \vdots & \\
 \lambda_1+\dots+\lambda_{n-1} & \geq & k'_1+\dots+k'_{n-1}\\
 \lambda_1+\dots+\lambda_n & \geq & k'_1+\dots+k'_n\end{array}\right.
 \Leftrightarrow\ \ \left\{\begin{array}{rcl} \lambda_n & \leq & k'_n \\
\lambda_{n-1}+\lambda_n & \leq & k'_{n-1}+k'_n\\
 & \vdots & \\
 \lambda_2+\dots+\lambda_{n} & \leq & k'_2+\dots+k'_{n}\\
 \lambda_1+\dots+\lambda_n & \leq & k'_1+\dots+k'_n\end{array}\right.
\end{equation} 

Before proving the theorem, we establish the following combinatorial bijection underlying the branching rules (\ref{BR-PHP}). It is the generalisation for general $\bk$ of the natural bijection, for $\lambda\vdash n$,
\[\STab(\lambda)\ \ \stackrel{1\text{-}1}{\longleftrightarrow}\ \ \bigcup_{\begin{array}{c}
\\[-1.6em]
\scriptstyle{\mu\,\vdash n-1} \\[-0.4em]
\scriptstyle{\mu\subset \lambda}
\end{array}}\STab(\mu)\ ,\]
underlying the branching rules for the chain of Hecke algebras $H_n(q)$.
\begin{lem}\label{lem-BR}
For any $\lambda\vdash k_1+\dots+k_n$, we have:
\begin{equation}\label{comb-bij}
\SSTab(\lambda,\bk_{\vert n})\ \ \stackrel{1\text{-}1}{\longleftrightarrow}\ \ \bigcup_{\mu\in\text{Res}_{\bk}(\lambda)}\SSTab(\mu,\bk_{\vert n-1})\ .
\end{equation}
\end{lem}
\begin{proof}
The map from left to right is given by starting from $\bbT\in\SSTab(\lambda,\bk_{\vert n})$ and removing the $k_n$ boxes containing $n$. Denote $\mu$ the shape of the resulting tableau $\bbT'$. Then we have obviously that $\bbT'\in\SSTab(\mu,\bk_{\vert n-1})$ with $\mu\vdash k_1+\dots+k_{n-1}$ and $\mu\subset\lambda$. As $\bbT$ is semistandard, we also have immediately that $\lambda/\mu$ contains at most one box in each column. It remains to show that $\mu\geq \bk_{\vert n-1}^{\text{ord}}$. As we have at hand an element $\bbT'$ of $\SSTab(\mu,\bk_{\vert n-1})$, this set is thus non-empty. So we can apply Lemma \ref{lem-Kos}.

The map from right to left is given by starting from $\bbT'\in\SSTab(\mu,\bk_{\vert n-1})$ with $\mu\in\text{Res}_\bk(\lambda)$ and by adding to $\bbT'$ the boxes of $\lambda/\mu$ filled with numbers $n$. The resulting tableau is clearly an element of $\SSTab(\lambda,\bk_{\vert n})$.

By construction the two maps are inverse to each other.
\end{proof}

\begin{proof}[Proof of the Theorem]
$\bullet$ First we show how item 2 follows from item 1. We know that a complete set of pairwise non-isomorphic irreducible representations of $H_{\bk,n}$ is given by the non-zero $W_{\bk,\lambda}=P_{\bk,n}(V_{\lambda})$ from the results recalled in Appendix \ref{app-ss}. The assertion about the dimension of $W_{\bk,\lambda}$ is immediate from item 1 since the set $\{w_{\bbT}\}_{\bbT\in\SSTab(\lambda,\bk_{\vert n})}$ is clearly linearly independent. So we only need to show that, for any $\lambda\vdash k_1+\dots+k_n$, we have:
$$\SSTab(\lambda,\bk_{\vert n})\neq\emptyset\ \ \ \ \ \Longleftrightarrow\ \ \ \ \ \lambda\geq \bk^{\text{ord}}_{\vert n}\ .$$
This is Lemma \ref{lem-Kos} with $\mu=(k_1,\dots,k_n)$.

\vskip .2cm
$\bullet$ Assume that $n=1$ and let $\lambda\vdash k_1$. Here we have $P_{\bk,n}=P_{k_1}\in H_{k_1}(q)$ and the subspace $W_{\bk,\lambda}=P_{k_1}(V_{\lambda})$ is thus obtained as a particular case of Proposition \ref{prop-rep}. We have:
\[W_{\bk,\lambda}=\CC\bigl(\sum\limits_{\bT\in \STab(\lambda)}v_{\bT}\bigr)\ \ \text{if $\lambda=(k_1)$,}\ \ \ \ \ \ \text{and}\ \ \ \ \ W_{\bk,\lambda}=0\ \ \text{if $\lambda\neq (k_1)$.}\]
Besides, for any $\lambda\vdash k_1$, there is a single Young tableau of shape $\lambda$ and weight $(k_1)$ (since all the boxes are filled with $1$'s). Clearly, this Young tableau is semistandard if and only if $\lambda=(k_1)$. So we have $\SSTab(\lambda,(k_1))=\emptyset$ if $\lambda\neq(k_1)$, while $\SSTab\lambda,(k_1))$ consists of one element $\bbT$ if $\lambda\neq(k_1)$. In this latter case, we have in addition that $\overline{\bT}=\bbT$ for any $\bT\in\STab(\lambda)$. This proves item 1 for $n=1$.

\vskip .2cm
$\bullet$ Now let $n>1$. Recall that for any decomposition $N=N_1+N_2$ with $N_1,N_2\geq 0$, there is a parabolic subalgebra of the Hecke algebra $H_{N}(q)$ isomorphic to $H_{N_1}(q)\otimes H_{N_2}(q)$, where the copy of $H_{N_1}(q)$ is generated by the $N_1-1$ first generators of $H_{N}(q)$ and the copy of $H_{N_2}(q)$ is generated by the $N_2-1$ last ones.

Let $\lambda\vdash k_1+\dots+k_n$. We start by explaining that the restriction of $V_{\lambda}$ to the subalgebra $H_{k_1+\dots+k_{n-1}}(q)\otimes H_{k_n}(q)$ of $H_{k_1+\dots+k_n}(q)$ decomposes as follows:
\begin{equation}\label{proof-dec}
\text{Res}_{H_{k_1+\dots+k_{n-1}}(q)\otimes H_{k_n}(q)}(V_{\lambda})\,\cong\,\bigoplus_{\begin{array}{c}
\\[-1.6em]
\scriptstyle{\mu\,\vdash k_1+\dots+k_{n-1}} \\[-0.4em]
\scriptstyle{\mu\subset \lambda}
\end{array}} V_{\mu}\otimes V_{\lambda/\mu}\ .
\end{equation}
Let $\bT\in\STab(\lambda)$. We denote $\bT_{\downarrow}$ the standard Young tableau obtained from $\bT$ by keeping the boxes with numbers $1,\dots, k_1+\dots+k_{n-1}$, and let $\mu$ be its shape. The boxes of $\bT$ containing the remaining numbers $k_1+\dots+k_{n-1}+1,\dots,k_1+\dots+k_n$ form a tableau of shape $\lambda/\mu$ and let $\bT_{\uparrow}\in\STab(\lambda/\mu)$ denote the corresponding standard Young tableau obtained by shifting the numbers by $k_1+\dots+k_{n-1}$. The linear map defined by
\begin{equation}\label{proof-dec-iso}
v_{\bT}\mapsto v_{\bT_{\downarrow}}\otimes v_{\bT_{\uparrow}}\in V_{\mu}\otimes V_{\lambda/\mu}\ ,
\end{equation}
provides the isomorphism (\ref{proof-dec}). This follows immediately from a direct inspection of Formulas (\ref{rep-si}) giving the action of the generators of the Hecke algebra $H_{k_1+\dots+k_n}(q)$ on $V_{\lambda}$.

Then recall that the idempotent $P_{\bk,n}$ of $H_{k_1+\dots+k_n}(q)$ is by definition an element of the subalgebra $H_{k_1+\dots+k_{n-1}}(q)\otimes H_{k_n}(q)$ and can be written as $P_{\bk,n-1}\otimes P_{k_n}$. We deduce from (\ref{proof-dec}) that
\begin{equation}\label{proof-BR}
W_{\bk,\lambda}=P_{\bk,n}(V_\lambda)\,\cong\,\bigoplus_{\begin{array}{c}
\\[-1.6em]
\scriptstyle{\mu\,\vdash k_1+\dots+k_{n-1}} \\[-0.4em]
\scriptstyle{\mu\subset \lambda}
\end{array}} P_{\bk,n-1}(V_{\mu})\otimes P_{k_n}(V_{\lambda/\mu})\ .
\end{equation}
Using the induction hypothesis together with Proposition \ref{prop-rep}, we obtain that $W_{\bk,\lambda}$ is spanned by the set
\begin{equation}\label{proof-set}
 \sum\limits_{\bT\in \STab(\lambda/\mu)}w_{\bbT'}\otimes v_{\bT}\ ,\ \ \ \ \text{where $\mu\in\text{Res}_{\bk}(\lambda)$ and $\bbT'\in\SSTab(\mu,\bk_{\vert n-1})$\,.}
 \end{equation}
From Lemma \ref{lem-BR}, we know that this set is in bijection with the set $\SSTab(\lambda,\bk_{\vert n})$. Moreover, we have at once that the vectors in (\ref{proof-set}) correspond under the isomorphism (\ref{proof-dec-iso}) to vectors $w_{\bbT}$ in $V_{\lambda}$ where $\bbT\in\SSTab(\lambda,\bk_{\vert n})$ (The tableau $\bbT$ is obtained by adjoining to $\bbT'$ the boxes of $\lambda/\mu$ filled with letters $n$).
This concludes the proof of item 1.

\vskip .2cm
$\bullet$ Finally, we recall that the algebra $H_{\bk,n}(q)$ is isomorphic to
$P_{\bk,n}H_{k_1+\dots+k_n}(q)P_{\bk,n}$ 
and that, after identification, the inclusion of $H_{\bk,n-1}(q)$ into $H_{\bk,n}(q)$ is given by $P_{\bk,n-1}xP_{\bk,n-1}\mapsto P_{\bk,n-1}xP_{\bk,n-1}\otimes P_{k_n}$. Then the branching rule stated in item 3 follows immediately from (\ref{proof-BR}) together with Proposition \ref{prop-rep}.
\end{proof}

\begin{rem}
Combining item 2 of the preceding theorem with the classical fact that the Kostka numbers $K_{\lambda,\mu}$ do not depend on the ordering of the composition $\mu$ (\cite[Theorem 7.10.2]{St}), we see clearly that, up to isomorphism, the algebra $H_{\bk,n}(q)$ does not depend on the ordering of $(k_1,\dots,k_n)$. However, the chain of algebras $\{H_{\bk,n}(q)\}_{n\geq 0}$ depends obviously on the ordering of $\bk$, and therefore so does its Bratteli diagram, and this will reflect in some statements later about minimal generating sets of quotients of Bratteli diagram, see Proposition \ref{prop-quot2} and Theorem \ref{theo-quot-rep}, item 3.
\end{rem}

\begin{rem}
A numerical consequence of item 2 of the preceding theorem is that:
\[\dim H_{\bk,n}(q)=\sum_{\lambda\vdash k_1+\dots+k_n} |\SSTab(\lambda,\bk_{\vert n})|^2\ .\]
Indeed recall that $|\SSTab(\lambda,\bk_{\vert n})|=0$ if we do not have $\lambda\geq \bk_{\vert n}^{ord}$. The dimension of $H_{\bk,n}(q)$ is  the number of integer matrices with non-negative entries such that the sum of the $a$-th row is $k_a$ and the sum of the $a$-th column is $k_a$. So we have that the number of such matrices is equal to the number of pairs of semistandard Young tableaux of the same shape and of content $(k_1,\dots,k_n)$. This is also a consequence of a bijection between these two sets called Robinson--Schensted--Knuth correspondence \cite{Kn}. 
\end{rem}

\subsection{The situation of a constant sequence $\bk=(k,k,...)$}\label{subsec-rep-const}

We single out the situation of a constant sequence $\bk=(k,k,\dots)$ for an integer $k\geq 1$. In this situation, the parametrisation of irreducible representations in item 2 of Theorem \ref{thm-rep} is much simpler since in fact we will check that $\lambda\geq \bk^{\text{ord}}_{\vert n}$ is simply equivalent to $l(\lambda)\leq n$ (for general $\bk$, this is only a necessary condition).
\begin{coro}\label{coro-rep-const}
Let $\bk=(k,k,\dots)$ for an integer $k\geq 1$. A complete set of pairwise non-isomorphic irreducible (non-zero) representations of $H_{\bk,n}$ is 
\[\{W_{\bk,\lambda}\}_{\lambda\in S_{\bk,n}}\ \ \ \ \ \text{with $S_{\bk,n}:=\{\lambda\vdash kn\ |\ l(\lambda)\leq n\}$}.\]
\end{coro}
\begin{proof}
For a partition $\lambda\vdash kn$, we need to check that $\lambda\geq (k,k,\dots,k)$ is satisfied for any partition $\lambda$ such that $l(\lambda)\leq n$ (this is enough, since as already explained, $l(\lambda)> n$ contradicts $\lambda\geq (k,k,\dots,k)$).

So let $\lambda\vdash kn$ such that $l(\lambda)\leq n$ and assume that $\lambda\geq (k,k,\dots,k)$ is not true. In particular, the inequalities in the first set in (\ref{k-condition}) (where here $k'_1=k'_2=\dots=k'_n=k$) are not all satisfied, so let $i\in\{1,\dots,n\}$ be the smallest index for which the inequality is false. 

If $i=1$ then $\lambda_1<k$ and from $l(\lambda)\leq n$ we have that the size of $\lambda$ is strictly smaller than $kn$. This is a contradiction.

So we have $i>1$. Let $\alpha$ be such that
\[\left\{\begin{array}{l}\lambda_{1}+\dots+\lambda_{i-1}=(i-1)k+\alpha\ ,\\[0.3em]
\lambda_1+\dots+\lambda_i<ik\ .\end{array}\right.\]
We have $\alpha\geq0$ by minimality of $i$. So we find that $\lambda_i<k-\alpha$. From this and the fact that $l(\lambda)\leq n$, we obtain
\[|\lambda|\leq\lambda_1+\dots+\lambda_{i}+(n-i)\lambda_i<ik+(n-i)(k-\alpha)=nk-\alpha(n-i)\ ,\]
which shows that $|\lambda|$ is strictly smaller than $kn$. This is a contradiction.
\end{proof}

\begin{exa}
$\bullet$ If $\bk=(1,1,\dots)$ is the constant sequence of $1$'s, the theorem expresses simply the representation theory of the chain of Hecke algebras $H_n(q)$ in the usual way. Indeed one has in this case $P_{\bk,n}=1$, $W_{\bk,\lambda}=V_{\lambda}$, $S_{\bk,n}=\{\lambda\vdash n\}$, a semistandard tableau of weight $\bk_{\vert n}=(1,\dots,1)$ is simply a standard  tableau, and the branching rules are only given by $\mu\subset\lambda$ since $\lambda/\mu$ contains only one box.

$\bullet$ The first levels of the Bratteli diagrams of the chains $\{H_{\bk,n}\}_{n\geq 0}$ for $\bk=(2,2,2,2,...)$ and for $\bk=(3,1,1,1,...)$ are given in Appendix \ref{app2}-\ref{app3}.
\end{exa}

\section{$H_{\bk-1,n}(q)$ as a subalgebra and as a quotient of $H_{\bk,n}(q)$}\label{sec-quot-k}

In this section we assume that our sequence $\bk=(k_1,k_2,...)$ consists of strictly positive integers, and we set $\bk-1=(k_1-1,k_2-1,...)\in \mathbb{Z}_{\geq0}^{\infty}$ the sequence obtained by decreasing every entry of $\bk$ by 1. We are going to study the connections between $H_{\bk,n}(q)$ and $H_{\bk-1,n}(q)$.

We use the definitions and the terminology fixed in the appendix for quotients of semisimple algebras and quotients of Bratteli diagrams.

\subsection{$H_{\bk-1,n}(q)$ as a subalgebra of $H_{\bk,n}(q)$}\label{subsec-subal-k}

Let $n\geq 0$. Recall that given a subset $\Gamma$ of the parametrising set $S_{\bk,n}$ of irreducible representations of $H_{\bk,n}(q)$, we have an associated subalgebra. Namely, after applying the Wedderburn decomposition and with the notation of Theorem \ref{thm-rep}, the subalgebra is $\bigoplus_{\lambda\in \Gamma} \text{End}(W_{\bk,\lambda})$.

\begin{prop}\label{prop-subal-k}
Let $\mathcal{A}$ be the subalgebra of $H_{\bk,n}(q)$ corresponding to the following subset of partitions:
\[\{\ \lambda\in S_{\bk,n}\ |\ l(\lambda)=n\ \}\ .\]
Then $\mathcal{A}$ is isomorphic to $H_{\bk-1,n}(q)$.
\end{prop}
\begin{proof}
From item 2 of Theorem \ref{thm-rep}, it follows at once that the Artin--Wedderburn decompositions (see Appendix) of $H_{\bk,n}(q)$ and $H_{\bk-1,n}(q)$ read:
\[H_{\bk,n}(q)\cong \bigoplus_{\lambda\in S_{\bk,n}} \text{End}(W_{\bk,\lambda})\ \ \ \ \quad
\text{and}\ \ \ \ \ \quad
H_{\bk-1,n}(q)\cong \bigoplus_{\lambda'\in S_{\bk-1,n}} \text{End}(W_{\bk-1,\lambda'})\ ,\]
while the subalgebra $\mathcal{A}$ of $H_{\bk,n}(q)$ is given by:
\[\mathcal{A}\cong \bigoplus_{\begin{array}{c}
\\[-1.6em]
\scriptstyle{\lambda\,\in S_{\bk,n}} \\[-0.4em]
\scriptstyle{l(\lambda)= n}
\end{array}} \text{End}(W_{\bk,\lambda})\ .\]
Therefore, it will be enough to give a bijection between $\{\lambda\in S_{\bk,n}\ |\ l(\lambda)=n\}$ and $S_{\bk-1,n}$ respecting the dimensions of the corresponding irreducible representations (which are also explicited in item 2 of Theorem \ref{thm-rep}).

Let $\lambda\in S_{\bk,n}$ such that $l(\lambda)=n$. It means that the first column of $\lambda$ contains $n$ boxes. The bijection will be given simply by the removal of the first column. Namely, we set
\begin{equation}\label{bij-proof}
\begin{array}{crcl}
\phi_n\ : & \{\lambda\in S_{\bk,n}\ |\ l(\lambda)=n\} & \to & S_{\bk-1,n}\\[0.5em]
 & \lambda & \mapsto & \phi_n(\lambda)
 \end{array},
\end{equation}
where $\phi_n(\lambda)$ is the partition obtained from $\lambda$ by removing the first column. We must check that $\phi_n$ is indeed a bijection. 

First $\phi_n$ takes values in $S_{\bk-1,n}$. Indeed $\phi_n(\lambda)$ is of the correct size since $|\phi_n(\lambda)|=|\lambda|-n=(k_1-1)+(k_2-1)+\dots+(k_n-1)$. And moreover, setting $\bk^{\text{ord}}_{\vert n}=(k'_1,\dots,k'_n)$, we have that $\phi_n(\lambda)\geq (k'_1-1,\dots,k'_n-1)$ since it follows immediately from $\lambda\geq \bk_{\vert n}^{\text{ord}}$ that $(\lambda_1-1)+\dots+(\lambda_i-1)\geq (k'_1-1)+\dots+(k'_i-1)$ for all $i$.

A similar argument shows that adding a first column of size $n$ to partitions in $S_{\bk-1,n}$ provides the inverse map.

\vskip .2cm
Now let $\lambda\in S_{\bk,n}$, and set $\lambda':=\phi_n(\lambda)$. It remains to show that the cardinality of $\SSTab(\lambda,\bk_{\vert n})$ is equal to the cardinality of $\SSTab(\lambda',(\bk-1)_{\vert n})$. Let $\bbT\in \SSTab(\lambda,\bk_{\vert n})$. As the first column of $\lambda$ contains $n$ boxes, it has to be filled, in order for $\bbT$ to be semistandard, by $1,2,\dots,n$ in ascending order. Removing this first column we thus obtain an element of $\SSTab(\lambda',(\bk-1)_{\vert n})$. This gives the desired bijection proving thereby that $W_{\bk,\lambda}$ and $W_{\bk-1,\lambda'}$ have the same dimension and concluding the proof.
\end{proof}

\subsection{$H_{\bk-1,n}(q)$ as a quotient of $H_{\bk,n}(q)$}\label{subsec-quot-k}

\begin{prop}\label{prop-quot1}
Let $I_{\bk,n}^{<n}$ be the ideal of $H_{\bk,n}(q)$ corresponding to the following subset of partitions:
\[S_{\bk,n}^{<n}:=\{\ \lambda\in S_{\bk,n}\ |\ l(\lambda)<n\ \}\ .\]
\begin{enumerate}
\item The quotient $H_{\bk,n}(q)/I_{\bk,n}^{<n}$ is isomorphic to $H_{\bk-1,n}(q)$;
\item Let $S^{<}_{\bk}$ be the union of $S^{<n}_{\bk,n}$ for all $n$. The quotient of the Bratteli diagram of $\{H_{\bk,n}(q)\}_{n\geq 0}$ generated by $S^{<}_{\bk}$ is the Bratteli diagram of $\{H_{\bk-1,n}(q)\}_{n\geq 0}$.
\end{enumerate}
\end{prop}
\begin{proof}
\textbf{1.} By standard facts recalled in the appendix, the quotient $H_{\bk,n}(q)/I_{\bk,n}^{<n}$ is isomorphic to the subalgebra of $H_{\bk,n}(q)$ corresponding to the subset of $S_{\bk,n}$ complementary to $S_{\bk,n}^{<n}$. This subset is $\{\lambda\in S_{\bk,n}\ |\ l(\lambda)=n\}$. So the first item follows immediately from Proposition \ref{prop-subal-k}.

\vskip .2cm
\textbf{2.} Let $\langle S^{<}_{\bk}\rangle$ be the set of vertices generated by $S^{<}_{\bk}$ which have been removed to obtain the quotient of the Bratteli diagram of $\{H_{\bk,n}(q)\}_{n\geq 0}$. We start by showing that we have here $\langle S^{<}_{\bk}\rangle=S^{<}_{\bk}$.

Let $\lambda\in S^{<n}_{\bk,n}$ for some $n\geq 0$, namely $\lambda\in S_{\bk,n}$ with $l(\lambda)<n$. From the branching rules proved in Theorem \ref{thm-rep}, recall that $\lambda'\in S_{\bk,n+1}$ is connected to $\lambda$ if and only if $\lambda\in\text{Res}_\bk(\lambda')$ where
$$\text{Res}_k(\lambda'):=\bigl\{ \mu\in S_{\bk,n}\ |\ \text{$\mu\subset\lambda'$ and $\lambda'/\mu$ contains at most one box in each column}\bigr\}.$$
So if $\lambda'$ is connected to $\lambda$, we have that $\lambda'/\lambda$ contains at most one box in each column and then we obtain $l(\lambda')<n+1$ from $l(\lambda)<n$. In other words $\lambda'\in S^{<n+1}_{\bk,n+1}$. This shows that $\langle S^{<}_{\bk}\rangle=S^{<}_{\bk}$.

So, for any level $n\geq 0$, the vertices of the quotient of the Bratteli diagram of $\{H_{\bk,n}(q)\}_{n\geq 0}$ generated by $S^{<}_{\bk}$ are indexed by partitions $\lambda\in S_{\bk,n}$ with $l(\lambda)=n$. We already have a bijection, which was denoted $\phi_n$ in (\ref{bij-proof}), between this set and the set of vertices of level $n$ of the Bratteli diagram of $\{H_{\bk,n-1}(q)\}_{n\geq 0}$.

\vskip .2cm
Then it remains to show that the edges are the same in the two diagrams, namely that the bijections $\phi_n$ commute with the branching rules of the two chains. More precisely, recall that we have
\[\text{Res}_{H_{\bk,n-1}(q)}(W_{\bk,\lambda})\,\cong\,\bigoplus_{\mu\in\text{Res}_\bk(\lambda)} W_{\bk,\mu}\ \ \ \ \ \text{and}\ \ \ \ \ \text{Res}_{H_{\bk-1,n-1}(q)}(W_{\bk-1,\phi_n(\lambda)})\,\cong\,\bigoplus_{\mu\in\text{Res}_{\bk-1}(\phi_n(\lambda))} W_{\bk-1,\mu}\ .\]
So it remains to prove the following equality of sets:
\[\phi_{n-1}\bigl(\text{Res}_\bk(\lambda)\bigr)=\text{Res}_{\bk-1}\bigl(\phi_n(\lambda)\bigr)\ \ \ \ \ 
\forall \lambda\in S_{\bk,n}\ \text{with}\ l(\lambda)=n\ .\]
Note that if $\mu\in \text{Res}_\bk(\lambda)$ with $l(\lambda)=n$, then $l(\mu)=n-1$ since $\lambda/\mu$ contains at most one box in each column. So it is well-defined to apply to $\text{Res}_\bk(\lambda)$ the bijection $\phi_{n-1}$.

Finally, the above equality of sets follows immediately from the following two inclusions which are straightforward to check from the definitions: $\phi_{n-1}\bigl(\text{Res}_\bk(\lambda)\bigr)\subset \text{Res}_{\bk-1}\bigl(\phi_n(\lambda)\bigr)$ and $\phi^{-1}_{n-1}\bigl(\text{Res}_{\bk-1}\bigl(\phi_n(\lambda)\bigr)\bigr)\subset \text{Res}_\bk(\lambda)$.
\end{proof}

\paragraph{The minimal set of generators.}
We proved that the algebras $\{H_{\bk-1,n}(q)\}_{n\geq 0}$ can be seen as quotients of $\{H_{\bk,n}(q)\}_{n\geq 0}$. More precisely we identified a set of partitions 
$$S^{<}_{\bk}=\bigcup_{n\geq 0}\{\lambda\in S_{\bk,n}\ |\ l(\lambda)<n\}\,,$$ 
such that the quotient of the Bratteli diagram of $\{H_{\bk,n}(q)\}_{n\geq 0}$ generated by $S^{<}_{\bk}$ coincides with the Bratteli diagram of $\{H_{\bk-1,n}(q)\}_{n\geq 0}$.

Let $S_{min}$ be the minimal set of partitions generating this quotient (see the appendix for this terminology). For the following statement, recall that $\lambda_{n-1}$ is the number of boxes in line number $n-1$ of the partition $\lambda$. Note that $S_{min}$ depends on the ordering of $\bk$.

\begin{prop}\label{prop-quot2}
We have:
\[S_{min}=\bigcup_{n \geq 0}\{\lambda\in S_{\bk,n}\ |\ l(\lambda)<n\,,\ \lambda_{n-1}>k_n\}\,.\]
\end{prop}
\begin{proof} To extract a minimal set of generators from $S^{<}_{\bk}$, we must remove $\lambda$ from $S^{<}_{\bk}$ if and only if there is $\mu\in\text{Res}_\bk(\lambda)$ such that $\mu$ is already in $S^{<}_{\bk}$.  

So let $\lambda\in S_{\bk,n}$ with $l(\lambda)<n$. We must prove that
\[\text{There is a partition $\mu\in\text{Res}_\bk(\lambda)$ such that $l(\mu)<n-1$}\ \ \ \Longleftrightarrow\ \ \ \lambda_{n-1}\leq k_n\,.
\]
Assume first that $\lambda_{n-1}>k_n$. Then we cannot obtain a partition $\mu$ with $l(\mu)<n-1$ by removing $k_n$ boxes from $\lambda$ since there are more than $k_n$ boxes in line $n-1$.

Assume now that $\lambda_{n-1}\leq k_n$. We apply the procedure of Lemma \ref{lem-comb}. This furnishes a semistandard tableau $\bbT\in\SSTab(\lambda,\bk_{\vert n})$ with the property that the last line of $\lambda$ is filled with letters $n$ (since $\lambda_{n-1}\leq k_n$).
Removing from $\lambda$ the boxes wich are filled with $n$ in $\bbT$, we obtain a partition $\mu\in \text{Res}_\bk(\lambda)$. Since we removed in particular all the boxes of the line $n-1$ of $\lambda$ we have $l(\mu)<n-1$, as desired.
\end{proof}

\subsection{The situation of a constant sequence $\bk=(k,k,...)$}

We single out the situation of a constant sequence $\bk=(k,k,\dots)$ for an integer $k\geq 1$. In this situation, we can give more information on the minimal generating set $S_{min}$ for the quotient of the Bratteli diagram of $\{H_{\bk,n}(q)\}_{n\geq 0}$ giving the Bratteli diagram of $\{H_{\bk-1,n}(q)\}_{n\geq 0}$
\begin{coro}\label{prop-quot2-const}
Let $\bk=(k,k,\dots)$ for an integer $k\geq 1$. We have:
\[S_{min}=\bigcup_{n\geq 0}\{\lambda\vdash kn\ |\ l(\lambda)<n\,,\ \lambda_{n-1}>k\}\,.\]
Furthermore, $S_{min}$ consists only of vertices of levels $\leq k+1$ (and contains at least a vertex of level  $k+1$).
\end{coro}
\begin{proof} The description of $S_{min}$ follows immediately from Proposition \ref{prop-quot2} together with the description of $S_{\bk,n}$ in this case given in Corollary \ref{coro-rep-const}. 

To prove that $S_{min}$ consists only of vertices of levels $\leq k+1$, we must show that for $n>k+1$, there is no $\lambda\vdash kn$ with $l(\lambda)<n$ and $\lambda_{n-1}>k$. Indeed, this would imply
\[|\lambda|\geq \lambda_{n-1}(n-1)\geq (k+1)(n-1)=kn+n-(k+1)>kn\ .\]
Moreover, consider the partition $\lambda=(k+1,\dots,k+1)$ consisting of $k$ lines of $k+1$ boxes each. With $n=k+1$, we have $\lambda\vdash kn$, $l(\lambda)=n-1<n$ and $\lambda_{n-1}=k+1>k$. Therefore, $\lambda\in S_{min}$ and is a vertex of level $k+1$.
\end{proof}

\paragraph{Example of $\bk=(1,1,\dots,)$.} If $k=1$, recall that $H_{\bk,n}(q)$ is the Hecke algebra $H_n(q)$ while $H_{\bk-1,n}(q)=\CC$ for any $n\geq 0$. In this example, the quotient explained in Proposition \ref{prop-quot2-const} kills all partitions in the Bratteli diagram of the Hecke algebras, except the ones of the form $(1,\dots,1)$ (a single column of $n$ boxes) for $n\geq 0$. So there remains only one vertex in each level, and the resulting quotient is indeed the Bratteli diagram of the constant chain of algebras $\CC$. Then Proposition \ref{prop-quot2} asserts that this quotient is generated by the single partition $\begin{array}{cc}
\fbox{\phantom{\scriptsize{$2$}}} &\hspace{-0.35cm}\fbox{\phantom{\scriptsize{$2$}}}
\end{array}$.

In algebraic terms, this amounts to saying that, for any $n\geq 2$, the quotient of the Hecke algebra $H_n(q)$ by the relation $\sigma_1+q^{-1}=0$ is isomorphic to $\CC$. This is quite straightforward to check directly.

\vskip .2cm
\paragraph{Example of $\bk=(2,2,\dots,)$.} If $k=2$, Proposition \ref{prop-quot2}, item 1, explains that the Hecke algebra $H_n(q)$ is isomorphic to a quotient of $H_{\bk,n}(q)$. In the Bratteli diagram, this is seen by keeping at each level $n$ the partitions of $2n$ with exactly $n$ lines (see for example Appendix \ref{app2}). Here Corollary \ref{prop-quot2-const} asserts that in fact the quotient is generated by two partitions: $\begin{array}{cccc}
\fbox{\phantom{\scriptsize{$2$}}} &\hspace{-0.35cm}\fbox{\phantom{\scriptsize{$2$}}} &\hspace{-0.35cm}\fbox{\phantom{\scriptsize{$2$}}} &\hspace{-0.35cm}\fbox{\phantom{\scriptsize{$2$}}}
\end{array}$ (of level 2) and $\begin{array}{ccc}
\fbox{\phantom{\scriptsize{$2$}}} &\hspace{-0.35cm}\fbox{\phantom{\scriptsize{$2$}}} & \hspace{-0.35cm}\fbox{\phantom{\scriptsize{$2$}}}\\[-0.2em]
\fbox{\phantom{\scriptsize{$2$}}} &\hspace{-0.35cm}\fbox{\phantom{\scriptsize{$2$}}} & \hspace{-0.35cm}\fbox{\phantom{\scriptsize{$2$}}}
\end{array}$ (of level 3).

Algebraically, it means that the Hecke algebra $H_n(q)$ is isomorphic to a quotient of $H_{\bk,n}(q)$ by one relation coming from level 2 (that is, in $H_{\bk,2}(q)$) and one relation coming from level 3 (that is, in $H_{\bk,3}(q)$).

\section{Representation theory and Bratteli diagrams for $\oH^N_{\bk,n}$}\label{sec-quot-N}

In this section we keep $\bk=(k_1,k_2,...)\in \mathbb{Z}_{\geq0}^{\infty}$ as before and we fix $N>1$. Recall from Section \ref{sec-cent} that, for $n\geq 0$, the centraliser of the action of $U_q(gl_N)$ on the representation
\[L^N_{(k_1)}\otimes\dots\otimes L^N_{(k_n)}\]
was denoted $\oH^N_{\bk,n}(q)$ (where $L^N_{(k)}$ is the $k$-th symmetric power representation of $U_q(gl_N)$). Morever, this centraliser was obtained as a homomorphic image (that is, a quotient) of the algebra $H_{\bk,n}(q)$. Recall also that the corresponding ideal was denoted $I_{\bk,n}^N$, so that we have $\oH^N_{\bk,n}(q)=H_{\bk,n}(q)/I_{\bk,n}^N$.

\subsection{Chain structure of $\{\oH^N_{\bk,n}(q)\}_{n\geq 0}$}

In order to speak of Bratteli diagrams, we will first make explicit the inclusion maps making the family of algebras $\{\oH^N_{\bk,n}(q)\}_{n\geq 0}$ into a chain of algebras.

By definition, for any $n\geq 0$, the algebra $\oH^N_{\bk,n}(q)$ is an algebra of endomorphisms of the vector space $L^N_{(k_1)}\otimes\dots\otimes L^N_{(k_n)}$, so that the map in the following proposition makes sense.
\begin{prop}\label{prop-incl}
For $n\geq 0$, the following map provides an inclusion of algebras:
\[\oH^N_{\bk,n}(q)\ni\ x \, \mapsto \, x\otimes\text{Id}_{L_{(k_{n+1})}^N} \ \in \oH^N_{\bk,n+1}(q)\,.\]
\end{prop}
\begin{proof}
This is in fact a particular case of the following general situation. Let $U$ be an algebra and let $\Delta$ be a morphism of algebras $U\to U\otimes U$. For any representations $L,M$ of $U$, the space $L\otimes M$ is also a representation of $U$ for the action given by composing the natural action of $U\otimes U$ on $L\otimes M$ with the map $\Delta$.

Under this assumption on $U$, we have that the following map
\[\text{End}_U(L)\ni\ x \, \mapsto \, x\otimes\text{Id}_L \ \in \text{End}_U(L\otimes M)\,,\]
is a well-defined injective map, which gives the natural inclusion of algebras $\text{End}_U(L)\to \text{End}_U(L\otimes M)$. The injectivity is obvious. Moreove, for $x\in \text{End}_U(L)$ it is immediate that $x\otimes \text{Id}_M$ commutes with the action of any element of $U\otimes U$, in particular with the action of all elements of the form $\Delta(u)$, for $u\in U$. Thus the map indeed takes values in the centraliser $\text{End}_U(L\otimes M)$.

We cover the situation of the proposition by taking $U=U_q(gl_N)$ and $\Delta$ the coproduct of $U_q(gl_N)$ that we used to make tensor products of representations, $L=L^N_{(k_1)}\otimes\dots\otimes L^N_{(k_n)}$ and $M=L_{(k_{n+1})}^N$. Note that $L\otimes M=\Bigl(L^N_{(k_1)}\otimes\dots\otimes L^N_{(k_n)}\Bigr)\otimes L_{(k_{n+1})}^N=L^N_{(k_1)}\otimes\dots\otimes L^N_{(k_n)}\otimes L_{(k_{n+1})}^N$ by coassociativity of the coproduct $\Delta$ of $U_q(gl_N)$.
\end{proof}

\begin{rem}\label{rem-sub}
Keeping $U$, $\Delta$ as in the proof of the proposition and assume that the morphism $\Delta$ satisfies the coassociativity property, so that we do not have to indicate parenthesis in $n$-fold tensor products of representations. Then one can check that a more general property is satisfied (defining ``parabolic subalgebras'' in centralisers). Namely, if $n\geq 0$ and $L_1,\dots, L_n$ are representations of $U$, we have the following inclusion of algebras:
\[\text{End}_U(L_1)\otimes \dots\otimes \text{End}_U(L_n)\ni x_1\otimes\dots\otimes x_n\mapsto x_1\otimes\dots\otimes x_n\in \text{End}_U(L_1\otimes\dots\otimes L_n)\,.\]
\end{rem}

\subsection{Bratteli diagram of $\{\oH^N_{\bk,n}(q)\}_{n\geq 0}$}

The next result identifies the ideals $I_{\bk,n}^N$ in terms of the representation theory, and explains how to obtain the Bratteli diagram for the chain of algebras $\oH^N_{\bk,n}(q)$ easily from the Bratteli diagram of the chain of algebras $H_{\bk,n}(q)$. So in particular we obtain the representation theory of all the algebras $\oH^N_{\bk,n}$ and the branching rules from $\oH^N_{\bk,n}(q)$ to $\oH^N_{\bk,n-1}(q)$.

We are using the terminology of the appendix for quotients of semisimple algebras and quotients of Bratteli diagrams. We recall that Theorem \ref{thm-rep} gives the representation theory of $H_{\bk,n}(q)$ to which the following statements refer to.
\begin{thm}\label{theo-quot-rep}
$\ $
\begin{enumerate}
\item For $n\geq 0$, the ideal $I_{\bk,n}^N$ of $H_{\bk,n}(q)$ corresponds to the following subset of partitions:
\[\{\ \lambda\in S_{\bk,n}\ |\ l(\lambda)>N\ \}\ ;\]
\item The Bratteli diagram of the chain  $\{\oH_{\bk,n}^N(q)\}_{n\geq0}$ is the quotient of the Bratteli diagram of $\{H_{\bk,n}(q)\}_{n\geq0}$ generated by:
\[S=\bigcup_{n\geq 0}\{\ \lambda\in S_{\bk,n}\ |\ l(\lambda)>N\ \}\ .\]
\item Assume that the sequence $\bk$ is in decreasing order. The quotient is generated by:
\[S_{min}:=\{\lambda\in S_{\bk,N+1}\ |\ l(\lambda)=N+1\}\ .\]
In particular, $S_{min}$ only contains vertices of level $N+1$ and is the minimal generating set for the quotient.
\end{enumerate}
\end{thm}
\begin{proof} 
\textbf{1.} Recall that the Hecke algebra $H_{k_1+\dots+k_n}(q)$ has the following Artin--Wedderburn decomposition:
\[H_{k_1+\dots+k_n}(q)\cong \bigoplus_{\lambda\vdash k_1+\dots+k_n} \text{End}(V_{\lambda})\ .\]
We proved in Proposition \ref{prop-PHP} that $H_{\bk,n}(q)\cong P_{\bk,n}H_{k_1+\dots+k_n}(q)P_{\bk,n}$ and in Theorem \ref{thm-rep} that its Artin--Wedderburn decomposition is
\[H_{\bk,n}(q)\cong P_{\bk,n}H_{k_1+\dots+k_n}(q)P_{\bk,n}\cong \bigoplus_{\lambda\in S_{\bk,n}} \text{End}(W_{\bk,\lambda})\ ,\]
where $W_{\bk,\lambda}=P_{\bk,n}(V_{\lambda})$ is non-zero for $\lambda\in S_{\bk,n}$.

Then we have from Proposition \ref{prop-SWk-ideal} that the ideal $I_{\bk,n}^N$ corresponds in $H_{k_1+\dots+k_n}(q)$ to $P_{\bk,n}I_{k_1+\dots+k_n}^NP_{\bk,n}$ where $I_{k_1+\dots+k_n}^N$ is the ideal of $H_{k_1+\dots+k_n}(q)$ appearing in the classical Schur--Weyl duality in Theorem \ref{theo-SW}. From this theorem, we have
\[I_{k_1+\dots+k_n}^N\cong \bigoplus_{\substack{\lambda\vdash k_1+\dots+k_n \\ l(\lambda)>N}} \text{End}(V_{\lambda})\ .\]
From all this, it follows directly that
\[I_{\bk,n}^N\cong \bigoplus_{\substack{\lambda\in S_{\bk,n} \\ l(\lambda)>N}} \text{End}(W_{\bk,\lambda})\ .\]
This proves item 1.

\vskip .2cm
\textbf{2.} Let $\mathcal{D}$ be the Bratteli diagram of the chain $\{H_{\bk,n}(q)\}_{n\geq0}$ described in Theorem \ref{thm-rep}.

First we note that the set $S$ in item 2 do not generate a larger set of partitions. In other words, $\langle S\rangle=S$. Indeed if $\lambda$ is a partition in $\mathcal{D}$ such that there is a path from an element of $S$ to $\lambda$, then in particular $\lambda$ contains (as a subpartition) an element of $ S$ and therefore we have $l(\lambda)>N$.

So at this point, we have, from item 1, that the quotient of $\mathcal{D}$ generated by $S$ contains the correct vertices. It remains to show that the edges of the quotient of $\mathcal{D}$ generated by $S$ indeed express the branching rules for the chain $\{\oH_{\bk,n}^N(q)\}_{n\geq0}$. It amounts to verifying that the following diagram is commutative:
\[\begin{array}{ccl}
H_{\bk,n}(q) & \longrightarrow & \oH_{\bk,n}^N(q)\subset \text{End}\bigl(L_{(k_1)}^N\otimes \dots \otimes L_{(k_n)}^N\bigr)\\[0.5em]
\!\!\!\!\iota\ \downarrow & & \ \ \ \ \downarrow\ \overline{\iota}\\[0.5em]
H_{\bk,n+1}(q) & \longrightarrow & \oH_{\bk,n+1}^N(q)\subset \text{End}\bigl(L_{(k_1)}^N\otimes \dots \otimes L_{(k_n)}^N\otimes L_{(k_{n+1})}^N\bigr)
\end{array}\]
where the horizontal maps are the representation maps of the algebras $H_{\bk,n}(q)$ and $H_{\bk,n+1}(q)$ on the tensor spaces and the vertical maps are the inclusion of algebras. Identifying via Proposition \ref{prop-PHP} the algebra $H_{\bk,n}$ with respectively $P_{\bk,n}H_{k_1+\dots+k_n}(q)P_{\bk,n}$ and similarly for $H_{\bk,n+1}(q)$, we recall that the inclusion map $\iota$ is given by 
$$\iota(P_{\bk,n}xP_{\bk,n})=P_{\bk,n}xP_{\bk,n}\otimes P_{k_{n+1}}\ .$$
Moreover, when acting on $L_{(k_1)}^N\otimes \dots \otimes L_{(k_n)}^N\otimes L_{(k_{n+1})}^N$, the first factor $P_{\bk,n}xP_{\bk,n}$ only acts non-trivially on the $n$ first vector spaces while $P_{k_{n+1}}$ acts as the identity on $L_{(k_{n+1})}^N$. 

On the other hand, the inclusion map $\overline{\iota}$ is given in Proposition \ref{prop-incl}, by $x\mapsto x\otimes\text{Id}_{L_{(k_{n+1})}^N}$. The verification of the commutativity of the diagram is then immediate.

\vskip .2cm
\textbf{3.} If the generating property is true, the fact that $S_{min}$ is minimal is obvious since it contains only vertices of level $N+1$. So it remains to show that the set $S_{min}$ indeed generates the correct subset of partitions in $\mathcal{D}$. For this, we must show that
\[\langle S_{min}\rangle = S\,,\ \ \ \text{where $S=\bigcup_{n\geq 0}\{\ \lambda\in S_{\bk,n}\ |\ l(\lambda)>N\ \}$}\ .\]
First, we know already that the set $S$ do not generate more partitions than those with $l(\lambda)>N$. As $S_{min}\subset S$, this shows the inclusion $\subset$.

Let $n\geq 0$ and let $\lambda\in S_{\bk,n}$ such that $l(\lambda)>N$. We prove that $\lambda\in \langle S_{min}\rangle$ by induction on $n$. If $n\leq N$ there is no partition $\lambda$ satisfying the assumptions. For the induction basis, if $n=N+1$ then the partitions $\lambda$ satisfying the assumptions are exactly the elements of $S_{min}$.

Assume that $n>N+1$. We must show that there is $\mu\in \text{Res}_\bk(\lambda)$ such that $l(\mu)>N$.

First, if $l(\lambda)>N+1$, then for any $\mu\in \text{Res}_\bk(\lambda)$ we have $l(\mu)>N$ since $\lambda/\mu$ must contain at most one box in each column.

So we can assume that $l(\lambda)=N+1$. Recall here that by assumption we have $\bk_{\vert n}^{\text{ord}}=\bk_{\vert n}$. Note that the inequalities in (\ref{k-condition}) expressing that $\lambda\geq\bk_{\vert n}$ cannot be all equalities since, from $l(\lambda)=N+1<n$ we have $\lambda_1+\dots+\lambda_{N+1}=|\lambda|=k_1+\dots+k_n>k_1+\dots+k_{N+1}$. So there is $j\in\{1,\dots,N+1\}$ such that:
\[\begin{array}{rcl}\lambda_1+\dots+\lambda_{j-1} & = & k_1+\dots+k_{j-1}\,,\\
\lambda_j & > & k_j\,,\\
\lambda_j+\lambda_{j+1} & > & k_j+k_{j+1}\,,\\
 & \vdots & \\
 \lambda_{j}+\dots+\lambda_{N+1} & > & k_j+\dots+k_{N+1}\ .\end{array}\]
In words, $j$ is the line from which all the inequalities up to the last line $N+1$ are strict. Moreover, let $j'\geq j$ be such that $\lambda_j=\dots=\lambda_{j'}$ and $\lambda_{j'+1}<\lambda_j$.

We apply Lemma \ref{lem-comb} to find a semistandard tableau in $\SSTab(\lambda,\bk_{\vert n})$ and in turn, by removing the boxes with $n$ and keeping the resulting shape, we have a partition $\mu\in \text{Res}_\bk(\lambda)$. Note that $l(\mu)\in\{N,N+1\}$. From the explicit procedure of Lemma \ref{lem-comb} we see that since $\lambda_{j}=\dots=\lambda_{j'}>k_j\geq k_n$, we have that the $j'$ first lines of $\mu$ are the same as the ones of $\lambda$.

Now, remove one box from $\mu$ at the end of line $j'$ and add it in line $N+1$. It is easy to see that this results in a partition $\tilde{\mu}$. Obviously, $l(\tilde{\mu})=N+1$. We claim that $\tilde{\mu}\in \text{Res}_\bk(\lambda)$. 

To verify the claim, first note that $\tilde{\mu}\subset \lambda$ since the only box we add to $\mu$ was the first box to be removed $\lambda$ in the procedure of Lemma \ref{lem-comb}. We have also immediately that $\lambda/\tilde{\mu}$ contains at most one box in each column. So it remains only to check that $\tilde{\mu}\geq\bk_{\vert n-1}$. Recall from the proof of Lemma \ref{lem-comb} that $\mu\geq\bk_{\vert n-1}$ and furthermore that
\[\mu_1+\dots+\mu_a\in\{\lambda+\dots+\lambda_a,\lambda_1+\dots+\lambda_a+\lambda_{a+1}-k_n\}\ ,\]
depending on the value of $a$. In particular we see that the inequalities are strict if $a>j'$. As the partial sums $\tilde{\mu}_1+\dots\tilde{\mu}_a$ are different from the ones for $\mu$ only for $a>j'$ and they differ only by one, we conclude that $\tilde{\mu}\geq\bk_{\vert n-1}$, which ends the proof that $\lambda\in \langle S_{min}\rangle$.
\end{proof}
\begin{rem}
As recalled before, the structure of $H_{\bk,n}$ does not depend on the ordering of $(k_1,\dots,k_n)$ but the whole chain depends on the ordering of $\bk$. This can be seen in the preceding Theorem, item 3, which would be false without the assumption $k_i\geq k_{i+1}$ for all $i$. This can be seen in the following example. Take $\bk=(1,1,1,3)$ and $N=2$. Then in the Bratteli diagram, there is partition of length 3 at level 4: $\begin{array}{ccc}
\fbox{\phantom{\scriptsize{$2$}}} & \hspace{-0.35cm} \fbox{\phantom{\scriptsize{$2$}}} & \hspace{-0.35cm}\fbox{\phantom{\scriptsize{$2$}}}\\[-0.2em]
\fbox{\phantom{\scriptsize{$2$}}} & \hspace{-0.35cm} \fbox{\phantom{\scriptsize{$2$}}}\\[-0.2em]
\fbox{\phantom{\scriptsize{$2$}}}
\end{array}$, which is not connected to the single partition of length 3 at level 3: $\begin{array}{ccc}
\fbox{\phantom{\scriptsize{$2$}}}\\[-0.2em]
\fbox{\phantom{\scriptsize{$2$}}}\\[-0.2em]
\fbox{\phantom{\scriptsize{$2$}}}
\end{array}$ (indeed, one would have to add two boxes in the same column). So in this example the quotient of the Bratteli diagram is not generated by the partition of length 3 at level 3.
\end{rem}

\subsection{Comparison of the chain of quotients with the chain $\{H_{\bk,n}(q)\}_{n\geq 0}$}

The centralisers form a chain of algebras
\begin{equation}\label{chain-oPHP}
\CC=\oH_{\bk,0}^N(q)\subset \oH_{\bk,1}^N(q)\subset \oH_{\bk,2}^N(q)\subset\dots\dots \subset \oH_{\bk,n}^N(q)\subset \oH^N_{\bk,n+1}(q)\subset \dots\dots\ ,
\end{equation}
which are obtained as quotients of the fused Hecke algebras:
\[\CC=H_{\bk,0}(q)\subset H_{\bk,1}(q)\subset H_{\bk,2}(q)\subset\dots\dots \subset H_{\bk,n}(q)\subset H_{\bk,n+1}(q)\subset \dots\dots\ .\]

From Theorem \ref{theo-quot-rep}, we obtain the following corollary to decide when the centraliser coincides with the fused Hecke algebras. We note that the statement is the same as (\ref{comp-SW}) in the classical Schur--Weyl duality relating the centraliser to the Hecke algebra.
\begin{coro}
The centraliser $\oH^N_{\bk,n}(q)$ coincides with the algebra $H_{\bk,n}(q)$ if and only if $n\leq N$. 
\end{coro}
\begin{proof}
Note that for every $\bk$ and $n$, there is a partition $\lambda\in S_{\bk,n}$ with $l(\lambda)=n$. One can take $\lambda=\bk_{\vert n}^{\text{ord}}$ (it is of length $n$ since the part of $(k_1,\dots,k_n)$ are non-zero). Then item 1 of Theorem \ref{theo-quot-rep} implies immediately that the ideal $I_{\bk,n}^N$ is non-zero if and only if $N>n$, and thus the corollary follows.
\end{proof}

\subsection{Decomposition of tensor products $L_{(k_1)}^N\otimes\dots\otimes L_{(k_n)}^N$}

We can combine what we have obtained so far to deduce immediately the decomposition of the representation $L_{(k_1)}^N\otimes\dots\otimes L_{(k_n)}^N$ of $U_q(gl_N)$, known as the Pieri rule, a particular case of the Littlewood--Richardson rule. We use Theorem \ref{theo-SWk}, Theorem \ref{thm-rep} item 2 and Theorem \ref{theo-quot-rep} item 1. We obtain the analogue of the full statement of the Schur--Weyl duality (see Theorem \ref{theo-SW}):

As a $U_q(gl_N)\otimes H_{\bk,n}(q)$-module, the space $L_{(k_1)}^N\otimes\dots\otimes L_{(k_n)}^N$ decomposes as follows:
\begin{equation}\label{SWk}
L_{(k_1)}^N\otimes\dots\otimes L_{(k_n)}^N=\bigoplus_{\begin{array}{c}
\\[-1.6em]
\scriptstyle{\lambda\in S_{\bk,n}} \\[-0.4em]
\scriptstyle{l(\lambda)\leq N}
\end{array}} L^N_{\lambda}\otimes W_{\bk,\lambda}\ ,
\end{equation}
where we recall that the set $S_{\bk,n}=\{\lambda\vdash k_1+\dots+k_n\ |\ \lambda\geq \bk_{\vert n}^{\text{ord}}\}$ parametrises the irreducible representations of $H_{\bk,n}(q)$ and $W_{\bk,\lambda}$ is the corresponding irreducible representation constructed in Theorem \ref{thm-rep}. The partition $\bk_{\vert n}^{\text{ord}}$ is obtained from $\bk_{\vert n}=(k_1,\dots,k_n)$ by ordering the parts in decreasing order.

In particular, adding the information on the dimension of $W_{\bk,\lambda}$, we obtain that as a $U_q(gl_N)$-module:
\[L_{(k_1)}^N\otimes\dots\otimes L_{(k_n)}^N=\bigoplus_{\begin{array}{c}
\\[-1.6em]
\scriptstyle{\lambda\in S_{\bk,n}} \\[-0.4em]
\scriptstyle{l(\lambda)\leq N}
\end{array}} \bigl(L^N_{\lambda}\bigr)^{\oplus K_{\lambda,\bk_{\vert n}}}\ ,
\]
where $K_{\lambda,\bk_{\vert n}}$ is the Kostka number counting the number of semistandard Young tableaux of shape $\lambda$ and of weight $(k_1,\dots,k_n)$.

\section{Algebraic description of the centralisers $\oH^N_{\bk,n}(q)$}\label{sec-quot}

The centraliser $\oH^N_{\bk,n}(q)$ is described as the quotient $\oH^N_{\bk,n}(q)=H_{\bk,n}(q)/I_{\bk,n}^N$ and the corresponding ideal $I_{\bk,n}^N$ was described in the preceding section in terms of the representation theory of $H_{\bk,n}(q)$. In this section, we aim at an algebraic description of $I_{\bk,n}^N$.

In this section, we assume that the sequence $\bk=(k_1,k_2,\dots)$ does not contain $0$, and that it is already in decreasing order.

\paragraph{Generalisation of the $q$-antisymmetriser.} We define an element $H_{\bk,n}(q)$ by a simple diagrammatic procedure. For convenience, we give first the definition in the situation $q^2=1$ and then treat the general case (examples are given below). We will give an equivalent more algebraic definition just after the examples.

Let $w\in \mS_n$. We define a fused permutation denoted $|w|_{\bk}$ by the following procedure. Start from the permutation diagram of $w$ and add vertical edges (if necessary) at each ellipse to form the diagram of a fused permutation corresponding to $\bk$. More precisely, for each $a\in\{1,\dots,n\}$, we add $k_a-1$ vertical edges connecting the $a$-th top ellipse to the $a$-th bottom ellipse. Then we set:
\begin{equation}\label{sym1}
\AS_{\bk,n}(1)=\sum_{w\in\mS_n}(-1)^{\ell(w)}|w|_{\bk}\ \in H_{\bk,n}(1)\ .
\end{equation}

Similarly, we define an element $|\si_w|_{\bk}\in H_{\bk,n}(q)$ as follows. We start with the braid diagram of $\si_w$ as defined in Section \ref{sec-fus-br}. We promote all dots into ellipses, and for each $a\in\{1,\dots,n\}$, we add $k_a-1$ vertical edges connecting the $a$-th top ellipse to the $a$-th bottom ellipse. The rule is a follows: At each ellipse, the new strands are attached to the right of the one strand already present; the added strands do not cross each other; the new strands are ``above'' the original ones forming $\si_w$ (above in the natural sense, as shown in the examples below). 
Then we set:
\begin{equation}\label{symq}
\AS_{\bk,n}(q)=\sum_{w\in\mS_n}(-q^{-1})^{\ell(w)}|\si_w|_{\bk}\ \in H_{\bk,n}(q)\ .
\end{equation}

\begin{exa} Let $n=3$ and $\bk=(2,2,2,\dots)$. Here is depicted the procedure to obtain $\AS_{\bk,n}(1)$:
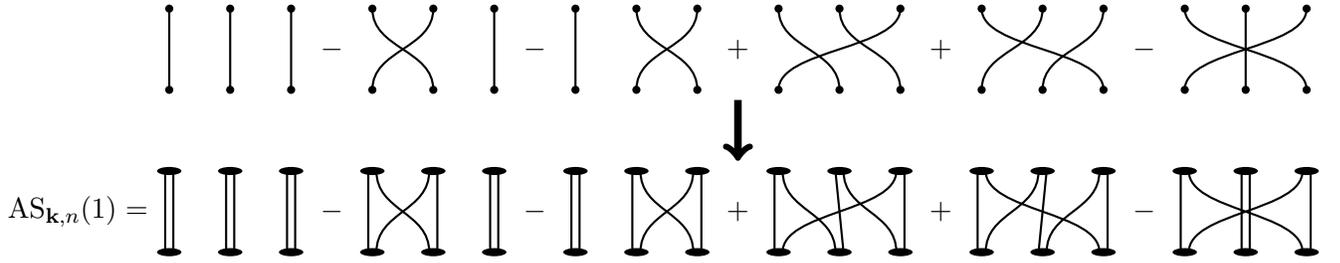
\begin{center}
 \begin{tikzpicture}[scale=0.27]
\fill (1,12) circle (0.2cm);\fill (1,8) circle (0.2cm);
\draw[thick] (1,12) -- (1,8);
\fill (4,12) circle (0.2cm);\fill (4,8) circle (0.2cm);
\draw[thick] (4,12) -- (4,8);
\fill (7,12) circle (0.2cm);\fill (7,8) circle (0.2cm);
\draw[thick] (7,12) -- (7,8);
\node at (9,10) {$-$};
\fill (11,12) circle (0.2cm);\fill (11,8) circle (0.2cm);
\draw[thick] (11,12)..controls +(0,-2) and +(0,+2) .. (14,8);
\fill (14,12) circle (0.2cm);\fill (14,8) circle (0.2cm);
\draw[thick] (14,12)..controls +(0,-2) and +(0,+2) .. (11,8);
\fill (17,12) circle (0.2cm);\fill (17,8) circle (0.2cm);
\draw[thick] (17,12) -- (17,8);
\node at (19,10) {$-$};
\fill (21,12) circle (0.2cm);\fill (21,8) circle (0.2cm);
\draw[thick] (21,12) -- (21,8);
\fill (24,12) circle (0.2cm);\fill (24,8) circle (0.2cm);
\draw[thick] (27,12)..controls +(0,-2) and +(0,+2) .. (24,8);
\fill (27,12) circle (0.2cm);\fill (27,8) circle (0.2cm);
\draw[thick] (24,12)..controls +(0,-2) and +(0,+2) .. (27,8);
\node at (29,10) {$+$};
\fill (31,12) circle (0.2cm);\fill (31,8) circle (0.2cm);
\draw[thick] (31,12)..controls +(0,-2) and +(0,+2) .. (34,8);
\fill (34,12) circle (0.2cm);\fill (34,8) circle (0.2cm);
\draw[thick] (34,12)..controls +(0,-2) and +(0,+2) .. (37,8);
\fill (37,12) circle (0.2cm);\fill (37,8) circle (0.2cm);
\draw[thick] (37,12)..controls +(0,-2) and +(0,+2) .. (31,8);
\node at (39,10) {$+$};
\fill (41,12) circle (0.2cm);\fill (41,8) circle (0.2cm);
\draw[thick] (41,12)..controls +(0,-2) and +(0,+2) .. (47,8);
\fill (44,12) circle (0.2cm);\fill (44,8) circle (0.2cm);
\draw[thick] (44,12)..controls +(0,-2) and +(0,+2) .. (41,8);
\fill (47,12) circle (0.2cm);\fill (47,8) circle (0.2cm);
\draw[thick] (47,12)..controls +(0,-2) and +(0,+2) .. (44,8);
\node at (49,10) {$-$};
\fill (51,12) circle (0.2cm);\fill (51,8) circle (0.2cm);
\draw[thick] (51,12)..controls +(0,-2) and +(0,+2) .. (57,8);
\fill (54,12) circle (0.2cm);\fill (54,8) circle (0.2cm);
\draw[thick] (54,12) -- (54,8);
\fill (57,12) circle (0.2cm);\fill (57,8) circle (0.2cm);
\draw[thick] (57,12)..controls +(0,-2) and +(0,+2) .. (51,8);
 
\draw[line width=1mm,->] (29,7.5) -- (29,4.5); 
 
\node at (-3.5,2) {$\AS_{\bk,n}(1)=$};
\fill (1,4) ellipse (0.6cm and 0.2cm);\fill (1,0) ellipse (0.6cm and 0.2cm);
\draw[thick] (0.8,4) -- (0.8,0);\draw[thick] (1.2,4) -- (1.2,0);
\fill (4,4) ellipse (0.6cm and 0.2cm);\fill (4,0) ellipse (0.6cm and 0.2cm);
\draw[thick] (3.8,4) -- (3.8,0);\draw[thick] (4.2,4) -- (4.2,0);
\fill (7,4) ellipse (0.6cm and 0.2cm);\fill (7,0) ellipse (0.6cm and 0.2cm);
\draw[thick] (6.8,4) -- (6.8,0);\draw[thick] (7.2,4) -- (7.2,0);
\node at (9,2) {$-$};
\fill (11,4) ellipse (0.6cm and 0.2cm);\fill (11,0) ellipse (0.6cm and 0.2cm);
\draw[thick] (10.8,4) -- (10.8,0);\draw[thick] (11.2,4)..controls +(0,-2) and +(0,+2) .. (13.8,0);
\fill (14,4) ellipse (0.6cm and 0.2cm);\fill (14,0) ellipse (0.6cm and 0.2cm);
\draw[thick] (13.8,4)..controls +(0,-2) and +(0,+2) .. (11.2,0);\draw[thick] (14.2,4) -- (14.2,0);
\fill (17,4) ellipse (0.6cm and 0.2cm);\fill (17,0) ellipse (0.6cm and 0.2cm);
\draw[thick] (16.8,4) -- (16.8,0);\draw[thick] (17.2,4) -- (17.2,0);
\node at (19,2) {$-$};
\fill (21,4) ellipse (0.6cm and 0.2cm);\fill (21,0) ellipse (0.6cm and 0.2cm);
\draw[thick] (20.8,4) -- (20.8,0);\draw[thick] (21.2,4) -- (21.2,0);
\fill (24,4) ellipse (0.6cm and 0.2cm);\fill (24,0) ellipse (0.6cm and 0.2cm);
\draw[thick] (23.8,4) -- (23.8,0);\draw[thick] (26.8,4)..controls +(0,-2) and +(0,+2) .. (24.2,0);
\fill (27,4) ellipse (0.6cm and 0.2cm);\fill (27,0) ellipse (0.6cm and 0.2cm);
\draw[thick] (24.2,4)..controls +(0,-2) and +(0,+2) .. (26.8,0);\draw[thick] (27.2,4) -- (27.2,0);
\node at (29,2) {$+$};
\fill (31,4) ellipse (0.6cm and 0.2cm);\fill (31,0) ellipse (0.6cm and 0.2cm);
\draw[thick] (30.8,4) -- (30.8,0);\draw[thick] (31.2,4)..controls +(0,-2) and +(0,+2) .. (33.8,0);
\fill (34,4) ellipse (0.6cm and 0.2cm);\fill (34,0) ellipse (0.6cm and 0.2cm);
\draw[thick] (33.8,4) -- (34.2,0);\draw[thick] (34.2,4)..controls +(0,-2) and +(0,+2) .. (36.8,0);
\fill (37,4) ellipse (0.6cm and 0.2cm);\fill (37,0) ellipse (0.6cm and 0.2cm);
\draw[thick] (36.8,4)..controls +(0,-2) and +(0,+2) .. (31.2,0);\draw[thick] (37.2,4) -- (37.2,0);
\node at (39,2) {$+$};
\fill (41,4) ellipse (0.6cm and 0.2cm);\fill (41,0) ellipse (0.6cm and 0.2cm);
\draw[thick] (40.8,4) -- (40.8,0);\draw[thick] (41.2,4)..controls +(0,-2) and +(0,+2) .. (46.8,0);
\fill (44,4) ellipse (0.6cm and 0.2cm);\fill (44,0) ellipse (0.6cm and 0.2cm);
\draw[thick] (43.8,4)..controls +(0,-2) and +(0,+2) .. (41.2,0);\draw[thick] (44.2,4) -- (43.8,0);
\fill (47,4) ellipse (0.6cm and 0.2cm);\fill (47,0) ellipse (0.6cm and 0.2cm);
\draw[thick] (46.8,4)..controls +(0,-2) and +(0,+2) .. (44.2,0);\draw[thick] (47.2,4) -- (47.2,0);
\node at (49,2) {$-$};
\fill (51,4) ellipse (0.6cm and 0.2cm);\fill (51,0) ellipse (0.6cm and 0.2cm);
\draw[thick] (50.8,4) -- (50.8,0);\draw[thick] (51.2,4)..controls +(0,-2) and +(0,+2) .. (56.8,0);
\fill (54,4) ellipse (0.6cm and 0.2cm);\fill (54,0) ellipse (0.6cm and 0.2cm);
\draw[thick] (53.8,4) -- (53.8,0);\draw[thick] (54.2,4) -- (54.2,0);
\fill (57,4) ellipse (0.6cm and 0.2cm);\fill (57,0) ellipse (0.6cm and 0.2cm);
\draw[thick] (56.8,4)..controls +(0,-2) and +(0,+2) .. (51.2,0);\draw[thick] (57.2,4) -- (57.2,0);
\end{tikzpicture}
\end{center}
Here is depicted the procedure to obtain to obtain $\AS_{\bk,n}(q)$:
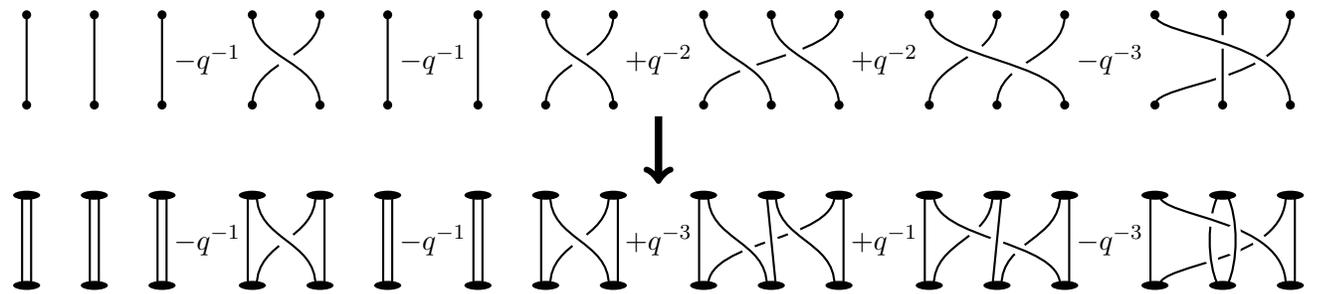
\begin{center}
 \begin{tikzpicture}[scale=0.3]
\fill (1,12) circle (0.2cm);\fill (1,8) circle (0.2cm);
\draw[thick] (1,12) -- (1,8);
\fill (4,12) circle (0.2cm);\fill (4,8) circle (0.2cm);
\draw[thick] (4,12) -- (4,8);
\fill (7,12) circle (0.2cm);\fill (7,8) circle (0.2cm);
\draw[thick] (7,12) -- (7,8);
\node at (9,10) {$-q^{-1}$};
\fill (11,12) circle (0.2cm);\fill (11,8) circle (0.2cm);
\draw[thick] (14,12)..controls +(0,-2) and +(0,+2) .. (11,8);\fill[white] (12.5,10) circle (0.4);
\draw[thick] (11,12)..controls +(0,-2) and +(0,+2) .. (14,8);
\fill (14,12) circle (0.2cm);\fill (14,8) circle (0.2cm);
\fill (17,12) circle (0.2cm);\fill (17,8) circle (0.2cm);
\draw[thick] (17,12) -- (17,8);
\node at (19,10) {$-q^{-1}$};
\fill (21,12) circle (0.2cm);\fill (21,8) circle (0.2cm);
\draw[thick] (21,12) -- (21,8);
\fill (24,12) circle (0.2cm);\fill (24,8) circle (0.2cm);
\draw[thick] (27,12)..controls +(0,-2) and +(0,+2) .. (24,8);\fill[white] (25.5,10) circle (0.4);
\fill (27,12) circle (0.2cm);\fill (27,8) circle (0.2cm);
\draw[thick] (24,12)..controls +(0,-2) and +(0,+2) .. (27,8);
\node at (29,10) {$+q^{-2}$};
\draw[thick] (37,12)..controls +(0,-2) and +(0,+2) .. (31,8);\fill[white] (33,9.6) circle (0.4);\fill[white] (35,10.4) circle (0.4);
\fill (31,12) circle (0.2cm);\fill (31,8) circle (0.2cm);
\draw[thick] (31,12)..controls +(0,-2) and +(0,+2) .. (34,8);
\fill (34,12) circle (0.2cm);\fill (34,8) circle (0.2cm);
\draw[thick] (34,12)..controls +(0,-2) and +(0,+2) .. (37,8);
\fill (37,12) circle (0.2cm);\fill (37,8) circle (0.2cm);
\node at (39,10) {$+q^{-2}$};
\fill (41,12) circle (0.2cm);\fill (41,8) circle (0.2cm);
\draw[thick] (47,12)..controls +(0,-2) and +(0,+2) .. (44,8);
\draw[thick] (44,12)..controls +(0,-2) and +(0,+2) .. (41,8);
\fill[white] (43,10.4) circle (0.4);\fill[white] (45,9.6) circle (0.4);
\draw[thick] (41,12)..controls +(0,-2) and +(0,+2) .. (47,8);
\fill (44,12) circle (0.2cm);\fill (44,8) circle (0.2cm);
\fill (47,12) circle (0.2cm);\fill (47,8) circle (0.2cm);
\node at (49,10) {$-q^{-3}$};
\fill (51,12) circle (0.2cm);\fill (51,8) circle (0.2cm);
\fill (54,12) circle (0.2cm);\fill (54,8) circle (0.2cm);
\fill (57,12) circle (0.2cm);\fill (57,8) circle (0.2cm);
\draw[thick] (57,12)..controls +(0,-3) and +(0,+1) .. (51,8);
\fill[white] (54,9.2) circle (0.3);
\draw[thick] (54,12) -- (54,8);
\fill[white] (54,10.8) circle (0.3);\fill[white] (55.7,10) circle (0.3);
\draw[thick] (51,12)..controls +(0,-1) and +(0,+3) .. (57,8);
 
\draw[line width=1mm,->] (29,7.5) -- (29,4.5); 
 
\fill (1,4) ellipse (0.6cm and 0.2cm);\fill (1,0) ellipse (0.6cm and 0.2cm);
\draw[thick] (0.8,4) -- (0.8,0);\draw[thick] (1.2,4) -- (1.2,0);
\fill (4,4) ellipse (0.6cm and 0.2cm);\fill (4,0) ellipse (0.6cm and 0.2cm);
\draw[thick] (3.8,4) -- (3.8,0);\draw[thick] (4.2,4) -- (4.2,0);
\fill (7,4) ellipse (0.6cm and 0.2cm);\fill (7,0) ellipse (0.6cm and 0.2cm);
\draw[thick] (6.8,4) -- (6.8,0);\draw[thick] (7.2,4) -- (7.2,0);
\node at (9,2) {$-q^{-1}$};
\fill (11,4) ellipse (0.6cm and 0.2cm);\fill (11,0) ellipse (0.6cm and 0.2cm);!
\draw[thick] (10.8,4) -- (10.8,0);\draw[thick] (13.8,4)..controls +(0,-2) and +(0,+2) .. (11.2,0);\fill[white] (12.5,2) circle (0.4);
\fill (14,4) ellipse (0.6cm and 0.2cm);\fill (14,0) ellipse (0.6cm and 0.2cm);
\draw[thick] (14.2,4) -- (14.2,0);\draw[thick] (11.2,4)..controls +(0,-2) and +(0,+2) .. (13.8,0);
\fill (17,4) ellipse (0.6cm and 0.2cm);\fill (17,0) ellipse (0.6cm and 0.2cm);
\draw[thick] (16.8,4) -- (16.8,0);\draw[thick] (17.2,4) -- (17.2,0);
\node at (19,2) {$-q^{-1}$};
\fill (21,4) ellipse (0.6cm and 0.2cm);\fill (21,0) ellipse (0.6cm and 0.2cm);
\draw[thick] (20.8,4) -- (20.8,0);\draw[thick] (21.2,4) -- (21.2,0);
\fill (24,4) ellipse (0.6cm and 0.2cm);\fill (24,0) ellipse (0.6cm and 0.2cm);
\draw[thick] (23.8,4) -- (23.8,0);\draw[thick] (26.8,4)..controls +(0,-2) and +(0,+2) .. (24.2,0);\fill[white] (25.5,2) circle (0.4);
\fill (27,4) ellipse (0.6cm and 0.2cm);\fill (27,0) ellipse (0.6cm and 0.2cm);
\draw[thick] (24.2,4)..controls +(0,-2) and +(0,+2) .. (26.8,0);\draw[thick] (27.2,4) -- (27.2,0);
\node at (29,2) {$+q^{-3}$};
\draw[thick] (36.8,4)..controls +(0,-2) and +(0,+2) .. (31.2,0);\draw[thick] (37.2,4) -- (37.2,0);
\fill[white] (33,1.6) circle (0.3);\fill[white] (35,2.4) circle (0.3);;\fill[white] (34,2) circle (0.3);
\fill (31,4) ellipse (0.6cm and 0.2cm);\fill (31,0) ellipse (0.6cm and 0.2cm);
\draw[thick] (30.8,4) -- (30.8,0);\draw[thick] (31.2,4)..controls +(0,-2) and +(0,+2) .. (33.8,0);
\fill (34,4) ellipse (0.6cm and 0.2cm);\fill (34,0) ellipse (0.6cm and 0.2cm);
\draw[thick] (33.8,4) -- (34.2,0);\draw[thick] (34.2,4)..controls +(0,-2) and +(0,+2) .. (36.8,0);
\fill (37,4) ellipse (0.6cm and 0.2cm);\fill (37,0) ellipse (0.6cm and 0.2cm);
\node at (39,2) {$+q^{-1}$};
\draw[thick] (46.8,4)..controls +(0,-2) and +(0,+2) .. (44.2,0);
\draw[thick] (43.8,4)..controls +(0,-2) and +(0,+2) .. (41.2,0);
\fill[white] (43,2.4) circle (0.3);\fill[white] (45,1.6) circle (0.3);
\draw[thick] (41.2,4)..controls +(0,-2) and +(0,+2) .. (46.8,0);
\fill[white] (44,2) circle (0.3);
\fill (41,4) ellipse (0.6cm and 0.2cm);\fill (41,0) ellipse (0.6cm and 0.2cm);
\draw[thick] (40.8,4) -- (40.8,0);
\fill (44,4) ellipse (0.6cm and 0.2cm);\fill (44,0) ellipse (0.6cm and 0.2cm);
\draw[thick] (44.2,4) -- (43.8,0);
\fill (47,4) ellipse (0.6cm and 0.2cm);\fill (47,0) ellipse (0.6cm and 0.2cm);
\draw[thick] (47.2,4) -- (47.2,0);
\node at (49,2) {$-q^{-3}$};
\draw[thick] (56.8,4)..controls +(0,-3) and +(0,+1) .. (51.2,0);
\fill[white] (53.5,1.2) circle (0.3);\fill[white] (54.5,1.6) circle (0.3);
\draw[thick] (53.8,4)..controls +(-0.5,-1) and +(-0.5,1) .. (53.8,0);
\fill[white] (53.5,2.9) circle (0.3);
\fill[white] (54,2.8) circle (0.3);\fill[white] (55.6,2) circle (0.3);
\draw[thick] (51.2,4)..controls +(0,-1) and +(0,+3) .. (56.8,0);
\fill[white] (54.5,2.4) circle (0.3);
\draw[thick] (54.2,4)..controls +(0.5,-1) and +(0.5,1) .. (54.2,0);
\draw[thick] (50.8,4) -- (50.8,0);
\fill (51,4) ellipse (0.6cm and 0.2cm);\fill (51,0) ellipse (0.6cm and 0.2cm);
\fill (54,4) ellipse (0.6cm and 0.2cm);\fill (54,0) ellipse (0.6cm and 0.2cm);
\fill (57,4) ellipse (0.6cm and 0.2cm);\fill (57,0) ellipse (0.6cm and 0.2cm);
\draw[thick] (57.2,4) -- (57.2,0);
\end{tikzpicture}
\end{center}
Note that the added vertical strands are indeed above all others. Initially, we attached them at each ellipse to the right of the existing strands, but in the above picture, we used some idempotent relations to suppress some crossings near the ellipses. This accounts for the modifications in the powers of $q$. Note that, at the end, the coefficient is always equal to $(-1)^{\ell(w)}$ times $q^{-1}$ to the power the sum of the crossings in the diagram (counted with signs).
\end{exa}

\paragraph{Algebraic interpretation.} Consider the following element of the usual Hecke algebra $H_{k_1+\dots+k_n}(q)$:
\begin{equation}\label{Gamma}
\Gamma=\si_{k_1}\dots\si_2\cdot \si_{k_1+k_2}\dots \si_3 \cdot \ldots\ldots\cdot \si_{k_1+\dots+k_{n-1}}\dots\si_n\ , 
\end{equation}
where the dots between $\si_{k_1+\dots+k_{a-1}}$ and $\si_a$ indicate the product of the generators in decreasing order of their indices (note that $\Gamma=1$ only if $\bk=(1,1,1,\dots)$). Here is the diagrammatic representation of $\Gamma$ (in the case $n=3$):
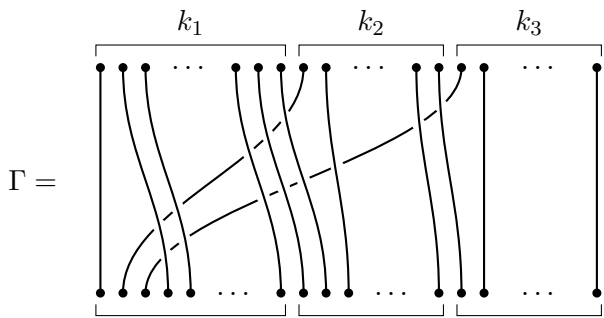
\begin{center}
 \begin{tikzpicture}[scale=0.3]
\node at (-2,0) {$\Gamma=$};
\draw (0.8,5.5) -- (0.8,6) -- (9.2,6) -- (9.2,5.5);\node at (5,7) {$k_1$};
\fill (1,5) circle (0.2);\fill (1,-5) circle (0.2);
\fill (2,5) circle (0.2);\fill (2,-5) circle (0.2);
\fill (3,5) circle (0.2);\fill (3,-5) circle (0.2);
\node at (5,5) {$\dots$};
\fill (4,-5) circle (0.2);
\fill (5,-5) circle (0.2);\node at (7,-5) {$\dots$};
\fill (7,5)  circle (0.2);
\fill (8,5)  circle (0.2);
\fill (9,5) circle (0.2);\fill (9,-5) circle (0.2);
\draw (0.8,-5.5) -- (0.8,-6) -- (9.2,-6) -- (9.2,-5.5);

\draw (9.8,5.5) -- (9.8,6) -- (16.2,6) -- (16.2,5.5);\node at (13,7) {$k_2$};
\fill (10,5) circle (0.2);\fill (10,-5) circle (0.2);
\fill (11,5) circle (0.2);\fill (11,-5) circle (0.2);
\node at (13,5) {$\dots$};\fill (12,-5) circle (0.2);
\node at (14,-5) {$\dots$};
\fill (15,5) circle (0.2);
\fill (16,5) circle (0.2);\fill (16,-5) circle (0.2);
\draw (9.8,-5.5) -- (9.8,-6) -- (16.2,-6) -- (16.2,-5.5);

\draw (16.8,5.5) -- (16.8,6) -- (23.2,6) -- (23.2,5.5);\node at (20,7) {$k_3$};
\fill (17,5) circle (0.2);\fill (17,-5) circle (0.2);
\fill (18,5) circle (0.2);\fill (18,-5) circle (0.2);
\node at (20.5,5) {$\dots$};\node at (20.5,-5) {$\dots$};
\fill (23,5) circle (0.2);\fill (23,-5) circle (0.2);
\draw (16.8,-5.5) -- (16.8,-6) -- (23.2,-6) -- (23.2,-5.5);

\draw[thick] (1,5) -- (1,-5);
\draw[thick] (10,5)..controls +(0,-4) and +(0,+4) .. (2,-5);
\draw[thick] (17,5)..controls +(0,-4) and +(0,+4) .. (3,-5);

\fill[white] (9.2,2.8) circle (0.3);\fill[white] (8.4,2) circle
(0.3);\fill[white] (7.6,1.2) circle (0.3);
\fill[white] (4.4,-1.2) circle (0.3);\fill[white] (3.6,-2) circle
(0.3);\fill[white] (4.7,-2.6) circle (0.3);\fill[white] (3.9,-3.3)
circle (0.3);
\fill[white] (16.1,3.3) circle (0.3);\fill[white] (15.1,2.5) circle (0.3);
\fill[white] (11.5,0.6) circle (0.3);\fill[white] (10,0) circle (0.3);
\fill[white] (9.1,-0.3) circle (0.3);\fill[white] (8.2,-0.7) circle (0.3);

\draw[thick] (2,5)..controls +(0,-4) and +(0,+4) .. (4,-5);
\draw[thick] (3,5)..controls +(0,-4) and +(0,+4) .. (5,-5);
\draw[thick] (7,5)..controls +(0,-4) and +(0,+4) .. (9,-5);
\draw[thick] (8,5)..controls +(0,-4) and +(0,+4) .. (10,-5);
\draw[thick] (9,5)..controls +(0,-4) and +(0,+4) .. (11,-5);
\draw[thick] (11,5)..controls +(0,-4) and +(0,+4) .. (12,-5);
\draw[thick] (15,5)..controls +(0,-4) and +(0,+4) .. (16,-5);
\draw[thick] (16,5)..controls +(0,-4) and +(0,+4) .. (17,-5);
\draw[thick] (18,5)..controls +(0,-4) and +(0,+4) .. (18,-5);
\draw[thick] (23,5)..controls +(0,-4) and +(0,+4) .. (23,-5);

\end{tikzpicture}
\end{center}

Let $w\in \mS_n$. Under the isomorphism between $H_{\bk,n}$ and $P_{\bk,n}H_{k_1+\dots+k_n}(q)P_{\bk,n}$, we claim that the element $|\si_w|_{\bk}$ corresponds to
\[P_{\bk,n}\Gamma \si_w \Gamma^{-1} P_{\bk,n}\ ,\]
where $\si_w$ is seen as an element of $H_{k_1+\dots+k_n}(q)$ by the natural inclusion (it involves only the first $n$ strands).

To check this claim, we consider braid diagrams in $H_{k_1+\dots+k_n}(q)$ and we see the lines of $k_1+\dots+k_n$ dots as $n$ packets of $k_1,\dots,k_n$ dots respectively (as shown in the above picture). Then, for each $a\in\{1,\dots,n\}$, the element $\si_{k_1+\dots+k_{a-1}}\dots\si_a$ does the following: it takes the first strand of the $a$-th packet and move it in position $a$ by passing below the strands to its left. So after application of $\Gamma$ the first strands of the $n$ packets have become the $n$ first strands. On these $n$ first strands, we apply $\si_w$ and then we move back these $n$ strands in their original position using $\Gamma^{-1}$. We see that these $n$ strands always stay below the other ones, and the other strands never cross each other and end up being vertical. So, gluing the packets of dots into ellipses, we obtain exactly the element $|\si_w|_{\bk}$ in its diagrammatical definition.

\subsection{Conjectural description of the ideal $I_{\bk,n}^N$.}

Now we are ready to present a conjectural description of the ideal $I_{\bk,n}^N$ resulting in an algebraic description of the centralisers $\oH^N_{\bk,n}(q)$. Recall that if $n\leq N$, there is nothing to do since $\oH^N_{\bk,n}(q)$ simply coincides with $H_{\bk,n}(q)$.

We make the following two conjectures. They generalise for example the description of the Temperley--Lieb algebra as a quotient of the Hecke algebra, see Example \ref{exa-TL}. Below we see the element $\AS_{\bk,N+1}(q)$ as an element of $H_{\bk,n}(q)$ for any $n\geq N+1$, by the natural inclusion of algebras (namely, in $H_{\bk,n}(q)$, the element $\AS_{\bk,N+1}(q)$ involves only the strands attached to the $N+1$ first ellipses).
\begin{conj}\label{conj1}
Let $n>N$. The algebra $\oH^N_{\bk,n}(q)$ is isomorphic to the quotient of the algebra $H_{\bk,n}(q)$ over the relation:
\[\AS_{\bk,N+1}(q)=0\ .\]
\end{conj}
\begin{conj}\label{conj2}
The element $\AS_{\bk,N+1}(q)$ is central in $H_{\bk,N+1}(q)$.
\end{conj}
From our description of the representations of the chain of algebras $\oH^N_{\bk,n}(q)$ in Theorem \ref{theo-quot-rep} (in particular, item 3), we have that in order to prove Conjecture \ref{conj1}, we need only to prove that $\AS_{\bk,N+1}(q)$ generates the ideal $I_{\bk,N+1}^N$ of $H_{\bk,N+1}(q)$. In other words, we need only consider the case $n=N+1$.

As a first step towards Conjecture \ref{conj1}, we check that the element $\AS_{\bk,N+1}(q)$ belongs to the correct ideal.
\begin{prop}\label{prop-conj}
If $n>N$, the element $\AS_{\bk,N+1}(q)$ belongs to the ideal $I_{\bk,n}^N$ of $H_{\bk,n}(q)$.
\end{prop}
\begin{proof}
From the algebraic description of the elements $|\si_w|_{\bk}$ given above, we have that under the isomorphism between $H_{\bk,n}$ and $P_{\bk,n}H_{k_1+\dots+k_{n}}(q)P_{\bk,n}$ the element $\AS_{\bk,N+1}(q)$ corresponds to
\[P_{\bk,n}\Gamma \Bigl(\sum_{w\in\mS_{N+1}}(-q^{-1})^{\ell(w)}\si_w\Bigr) \Gamma^{-1} P_{\bk,n}\ ,\]
where $\Gamma$ was given above, see (\ref{Gamma}). The element in the middle is the $q$-antisymmetriser on $N+1$ strands so, as recalled in Section \ref{sec-SW}, it belongs to the kernel of the representation of $H_{k_1+\dots+k_{n}}(q)$ on the tensor power of $L_{(1)}^N$. Thus, the element $\AS_{\bk,N+1}(q)$ belongs to the kernel of the representation of $H_{\bk,n}(q)$ on $L_{(k_1)}^N\otimes \dots\otimes L_{(k_n)}^N$ given in Theorem \ref{theo-SWk} (see Formula \eqref{act-PHP}).
\end{proof}

After this result, for Conjecture \ref{conj1}, it remains only to show that $\AS_{\bk,N+1}(q)$ generates in $H_{\bk,N+1}(q)$ an ideal of the correct dimension. In fact, combining Theorem \ref{theo-quot-rep}, item 3, with Proposition \ref{prop-subal-k}, we must show that the ideal generated by $\AS_{\bk,N+1}(q)$ is at least of dimension $\dim\bigl(H_{\bk-1,N+1}(q)\bigr)$.

At this point, by a deformation argument (see the proof of the special cases below), Conjecture \ref{conj1} can be reduced at least for generic $q$ to the situation $q^2=1$. The situation $q^2=1$ can be attacked by a combinatorial approach.

Both Conjectures are supported by their verifications in some special cases below. Also, explicit (computer-aided) calculations have allowed to check their validity for all $\bk$ and $N$ such that $k_1+\dots+k_{N+1}\leq 7$. 

\subsection{Verification in some cases}

In the statements below, by generic $q$ we mean that the statement is valid in the situation where $q$ is an indeterminate (see Remark \ref{rem-def}). One can also understand that it is valid for all but a finite number of values of $q$.
\begin{prop}\label{verif-conj}$\ $
\begin{enumerate}
\item If $\bk=(k,1,1,1,\dots)$ with $k$ arbitrary, then Conjectures \ref{conj1}-\ref{conj2} are true for any $N$.
\item If $N=2$, then for any $\bk$ Conjecture \ref{conj1} is true for $q^2=1$ and for $q$ generic.
\item If $\bk$ consists only of $1$'s and $2$'s, then Conjecture \ref{conj1} is true for $q^2=1$ and for $q$ generic for any $N$.
\item If $\bk$ consists only of $1$'s and $2$'s then Conjecture \ref{conj2} is true for $q^2=1$ for any $N$.
\end{enumerate}
\end{prop}
\begin{proof}
\textbf{1.} If $\bk=(k,1,1,1,\dots)$ then the ideal $I_{\bk,N+1}^N$ is of dimension 1 (the dimension of $H_{\bk-1,n}(q)$ for any $n\geq 0$; one can also see easily that there is a single Young semistandard tableau of size $k+N$ filled with $1,\dots,1$($k$ times)$,2,\dots,N+1$: the one of shape a hook with a first line of $k$ boxes). Thus there is nothing to check for the generation of the ideal by $AS_{\bk,N+1}$, and neither for the centrality of $AS_{\bk,N+1}$ since a generator of a one-dimensional ideal in a semisimple algebra must be a (minimal) central idempotent.

\vskip .2cm
\textbf{2.} In the situation of a generic $q$, the element $AS_{\bk,N+1}(q)$ specialises for $q^2=1$ to the element $AS_{\bk,N+1}(1)$. Moreover, in an irreducible representation of $H_{\bk,n}(q)$ we can specialise $q^2=1$ to obtain the corresponding representation of $H_{\bk,n}(1)$. More precisely, recall that the irreducible representations of $H_{\bk,n}(q)$ are obtained as $P_{\bk,n}(V)$ for some irreducible representations $V$ of the Hecke algebra $H_{k_1+\dots+k_n}(q)$. Using the explicit realisation (the seminormal form) of representations of the Hecke algebras given in Section \ref{sec-SW}, we see that we can specialise $q^2=1$. 

So now assume that in some irreducible representation of $H_{\bk,n}(q)$ the element $AS_{\bk,N+1}(q)$ is $0$. By specialisation, this implies that the element $AS_{\bk,N+1}(1)$ is $0$ in the corresponding representation of $AS_{\bk,N+1}(1)$. Therefore, identifying $H_{\bk,n}(q)$ and $H_{\bk,n}(1)$ through their Artin--Wedderburn decomposition, it means that the ideal of $H_{\bk,n}(q)$ generated by $AS_{\bk,N+1}(q)$ contains the ideal of $H_{\bk,n}(1)$ generated by $AS_{\bk,N+1}(1)$. In particular, for generic $q$, we have:
\[\dim\Bigl(H_{\bk,n}(q)AS_{\bk,N+1}(q)H_{\bk,n}(q)\Bigr)\geq \dim \Bigl(H_{\bk,n}(1)AS_{\bk,N+1}(1)H_{\bk,n}(1)\Bigr)\ .\]
Moreover, we explained after Proposition \ref{prop-conj} that in order to verify Conjecture \ref{conj1}, it remains to show that the ideal generated by $\AS_{\bk,N+1}(q)$ in $H_{\bk,N+1}(q)$ is at least of dimension $\dim\bigl(H_{\bk-1,N+1}(q)\bigr)=\dim\bigl(H_{\bk-1,N+1}(1)\bigr)$. So combining all this, we conclude that in order to prove item 2, we need now to check that 
\begin{equation}\label{ineq1}
\dim \Bigl(H_{\bk,N+1}(1)AS_{\bk,N+1}(1)H_{\bk,N+1}(1)\Bigr)\geq \dim\bigl(H_{\bk-1,N+1}(1)\bigr)\ .
\end{equation}
To prove this, we first introduce some combinatorial definitions and notations. Let $x$ be a fused permuation in $\cD_{\bk-1,n}$ (a basis element of $H_{\bk-1,n}(1)$). We define $|x|_{\bk}$ to be the fused permutation obtained by adding a (vertical) edge connecting the $a$-th top ellipse to the $a$-th bottom ellipse for each $a=1,\dots,n$. So $|x|_{\bk}$ is a fused permutation in $\cD_{\bk,n}$ (a basis element of $H_{\bk-1,n}(1)$). Note that the notation is coherent since if $\bk=(2,2,2,...)$ then $x$ is a usual permutation and this definition coincides with the former definition of $|x|_{\bk}$.

Recall that here $N=2$. To prove (\ref{ineq1}), we are going to prove that the following set is linearly independent in $H_{\bk,3}(1)$:
\begin{equation}\label{linind1}
\{\AS_{\bk,3}(1)\cdot|x|_{\bk}\ ,\ \ x\in \cD_{\bk-1,3}\}\ .
\end{equation}
To do so, we use a total order on the set of fused permutations. Recall from Section \ref{sec-def-fus-perm} that a fused permutation is associated (one to one) to a sequence of multisets $(I_1,\dots,I_n)$ of elements of $\{1,\dots,n\}$. For two multisets $I=\{i_1,\dots,i_k\}$ and $J=\{j_1,\dots,j_k\}$, we write the elements in ascending order: $i_1\leq \dots\leq i_k$ and $j_i\leq \dots\leq i_k$ and we consider the lexicographic order, that is:
\[I<J\ \ \text{if}\ \ i_1<j_1\ \text{or}\ (i_1=j_1\ \text{and}\ i_2<j_2)\ \dots \]
Then for two sequences of multisets, we set:
\[(I_1,\dots,I_n)<(J_1,\dots,J_n)\ \ \text{if}\ \ I_1<J_1\ \text{or}\ (I_1=J_1\ \text{and}\ I_2<J_2)\ \dots \]
Thus, we have a total order on $\cD_{_bk,n}$. For usual permutation (when $\bk=(1,1,1,\dots)$), we denote $w_0$ the largest element for this order: $w_0$ is simply the usual longest element of $\mS_n$ associated to the sequence $(\{n\},\{n-1\},\dots,\{1\})$.

Finally, an element of $X\in H_{\bk,n}(1)$ being a linear combination of elements of $\cD_{\bk,n}$, we define the dominant element of $X$:
\[\Dom(X)=\text{largest element of $\cD_{\bk,n}$ appearing in $X$ with non-zero coefficient.}\]
Note that finding the dominant element in a product $x\cdot x'$ of two fused permutation is easy diagrammatically. We follow the edge starting from the top first ellipse and when arriving at a middle ellipse, we always choose the edges going to the right most direction. Then we repeat the procedure starting for the edges starting from the top second ellipse, and so on.

The following lemma implies immediately that the set of elements $\Dom(X)$, when $X$ runs over the set (\ref{linind1}), are different and thereby we obtain the linear independence of the set (\ref{linind1}) and conclude the verification of item 2.

\begin{lem} Let $x,x'\in\cD_{\bk-1,3}$.
\begin{enumerate}
\item[(i)] We have $\Dom\bigl(|w|_{\bk}\cdot |x|_{\bk}\bigr)\leq \Dom\bigl(|w_0|_{\bk}\cdot |x|_{\bk}\bigr)$ for every $w\in\mS_3$.
\item[(ii)] We have $\Dom\bigl(|w_0|_{\bk}\cdot |x|_{\bk}\bigr)<\Dom\bigl(|w_0|_{\bk}\cdot |x'|_{\bk}\bigr)$ if $x < x'$.
\end{enumerate}
\end{lem}
\begin{proof}[Proof of the lemma] All the proof is better read while drawing diagram. Let $x\in\cD_{\bk-1,3}$ and associate to it the following sequence of multisets of elements of $\{1,2,3\}$:
\[x\ \ \rightsquigarrow\ \ (\{i_2,\dots,i_{k_1}\},\{j_2,\dots,j_{k_2}\},\{l_2,\dots,l_{k_3}\})\ .\]
It is straightforward diagrammatically to see that the sequences of multisets associated to the following elements are:
\[\begin{array}{rcl}
|x|_{\bk} & \rightsquigarrow & (E_1,E_2,E_3)=(\{1,i_2,\dots,i_{k_1}\},\{2,j_2,\dots,j_{k_2}\},\{3,l_2,\dots,l_{k_3}\})\\[0.5em]
\Dom\bigl(|w_0|_{\bk}\cdot |x|_{\bk}\bigr) & \rightsquigarrow & (E^m_1,E^m_2,E^m_3)=(\{3,i_2,\dots,i_{k_1}\},\{2,j_2,\dots,j_{k_2}\},\{1,l_2,\dots,l_{k_3}\})
\end{array}\]
Item (ii) follows then immediately. For item (i), let $w\in \mS_3$ and denote $(E'_1,E'_2,E'_3)$ the sequence of multisets associated to $\Dom\bigl(|w|_{\bk}\cdot |x|_{\bk}\bigr)$. Assume that $(E'_1,E'_2,E'_3)>(E_1^m,E_2^m,E_3^m)$. We will obtain a contradiction.

$\bullet$ If $w(1)=1$ then we have $E'_1=E_1<E_1^m$ which is a contradiction.

$\bullet$ If $w(1)=2$ then $E'_1$ is formed by $i_2,\dots,i_{k_1}$ together with a maximal element of $\{2,j_2,\dots,j_{k_2}\}$. As $E'_1\geq E^m_1$ this element must be 3, so that $3\in\{j_2,\dots,j_{k_2}\}$. Say $j_{k_2}=3$. So we have $E'_1=E_1^m$. Next, $E'_2$ is forced to contain $\{2,j_2,\dots,j_{k_2-1}\}$ and its additional element must be a 3 since $E'_2\geq E_2^m$ and $j_{k_2}=3$. So we have $E'_2=E_2^m$ and we are left, by collecting the remaining elements, with $E'_3=\{1,l_2,\dots,l_{k_3}\}=E^m_3$. So we have $(E'_1,E'_2,E'_3)=(E_1^m,E_2^m,E_3^m)$ which is a contradiction.

$\bullet$ If $w(1)=3$. If $w=w_0$ then $(E'_1,E'_2,E'_3)=(E_1^m,E_2^m,E_3^m)$ which is a contradiction. So we must have $w(2)=1$ and $w(3)=2$. We have $E'_1=\{i_2,\dots,i_{k_1},3\}=E^m_1$ where the last 3 comes from $E_3$. As $w(2)=1$ then $E'_2$ must contain a 1 (from $E_1$), and therefore, from $E'_2\geq E^m_2$ we have that $1\in\{j_2,\dots,j_{k_2}\}$, say $j_2=1$. Then we have $E'_2=\{1,2,j_3,\dots,j_{k_2}\}=E_2^m$. By collecting the remaining elements, we have $E'_3=\{j_2,l_2,\dots l_{k_3}\}$ which is equal to $E^m_3$ since $j_2=1$. So we again reach the contradiction $(E'_1,E'_2,E'_3)=(E_1^m,E_2^m,E_3^m)$.
\end{proof}

\vskip .0cm
\textbf{3.} Let $(k_1,\dots,k_{N+1})=(2,\dots,2,1,\dots,1)$ consisting of $L$ 2's for some $L\in\{1,\dots,N+1\}$ (recall that we assumed in this section that $\bk$ is already in decreasing order). So here $H_{\bk-1,N+1}(1)=\CC\mS_{L}$ since a fused pemutation of type $(1,\dots,1,0,\dots,0)$ is a permutation in $\mS_L$. With the same reasoning as in the beginning of the proof of \textbf{2}, we see that we are left to proving
\begin{equation}\label{ineq2}
\dim \Bigl(H_{\bk,N+1}(1)AS_{\bk,N+1}(1)H_{\bk,N+1}(1)\Bigr)\geq \dim\bigl(H_{\bk-1,N+1}(1)\bigr)=L!\ .
\end{equation}
For $w\in \mS_L$ we take its permutation diagram and we double all the edges. So we have 2 lines of $L$ ellipses which are joined by double edges acccording to the permutation $w$. Then we complete the lines of ellipses with $N+1-L$ ellipses and we add a vertical edge for these last ellipses. So at the end, we have a fused permutation in $\cD_{\bk,N+1}$ that we denote $w^{(2)}_{\bk}$. We claim that
the following set is linearly independent in $H_{\bk,N+1}(1)$:
\begin{equation}\label{linind2}
\{\AS_{\bk,N+1}(1)\cdot w^{(2)}_{\bk}\ ,\ \ w\in \mS_L\}\ .
\end{equation}
For $\pi\in\mS_{N+1}$ and $w\in\mS_L$ it is easy to see that $|\pi|_{\bk}\cdot w^{(2)}_{\bk}$ is equal to a single element of $\cD_{\bk,N+1}$ (no sum is involved) and is in fact just $|\pi|_{\bk}$ where the first $L$ ellipses of the bottom line have been permuted by $w$. Thus we see readily that in the sum $\AS_{\bk,N+1}(1)\cdot w^{(2)}_{\bk}$ there is a single element with only double edges on the first $L$ ellipses, this is $w^{(2)}_{\bk}$ (the term obtained from $|\text{Id}|_{\bk}$ in $\AS_{\bk,N+1}(1)$). Thus, $w^{(2)}_{\bk}$ appears in $\AS_{\bk,N+1}(1)\cdot w^{(2)}_{\bk}$ and does not appear in $\AS_{\bk,N+1}(1)\cdot w'^{(2)}_{\bk}$ if $w'\neq w$. This shows that the set \eqref{linind2} is indeed lienarly independent in $H_{\bk,N+1}(1)$, concluding the proof of the inequality in (\ref{ineq2}).  

\vskip .2cm
\textbf{4.} We keep the notation of the preceding item. Let $I'$ be the ideal such that the algebra $H_{\bk,N+1}(1)$ is the direct sum of $I'$ and $I_{\bk,N+1}^N$. We have seen that $\AS_{\bk,N+1}(1)$ belongs to $I_{\bk,N+1}^N$ so we have $x'\AS_{\bk,N+1}(1)=\AS_{\bk,N+1}(1)x'=0$ for all $x'\in I'$. So to show that $\AS_{\bk,N+1}(1)$ is central in $H_{\bk,N+1}(1)$, we must show that $\AS_{\bk,N+1}(1)$ commutes with all elements in $I_{\bk,N+1}^N$, which we have seen to be be spanned by elements in (\ref{linind2}). So finally, we must prove that $\AS_{\bk,N+1}(1)$ commutes with all elements $w^{(2)}_{\bk}$ with $w\in \mS_L$.

We need one final piece of notations. For $w\in\mS_L$ and $\pi\in\mS_{N+1}$, we draw the edges of $w$ and of $\pi$ on the same diagram, and we thus obtain a fused permutation of $\cD_{\bk,N+1}$ that we denote $\pi\odot w$. For example, we have $|\pi|_{\bk}=\pi\odot\text{Id}_L$,where $\text{Id}_mn$ denotes the identity in $\mS_n$. We have also $w^{(2)}_{\bk}=w\odot w$ (here and below we see $w$ both as an element of $\mS_L$ and of $\mS_{N+1}$ by the standard inclusion). With these notations, it is easy to check diagrammatically that we have:
\[|\pi|_{\bk}\cdot w^{(2)}_{\bk}=\pi w\odot w\ \ \ \ \text{and}\ \ \ \ w^{(2)}_{\bk}\cdot |\pi|_{\bk}=w\pi\odot w\ .\]
So we get finally
\[\AS_{\bk,N+1}(1)\cdot w^{(2)}_{\bk}=\sum_{\pi\in \mS_{N+1}}(-1)^{\ell(\pi)}\pi w\odot w=\sum_{\pi\in \mS_{N+1}}(-1)^{\ell(w\pi w^{-1})}w\pi\odot w=w^{(2)}_{\bk}\cdot\AS_{\bk,N+1}(1)\ ,\]
using that the sign $(-1)^{\ell(\pi)}$ is multiplicative. The proof is concluded.
\end{proof}

\appendix

\section{Artin--Wedderburn decompositions and Bratteli diagrams}\label{sec-Brat}

\subsection{Semisimple algebras and algebras of the form $PAP$}\label{app-ss}

\paragraph{Artin--Wedderburn decomposition of a semisimple algebra.} Let $A$ be a finite-dimensional semisimple algebra over $\CC$. Let $S$ be an indexing set for a complete set of pairwise non-isomorphic irreducible representations of $A$. Then Artin--Wedderburn theorem asserts that we have the following isomorphism of algebras:
\begin{equation}\label{AW}
A\cong \bigoplus_{\lambda\in S} \text{End}(V_{\lambda})\ ,
\end{equation}
where $V_{\lambda}$ is a realisation of the irreducible representation corresponding to $\lambda$. The isomorphism is given naturally by sending $a\in A$ to the endomorphism in $\text{End}(V_{\lambda})$ corresponding to the action of $a$ on $V_{\lambda}$.

\paragraph{Algebras of the form $PAP$ and their representations.} Let $A$ be a $\CC$-algebra with an idempotent $P$. Then the subset $PAP=\{PxP\ |\ x\in A\}$ is an algebra with unit $P$. We recall very basic and classical facts on the algebra $PAP$ which can be found for example in \cite[\S 6.2]{Gr}.

Let $\rho\ :\ A\mapsto\text{End}(V)$ be a representation of $A$ and $W:=\rho(P)(V)$ the image of the operator $\rho(P)$. The subspace $W$ is naturally a representation of the algebra $PAP$. Indeed $W$ is obviously invariant under the action of any element the form $\rho(PxP)$, and thus the action of $PAP$ on $W$ is given simply by restriction:
\begin{equation}\label{act-PAP}
\begin{array}{rcl}
PAP & \to & \text{End}\bigl(W\bigr) \\[0.5em]
PxP & \mapsto & {\rho(PxP)}_{|_{W}}
\end{array}\ .
\end{equation}
From now on, we will always remove the map $\rho$ from the notation, and keep the same notation for an element of an algebra and its action in a given representation.

\paragraph{Irreducible representations and semisimplicity.} Let $\{V_{\lambda}\}_{\lambda\in S}$ be a complete set of pairwise non-isomorphic irreducible representations of $A$. Then the set  
$$\{P(V_{\lambda})\ |\ \lambda\in S \text{ and } P(V_{\lambda})\neq 0\}$$ 
 is a complete set of pairwise non-isomorphic irreducible representations of the algebra $PAP$.

Moreover, if the algebra $A$ is a finite-dimensional semisimple algebra then the algebra $PAP$ is also a finite-dimensional semisimple algebra and its Artin--Wedderburn decomposition is
\[PAP\cong \bigoplus_{\substack{\lambda\in S\\ P(V_{\lambda})\neq0}} \text{End}\bigl(P(V_{\lambda})\bigr)\ .\]

\subsection{Minimal central idempotents and ideals}\label{app-idem}

Let $A$ be a finite-dimensional semisimple algebra over $\CC$ and $S$ an indexing set for a complete set of pairwise non-isomorphic irreducible representations of $A$.

\paragraph{Minimal central idempotents.} Let $\lambda\in S$. We define $E_{\lambda}$ as the element of $A$ corresponding under the Artin--Wedderburn decomposition of $A$ to $\text{Id}_{V_{\lambda}}$ in the component corresponding to $\lambda$ and $0$ in all other components. The set $\{E_{\lambda}\}_{\lambda\in S}$ is a complete set of minimal central orthogonal idempotents of $A$, meaning that they are central, they sum to 1, they satisfy $E_{\lambda}E_{\lambda'}=\delta_{\lambda,\lambda'}E_{\lambda}$ and they cannot be written as the sum of two non-zero central idempotents.

In any representation $W$ of $A$, the action of $E_{\lambda}$ projects onto the isotopic component  of $W$ corresponding to $\lambda$. More precisely, if the decomposition of $W$ into irreducible is
\[W\cong \bigoplus_{\lambda'\in S}V_{\lambda'}^{\oplus m_{\lambda'}}\ ,\]
then the action of $E_{\lambda}$ is the projection onto $V_{\lambda}^{\oplus m_{\lambda}}$ corresponding to this decomposition, that is:
\[{E_{\lambda}}_{|_{V_{\lambda}^{\oplus m_{\lambda}}}}=\text{Id}_{V_{\lambda}^{\oplus m_{\lambda}}}\ \ \ \ \text{and}\ \ \ \ E_{\lambda}\Bigl(\bigoplus_{\lambda'\in S\backslash\{\lambda\}}V_{\lambda'}^{\oplus m_{\lambda'}}\Bigr)=0\ .\]

\paragraph{Ideals and quotients.} From the decomposition (\ref{AW}), one sees immediately that ideals (and equivalently quotients) of $A$ are in correspondence with subsets $S'\subset S$ as follows:
\[I_{S'}:=\bigoplus_{\lambda\in S'} \text{End}(V_{\lambda})\ \ \ \ \ \ \text{and}\ \ \ \ \ \ A/I_{S'}\cong\bigoplus_{\lambda\in S\backslash S'} \text{End}(V_{\lambda})\ .\]
One set of generators of the ideal $I_{S'}$ consists of the elements $E_{\lambda}$ with $\lambda\in S'$.

Since we use it in Section \ref{sec-quot}, we recall that the ideal generated by an element $x\in A$ is $I_{S'}$ where $S'$ is the subset of irreducible representations such that $x$ acts as a non-zero element.

\subsection{Bratteli diagram of a chain of algebras}\label{app-bra}

 Let $\{A_n\}_{n\geq0}$ be a family of algebras and assume that, for any $n\geq0$, there is a given injective map from $A_n$ to $A_{n+1}$. We call these maps ``inclusion maps'' and we say that $\{A_n\}_{n\geq0}$, sometimes denoted as follows
 \[A_0\subset A_1\subset A_2\subset\dots\ldots\subset A_n\subset A_{n+1}\subset\dots\ ,\]
forms a chain of algebras. The inclusion maps allow to consider elements of $A_n$ as elements of $A_{n+1}$, and more generally of $A_{n+k}$ for $k\geq 1$, and in turn to consider $A_n$ as a subalgebra of $A_{n+1}$, and more generally of $A_{n+k}$ for $k\geq 1$.

From the inclusion of $A_n$ into $A_{n+1}$, any representation $V$ of $A_{n+1}$ can be seen as a representation of $A_n$ by restriction; we denote this representation of $A_n$ by $\text{Res}_{A_n}(V)$. 

\paragraph{Bratteli diagram of a chain of semisimple algebras.} Let $\{A_n\}_{n\geq0}$ be a chain of algebras and we assume here that the algebras $A_n$ are finite-dimensional semisimple algebras, since this will always be our setting in this paper. 

For an irreducible representation $V_i$ of $A_{n+1}$, the restriction to $A_n$ decomposes by semisimplicity into a direct sum of irreducible representations, namely,
\[\text{Res}_{A_n}(V_i)\cong\bigoplus m_j W_j\ ,\]
where $W_j$ are non-isomorphic irreducible representations of $A_n$ and the numbers $m_j$ are called the multiplicities. The knowledge of such decomposition for any irreducible representation of $A_{n+1}$ for any $n\geq 0$ are called the branching rules of the chain of algebras. Then the Bratteli diagram of the chain of algebras $\{A_n\}_{n\geq0}$ is the following graph:

\begin{itemize}
\item The set of vertices is partitioned into subsets indexed by $n\geq0$. We call $n$ the level. The vertices of level $n$ are indexed by the (isomorphism classes of) irreducible representations of the algebras $A_n$. 
\item The edges express the branching rules of the chain of algebras and they only connect vertices of adjacent levels. Let $V$ be an irreducible representation of $A_n$ and $V'$ an irreducible representation of $A_{n+1}$. Then there are $m$ edges connecting the vertices indexed by $V$ and $V'$ if and only if $V$ appears in the decomposition of $\text{Res}_{A_n}(V')$ with multiplicity $m$. 
\end{itemize} 

\noindent Graphically, we place all the vertices of a given level on an horizontal line, and we put the vertices of level $n+1$ below the vertices of level $n$. We often think of the edges as going down from vertices of level $n$ to vertices of level $n+1$.

Let $v,v'$ be two vertices of the Bratteli diagram. A path from $v$ to $v'$ is a sequence of vertices of the form $v,v_1,\dots,v_{k-1},v'$, for some $k\geq 1$, such that at each step of the sequence, the level increases by 1. In other words, a path from $v$ to $v'$, if it exists, is obtained by starting from $v$  and following edges only in the downward direction to reach $v'$.

The partial order $\leq$ on the set of vertices of a Bratteli diagram is defined by setting that $v\leq v'$ if and only if $v=v'$ or there is a path from $v$ to $v'$.

\paragraph{Dimensions.} We often add the following numerical information to the Bratteli diagram of the chain $\{A_n\}_{n\geq0}$: next to each vertex, we indicate the dimension of the corresponding irreducible representation. 

Obviously, for $n\geq 1$ and any vertex $V$ of level $n$, this dimension can be obtained from the preceding level. Indeed it is the sum of the dimensions of the irreducible representations of level $n-1$ connected to $V$ (counted with multiplicity indicated by the number of edges).

Moreover, from the Artin--Wedderburn decomposition of the algebras $A_n$, we have that the dimension of the algebra $A_n$ is the sum of the squares of the dimensions appearing at level $n$.

\begin{exa}
A standard example of a Bratteli diagram is the Young diagram, corresponding to the poset of partitions partially ordered by inclusion. It is the Bratteli diagram associated to the chain $\{\CC\mathfrak{S}_n\}_{n\geq0}$ of the complex group algebras of the symmetric groups. The first levels are given in Appendix \ref{app1}.
\end{exa}

\subsection{Quotients of Bratteli diagrams and chains of ideals}\label{app-quot-Brat}

We keep our setting of a chain $\{A_n\}_{n\geq 0}$ of finite-dimensional semisimple algebras. Let $\mathcal{D}$ be the Bratteli diagram of the chain $\{A_n\}_{n\geq 0}$ and let $S$ be a subset of vertices of $\mathcal{D}$. 

\paragraph{Quotients of Bratteli diagrams.} We denote by $\langle S\rangle$ the set of vertices $v'$ such that there exists $v\in S$ with $v\leq v'$ (\emph{i.e.} all vertices in $S$ and all vertices connected by a path to them). 
\begin{defi}
We define $\overline{\mathcal{D}}_S$ to be the diagram obtained from $\mathcal{D}$ by removing all vertices $v'\in\langle S\rangle$ and keeping only the edges of $\mathcal{D}$ which connect the remaining vertices.

We call the resulting diagram $\overline{\mathcal{D}}_S$ the quotient of $\mathcal{D}$ generated by $S$.
\end{defi}

Obviously, the quotient $\overline{\mathcal{D}}_S$ depends only on the set of vertices $\langle S\rangle$ generated by $S$. Hence, several choices of $S$ can lead to the same quotient. There is a unique minimal choice $S_{min}$, which is the set of minimal elements (for the partial order $\leq$) in $\langle S\rangle$. We call $S_{min}$ the \emph{minimal generating set} for the quotient $\overline{\mathcal{D}}_S$.

\begin{exa}
A standard example of a quotient of a Bratteli diagram is the following.  Take the Bratteli diagram of the chain $\{\CC\mathfrak{S}_n\}_{n\geq0}$ (see Subsection \ref{app1} below) and make the quotient generated by the vertex labelled by the partition $\begin{array}{c}
\fbox{\phantom{\scriptsize{$2$}}} \\[-0.2em]
\fbox{\phantom{\scriptsize{$2$}}} \\[-0.2em]
\fbox{\phantom{\scriptsize{$2$}}}
\end{array}$. It is easy to see that the remaining vertices are the partitions with no more than two lines. The quotient is equal to the Bratteli diagram of the chain of Temperley--Lieb algebras.
\end{exa}

\paragraph{Representation-theoretic meaning and chains of ideals.} We will explain the name Bratteli diagram for $\overline{\mathcal{D}}_S$, and its representation-theoretic meaning. 

For every $n\geq 0$, let $S_n$ be the set of vertices of level $n$ inside $\langle S\rangle$ (that is, the vertices of level $n$ which have been removed from $\mathcal{D}$). To $S_n$ corresponds an ideal $I_n$ of $A_n$. We have that $\{I_n\}_{n\geq0}$ forms a chain of ideals in $\{A_n\}_{n\geq0}$:
\begin{equation}\label{chain-ideal}
I_0\subset I_1\subset I_2\subset\dots\ldots\subset I_n\subset I_{n+1}\subset\dots\ ,
\end{equation}
which means that $I_n$, seen as a subset of $A_{n+1}$ using the inclusion map, is contained in $I_{n+1}$ for any $n\geq 0$. This chain property follows from the fact that, by definition of $\langle S\rangle$, every edge starting from $S_n$ ends in $S_{n+1}$ (see Remark \ref{rem-chain} below).

Now the quotient $\overline{\mathcal{D}}_S$ has the following meaning:

$\bullet$ The vertices of level $n$ are in bijection with the irreducible representations of $A_n/I_n$.

$\bullet$ The edges give the branching rules for the following restriction procedure: a representation $V'$ of $A_{n+1}/I_{n+1}$ can be seen as a representation of $A_{n+1}$ where $I_{n+1}$ acts as $0$. Thus, from the inclusion $A_n\subset A_{n+1}$, we can form the representation $\text{Res}_{A_n}(V')$ of $A_n$. Now by the inclusion $I_n\subset I_{n+1}$, we have that $I_n$ obviously acts as 0, and therefore $\text{Res}_{A_n}(V')$ can be seen as a representation of $A_n/I_n$.

\begin{rem}\label{rem-chain}
$\bullet$ The family of quotients $\{A_n/I_n\}_{n\geq 0}$ does not necessarily form a chain of algebras, even if the ideals $I_n$ form a chain of ideals. In fact, one can check that the quotients $A_n/I_n$ form a chain of algebras for the natural inclusion maps $x+I_n\mapsto x+I_{n+1}$
if and only if we have, for any $n\geq 0$, $I_n=I_{n+1}\cap A_n$ as subsets of $A_{n+1}$. This is stronger than the chain property for the ideals $I_n$.

$\bullet$ The property $I_n=I_{n+1}\cap A_n$ ensuring that the quotients form a chain of algebras can be seen easily in the Bratteli diagram. We have that $I_n=I_{n+1}\cap A_n$ if and only if:
\[v\in S_n\ \ \ \ \Longleftrightarrow\ \ \ \ \text{$\forall v'$ of level $n+1$, $v\leq v'$ implies $v'\in S_{n+1}$.}\]
We note that the weaker property $I_n\subset I_{n+1}$ is equivalent to the single implication $\Rightarrow$.
\end{rem}

\paragraph{Algebraic description of the quotients $A_n/I_n$.} To make explicit the situation we need in this paper, assume that the set $S_{min}$ contains vertices of a single level, say $N+1$ (the general case can be obtained by partitioning the set $S_{min}$ according to the level, and applying this procedure to each part). 

Denote by $X$ a generator of the ideal $I_{N+1}$ of $A_{N+1}$ (for example the sum of the central idempotents corresponding to $S_{min}$). Then we have that  
for any $n\geq N+1$, the ideal $I_n$ of $A_n$ is generated by the element $X^{\uparrow n}$ which is the element $X$ seen as an element of $A_n$ by the inclusion (indeed, this element $X^{\uparrow n}$ is non-zero precisely in the correct set of irreducible representations of $A_n$, by definition of $S_{min}$ and of a Bratteli diagram).

As a conclusion, we note that for all $n\geq N+1$, the algebra $A_n/I_n$ is the quotient of $A_n$ over the relation $X^{\uparrow n}=0$. We refer to Example \ref{exa-TL} for the well-known example of Temperley--Lieb algebras.

\section{Examples}

\subsection{The chain of Hecke algebras $H_n(q)$}\label{app1}

The Hecke algebra $H_n(q)$ is the particular case of the algebra $H_{\bk,n}(q)$ where $\bk=(1,1,\dots)$ is the infinite sequence of $1$'s. The first levels of the Bratteli diagram for the chain of Hecke algebras $\{H_n(q)\}_{n\geq0}$ is shown below. 

The shaded areas indicate the connections between the Hecke algebras $H_n(q)$ and the centralisers of the representations (corresponding to $\bk=(1,1,1,\dots)$) of $U_q(gl_N)$. Namely, by deleting the vertices included in the shaded area labelled $gl_N$ together with the edges touching them, we obtain the Bratteli diagram of the centralisers of $U_q(gl_N)$. For example, if $N=2$, the quotiented Bratteli diagram is the Bratteli diagram of the Temperley--Lieb algebras.

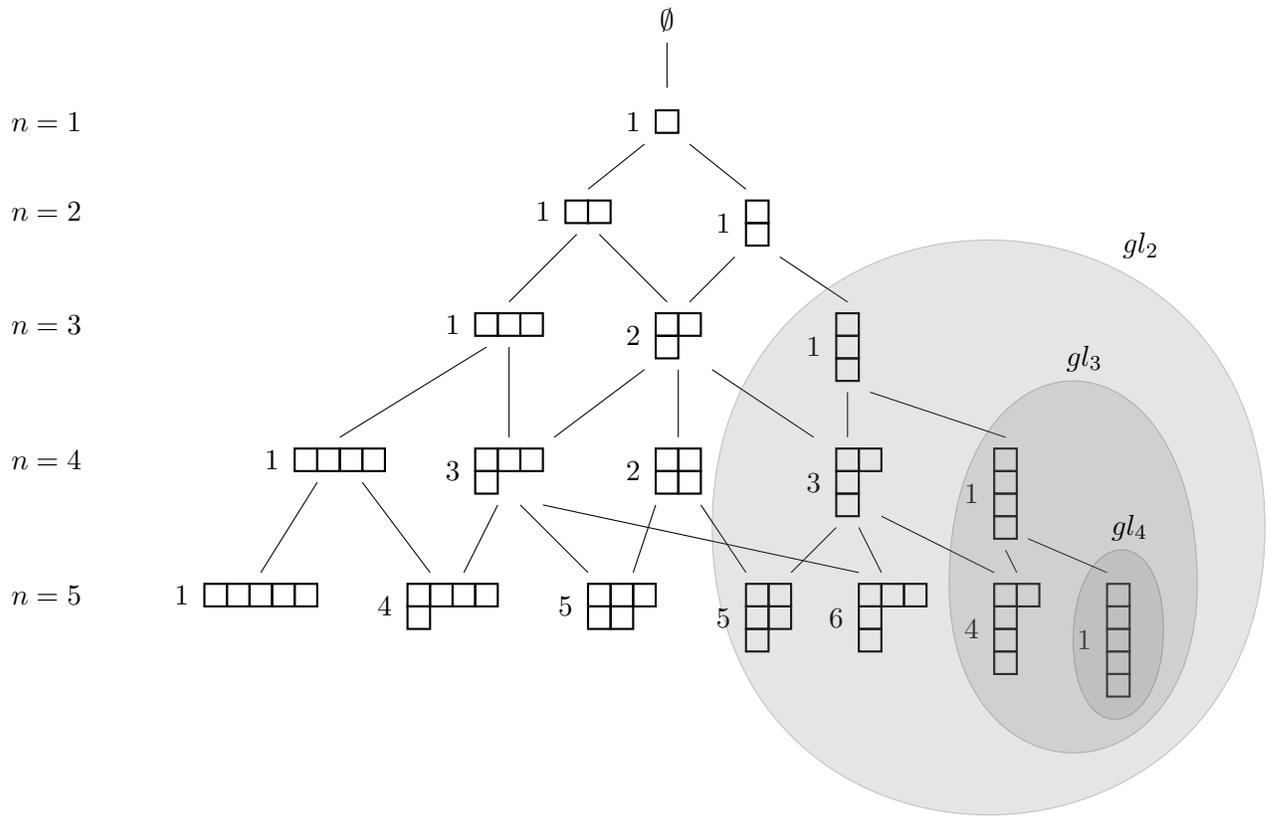
\begin{center}
 \begin{tikzpicture}[scale=0.3]
\node at (0.5,4) {$\emptyset$};
\draw ( 0.5,3) -- (0.5, 1);
\diag{0}{0}{1};\node at (-1,-0.5) {$1$};

\draw (-0.5,-1.5) -- (-3,-3.5);\draw (1.5,-1.5) -- (4,-3.5);

\diag{-4}{-4}{2};\node at (-5,-4.5) {$1$}; \diagg{4}{-4}{1}{1};\node at (3,-5) {$1$};

\draw (-3.5,-5.5) -- (-6.5,-8.5);\draw (-2.5,-5.5) -- (0.5,-8.5);\draw (3.5,-6.5) -- (1.5,-8.5);\draw (5.5,-6.5) -- (8.5,-8.5);

\diag{-8}{-9}{3};\node at (-9,-9.5) {$1$};\diagg{0}{-9}{2}{1};\node at (-1,-10) {$2$};\diaggg{8}{-9}{1}{1}{1};\node at (7,-10.5) {$1$};

\draw (-7.5,-10.5) -- (-14,-14.5);\draw (-6.5,-10.5) -- (-6.5,-14.5);\draw (-0.5,-11.5) -- (-4.5,-14.5);\draw (1,-11.5) -- (1,-14.5);\draw (2.5,-11.5) -- (7,-14.5);
\draw (8.5,-12.5) -- (8.5,-14.5);\draw (9.5,-12.5) -- (15.5,-14.5);

\diag{-16}{-15}{4};\node at (-17,-15.5) {$1$};\diagg{-8}{-15}{3}{1};\node at (-9,-16) {$3$};\diagg{0}{-15}{2}{2};\node at (-1,-16) {$2$};\diaggg{8}{-15}{2}{1}{1};\node at (7,-16.5) {$3$};\diagggg{15}{-15}{1}{1}{1}{1};\node at (14,-17) {$1$};

\draw (-15,-16.5)--(-17.5,-20.5);\draw (-13,-16.5)--(-10,-20.5);\draw (-7,-17.5)--(-8.5,-20.5);\draw (-6,-17.5)--(-3,-20.5);\draw (-5,-17.5)--(9,-20.5);
\draw (0,-17.5)--(-1,-20.5);\draw (2,-17.5)--(4,-20.5);
\draw (8,-18.5)--(6,-20.5);\draw (9,-18.5)--(10,-20.5);\draw (10,-18)--(15,-20.5);\draw (15.5,-19.5)--(16,-20.5);\draw (16.5,-19)--(20,-20.5);

\diag{-20}{-21}{5};\node at (-21,-21.5) {$1$};\diagg{-11}{-21}{4}{1};\node at (-12,-22) {$4$};\diagg{-3}{-21}{3}{2};\node at (-4,-22) {$5$};
\diaggg{4}{-21}{2}{2}{1};\node at (3,-22.5) {$5$};\diaggg{9}{-21}{3}{1}{1};\node at (8,-22.5) {$6$};\diagggg{15}{-21}{2}{1}{1}{1};\node at (14,-23) {$4$};\diaggggg{20}{-21}{1}{1}{1}{1}{1};\node at (19,-23.5) {$1$};

\draw[thin, fill=gray,opacity=0.2] (2.5,-18.5)..controls +(0,17) and +(0,17) .. (27,-18.5) .. controls +(0,-17) and +(0,-17) .. (2.5,-18.5);\node at (21.5,-6) {$gl_2$};
\draw[thin, fill=gray,opacity=0.2] (13,-21)..controls +(0,12) and +(0,12) .. (24,-21) .. controls +(0,-10) and +(0,-10) .. (13,-21);\node at (19,-11) {$gl_3$};
\draw[thin, fill=gray,opacity=0.2] (18.5,-23.5)..controls +(0,5) and +(0,5) .. (22.5,-23) .. controls +(0,-5) and +(0,-5) .. (18.5,-23.5);\node at (21,-18.5) {$gl_4$};

\node at (-27,-0.5) {$n=1$};\node at (-27,-4.5) {$n=2$};\node at (-27,-9.5) {$n=3$};\node at (-27,-15.5) {$n=4$};\node at (-27,-21.5) {$n=5$};

\end{tikzpicture}
\end{center}

\subsection{The chain of algebras $H_{\bk,n}(q)$ when $\bk=(2,2,2,\dots)$}\label{app2}

When $\bk=(2,2,2,\dots)$ is the infinite sequence of $2$'s, the Bratteli diagram for the chain of algebras $\{H_{\bk,n}(q)\}_{n\geq0}$ begins as: 

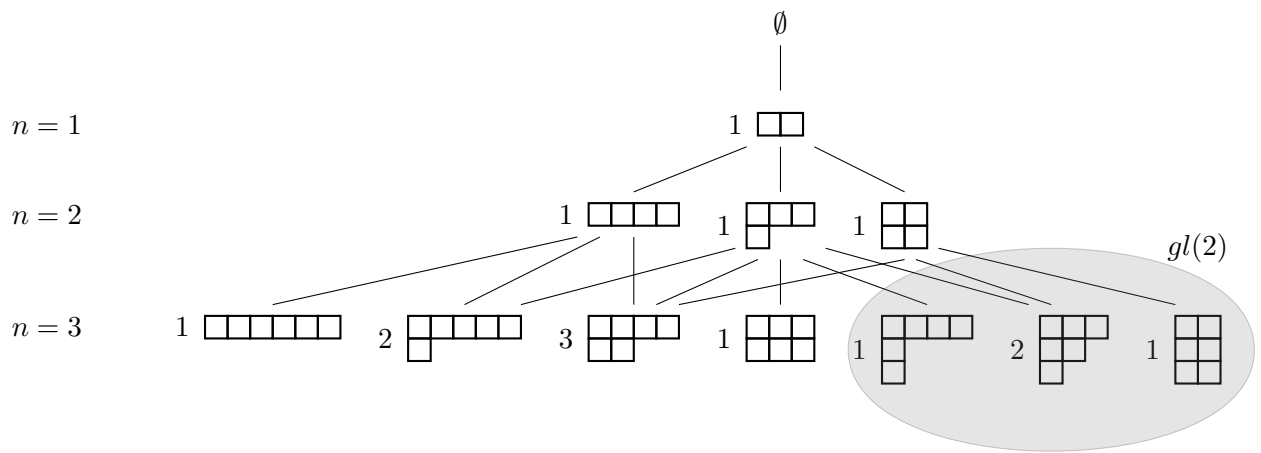
\begin{center}
 \begin{tikzpicture}[scale=0.3]
\node at (0.5,4) {$\emptyset$};
\draw ( 0.5,3) -- (0.5, 1);
\diag{-0.5}{0}{2};\node at (-1.5,-0.5) {$1$};
\draw (-1,-1.5) -- (-6,-3.5);\draw (0.5,-1.5)--(0.5,-3.5);\draw (2,-1.5) -- (6,-3.5);
\diag{-8}{-4}{4};\node at (-9,-4.5) {$1$};\diagg{-1}{-4}{3}{1};\node at (-2,-5) {$1$};\diagg{5}{-4}{2}{2};\node at (4,-5) {$1$};

\draw (-8.5,-5.5) -- (-22,-8.5);\draw (-7.5,-5.5) -- (-13.5,-8.5);\draw (-6,-5.5) -- (-6,-8.5);    \draw (-1.5,-6) -- (-11,-8.5); \draw (-0.5,-6.5) -- (-5,-8.5);\draw (0.5,-6.5) -- (0.5,-8.5);
\draw (1.5,-6.5) -- (7,-8.5);\draw (2.5,-6) -- (11.5,-8.5);\draw (6,-6.5) -- (-4,-8.5); \draw (6.5,-6.5) -- (12.5,-8.5);\draw (7.5,-6) -- (18,-8.5);

\diag{-25}{-9}{6};\node at (-26,-9.5) {$1$};\diagg{-16}{-9}{5}{1};\node at (-17,-10) {$2$};\diagg{-8}{-9}{4}{2};\node at (-9,-10) {$3$};
\diagg{-1}{-9}{3}{3};\node at (-2,-10) {$1$};\diaggg{5}{-9}{4}{1}{1};\node at (4,-10.5) {$1$}; \diaggg{12}{-9}{3}{2}{1};\node at (11,-10.5) {$2$};\diaggg{18}{-9}{2}{2}{2};\node at (17,-10.5) {$1$};

\draw[thin, fill=gray,opacity=0.2] (3.5,-10.5)..controls +(0,6) and +(0,6) .. (21.5,-10.5) .. controls +(0,-6) and +(0,-6) .. (3.5,-10.5);\node at (19,-6) {$gl(2)$};

\node at (-32,-0.5) {$n=1$};\node at (-32,-4.5) {$n=2$};\node at (-32,-9.5) {$n=3$};
\end{tikzpicture}
\end{center}

Note that there is no arrow from $\mu=\begin{array}{cc}
\fbox{\phantom{\scriptsize{$2$}}} &\hspace{-0.35cm}\fbox{\phantom{\scriptsize{$2$}}}\\[-0.2em]
\fbox{\phantom{\scriptsize{$2$}}} &\hspace{-0.35cm}\fbox{\phantom{\scriptsize{$2$}}}
\end{array}$ to $\lambda=\begin{array}{ccc}
\fbox{\phantom{\scriptsize{$2$}}} &\hspace{-0.35cm}\fbox{\phantom{\scriptsize{$2$}}} & \hspace{-0.35cm}\fbox{\phantom{\scriptsize{$2$}}}\\[-0.2em]
\fbox{\phantom{\scriptsize{$2$}}} &\hspace{-0.35cm}\fbox{\phantom{\scriptsize{$2$}}} & \hspace{-0.35cm}\fbox{\phantom{\scriptsize{$2$}}}
\end{array}$ even if $\mu\subset\lambda$ since $\lambda/\mu$ contains two boxes in the same column.

The shaded area indicates the connections between the fused Hecke algebras $H_{\bk,n}(q)$ and the centraliser of the representations (corresponding to $\bk=(2,2,2,\dots)$) of $U_q(gl_N)$, as in the preceding example.

\subsection{The chain of algebras $H_{\bk,n}(q)$ when $\bk=(3,1,1,1,\dots)$}\label{app3}

When $\bk=(3,1,1,1\dots)$, the Bratteli diagram for the chain of
algebras $\{H_{\bk,n}(q)\}_{n\geq0}$ begins as (the shaded areas have a similar meaning as in the preceding examples):

\begin{center}
 \begin{tikzpicture}[scale=0.3]
\node at (0.5,4) {$\emptyset$};
\draw ( 0.5,3) -- (0.5, 1);
\diag{-1}{0}{3};\node at (-2,-0.5) {$1$};
\draw (-0.5,-1.5) -- (-3,-3.5);\draw (1.5,-1.5) -- (3.5,-3.5);
\diag{-5}{-4}{4};\node at (-6,-4.5) {$1$};\diagg{2}{-4}{3}{1};\node at
(1,-5) {$1$};

\draw (-5,-5.5) -- (-10.5,-8.5);\draw (-3,-5.5) -- (-3,-8.5);\draw
(1.7,-6.3) -- (-1,-8.5);\draw (3.5,-6.3) -- (3.5,-8.5);\draw (5,-5.5)
-- (9.5,-8.5);

\diag{-13}{-9}{5};\node at (-14,-9.5) {$1$};\diagg{-5}{-9}{4}{1};\node
at (-6,-10) {$2$};\diagg{2}{-9}{3}{2};\node at (1,-10)
{$1$};\diaggg{8}{-9}{3}{1}{1};\node at (7,-10.5) {$1$};

\draw (-13,-10.5) -- (-22,-14.5);\draw (-10.5,-10.5) --
(-13.5,-14.5);\draw (-5.3,-11.3) -- (-12.5,-14.5);\draw (-4,-11.3) --
(-6,-14.5);\draw (-3,-10.5) -- (5,-14.5);
\draw (1.7,-11.3) -- (-4,-14.5);\draw (3,-11.3) -- (0.5,-14.5);\draw
(4.3,-11.3) -- (11.7,-14.5);\draw (7.7,-12.3) -- (7,-14.5);\draw
(9.3,-12.3) -- (13.5,-14.5);
\draw (11,-10.5) -- (19.5,-14.5);

\node at (-26,-15.5) {$1$};\diag{-25}{-15}{6};
\node at (-17,-16) {$3$};\diagg{-16}{-15}{5}{1};

\node at (-9,-16) {$3$};\diagg{-8}{-15}{4}{2};

\node at (-2,-16) {$1$};\diagg{-1}{-15}{3}{3};

\node at (4,-16.5) {$3$};\diaggg{5}{-15}{4}{1}{1};\node at (11,-16.5)
{$2$};\diaggg{12}{-15}{3}{2}{1};\node at (17,-17)
{$1$};\diagggg{18}{-15}{3}{1}{1}{1};

\draw[thin, fill=gray,opacity=0.2] (3.5,-16)..controls +(0,12) and
+(0,12) .. (22.5,-16) .. controls +(0,-12) and +(0,-12) ..
(3.5,-16);\node at (18,-7) {$gl(2)$};
\draw[thin, fill=gray,opacity=0.2] (16,-16)..controls +(0,5) and
+(0,5) .. (21.5,-16) .. controls +(0,-5) and +(0,-5) .. (16,-16);\node
at (19,-11.5) {$gl(3)$};

\node at (-32,-0.5) {$n=1$};\node at (-32,-4.5) {$n=2$};\node at
(-32,-9.5) {$n=3$};\node at (-32,-15.5) {$n=4$};
\end{tikzpicture}
\end{center}

\end{document}